%% file: AbateLille.tex
\documentclass[graybox, envcountsect, envcountsame]{svmult}

\usepackage{mathptmx}       
\usepackage{helvet}         
\usepackage{courier}        
\usepackage{amssymb}
\usepackage{makeidx}         
\usepackage{graphicx}        
\usepackage{multicol}        
\usepackage[bottom]{footmisc}
\smartqed
\newcommand{\Hol}{\mathord{\mathrm{Hol}}}
\newcommand{\Aut}{\mathord{\mathrm{Aut}}}
\newcommand{\Fix}{\mathord{\mathrm{Fix}}}
\newcommand{\R}{\mathbb{R}}
\newcommand{\C}{\mathbb{C}}
\newcommand{\n}{|\!\|}
\newcommand{\mlog}{{\textstyle\frac{1}{2}}\log}
\let\phe=\varphi
\let\eps=\varepsilon
\let\de=\partial

\makeindex             

\begin{document}

\title*{The Kobayashi distance in holomorphic dynamics and operator theory}
\author{Marco Abate}
\institute{Marco Abate \at Dipartimento di Matematica, Universit\`a di Pisa, Largo Pontecorvo 5, 56127 Pisa, Italy \email{marco.abate@unipi.it}}

%
%
\maketitle

\abstract*{The aim of this short course is to describe how to use a geometric structure (namely, a metric space structure) to explore and encode analytic properties of holomorphic functions and maps defined on complex manifolds. We shall first describe how to define the geometric structure, and then we shall present applications to holomorphic dynamics and to operator theory.}


\section{The Kobayashi distance}
\label{sec:1}
In this section we shall define the (invariant) distance we are going to use, and collect some of its main properties we shall need later on. It will not be a comprehensive treatise on the subject; much more informations can be found in, e.g., \cite{Abatebook, JarPflug, Kobayashibook}.

Before beginning, let us introduce a couple of notations we shall consistently use.

\begin{definition}
Let $X$ and $Y$ be two (finite dimensional) complex manifolds. We shall denote by $\Hol(X,Y)$ the set of all holomorphic maps from $X$ to $Y$, endowed with the compact-open topology (which coincides with the topology of uniform convergence on compact subsets), so that it becomes a metrizable topological space. Furthermore, we shall denote by $\Aut(X)\subset\Hol(X,X)$ the set of automorphisms, that is invertible holomorphic self-maps, of~$X$. More generally, if $X$ and $Y$ are topological spaces we shall denote by $C^0(X,Y)$ the space of continuous maps from $X$ to $Y$, again endowed with the compact-open topology.
\end{definition}

\begin{definition}
We shall denote by $\Delta=\{\zeta\in\mathbb{C}\mid |\zeta|<1\}$ the unit disk in the complex plane $\mathbb{C}$, by $B^n=\{z\in\mathbb{C}^n\mid\|z\|<1\}$ (where $\|\cdot\|$ is the Euclidean norm) the unit ball in the $n$-dimensional space $\mathbb{C}^n$, and by $\Delta^n\subset\mathbb{C}^n$ the unit polydisk in~$\mathbb{C}^n$.
Furthermore, $\langle\cdot\,,\cdot\rangle$ will denote the canonical Hermitian product on~$\C^n$.
\end{definition}

\subsection{The Poincar\'e distance}
\label{subsec:1.1}

The model for all invariant distances in complex analysis is the Poincar\'e distance on the unit disk of the complex plane; we shall then start recalling its definitions and main properties. 

\begin{definition}
The \emph{Poincar\'e} (or \emph{hyperbolic}) \emph{metric} on $\Delta$ is the Hermitian metric whose associated norm is given by
\[
\kappa_\Delta(\zeta;v)=\frac{1}{1-|\zeta|^2}|v|
\]
for all $\zeta\in\Delta$ and $v\in\mathbb{C}\simeq T_\zeta\Delta$. It is a complete Hermitian metric with constant Gaussian curvature~$-4$. 
\end{definition}

\begin{definition}
The \emph{Poincar\'e} (or \emph{hyperbolic}) \emph{distance}~$k_\Delta$ on~$\Delta$ is the integrated form of the Poincar\'e metric. It is a complete distance, whose expression is
\[
k_\Delta(\zeta_1,\zeta_2)=\mlog\frac{1+\left|\frac{\zeta_1-\zeta_2}{1-\overline{\zeta_1}\zeta_2}\right|}{1-\left|\frac{\zeta_1-\zeta_2}{1-\overline{\zeta_1}\zeta_2}\right|}\;.
\]
In particular,
\[
k_\Delta(0,\zeta)=\mlog\frac{1+|\zeta|}{1-|\zeta|}\;.
\]
\end{definition}

\begin{remark}
It is useful to keep in mind that the function 
\[
t\mapsto \mlog\frac{1+t}{1-t}
\]
is the inverse of the hyperbolic tangent $\tanh t=(\E^t-\E^{-t})/(\E^t+\E^{-t})$.
\end{remark}

Besides being a metric with constant negative Gaussian curvature, the Poincar\'e metric strongly reflects the properties of the holomorphic self-maps of the unit disk. For instance, the isometries of the Poincar\'e metric coincide with the holomorphic or anti-holomorphic automorphisms of~$\Delta$ (see, e.g., \cite[Proposition~1.1.8]{Abatebook}):

\begin{proposition}
\label{th:1.Pisom}
The group of smooth isometries of the Poincar\'e metric consists of all holomorphic and anti-holomorphic automorphisms of~$\Delta$.
\end{proposition}

More importantly, the famous \emph{Schwarz-Pick lemma} says that any holomorphic self-map of~$\Delta$ is nonexpansive for the Poincar\'e metric and distance (see, e.g., \cite[Theorem~1.1.6]{Abatebook}):

\begin{theorem}[Schwarz-Pick lemma]
\label{th:1.SPlemma}
Let $f\in\Hol(\Delta,\Delta)$ be a holomorphic self-map of $\Delta$. Then:
\begin{enumerate}
\item[\rm(i)] we have
\begin{equation}
\kappa_\Delta\bigl(f(\zeta);f'(\zeta)v\bigr)\le \kappa_\Delta(\zeta;v)
\label{eq:1.SPlemmau}
\end{equation}
for all $\zeta\in\Delta$ and $v\in\mathbb{C}$. Furthermore, equality holds for some $\zeta\in\Delta$ and $v\in\mathbb{C}^*$ if and only if equality holds for all $\zeta\in\Delta$ and all $v\in\mathbb{C}$ if and only if $f\in\Aut(\Delta)$;
\item[\rm(ii)] we have
\begin{equation}
k_\Delta\bigl(f(\zeta_1),f(\zeta_2)\bigr)\le k_\Delta(\zeta_1,\zeta_2)
\label{eq:1.SPlemmad}
\end{equation}
for all $\zeta_1$, $\zeta_2\in\Delta$. Furthermore, equality holds for some $\zeta_1\ne\zeta_2$ if and only if equality holds for all $\zeta_1$, $\zeta_2\in\Delta$ if and only if $f\in\Aut(\Delta)$. 
\end{enumerate}
\end{theorem}

In other words, \emph{holomorphic self-maps of the unit disk are automatically $1$-Lipschitz, and hence equicontinuous, with respect to the Poincar\'e distance.}

As an immediate corollary, we can compute the group of automorphisms of~$\Delta$, and thus, by 
Proposition~\ref{th:1.Pisom}, the group of isometries of the Poincar\'e metric (see, e.g., \cite[Proposition~1.1.2]{Abatebook}):

\begin{corollary}
\label{th:1.autDelta}
The group $\Aut(\Delta)$ of holomorphic automorphisms of $\Delta$ consists in all the functions $\gamma\colon\Delta\to\Delta$ of the form
\begin{equation}
\gamma(\zeta)=\E^{\I\theta}\frac{\zeta-\zeta_0}{1-\overline{\zeta_0}\zeta}
\label{eq:1.aut}
\end{equation}
with $\theta\in\mathbb{R}$ and $\zeta_0\in\Delta$. In particular, for every pair $\zeta_1$,~$\zeta_2\in\Delta$ there exists $\gamma\in\Aut(\Delta)$ such that $\gamma(\zeta_1)=0$ and $\gamma(\zeta_2)\in[0,1)$.
\end{corollary}

\begin{remark}
\label{rem:1.homDelta}
More generally, given $\zeta_1$, $\zeta_2\in\Delta$ and $\eta\in[0,1)$, it is not difficult to see that there is $\gamma\in\Aut(\Delta)$ such that $\gamma(\zeta_1)=\eta$ and $\gamma(\zeta_2)\in[0,1)$ with $\gamma(\zeta_2)\ge\eta$.
\end{remark}

A consequence of (\ref{eq:1.aut}) is that all automorphisms of $\Delta$ extends continuously to the boundary. It is customary to classify the elements of $\Aut(\Delta)$ according to the number of fixed points in~$\overline{\Delta}$: 

\begin{definition}
An automorphism $\gamma\in\Aut(\Delta)\setminus\{\mathrm{id}_\Delta\}$ is called \emph{elliptic} if it has a unique fixed point in~$\Delta$, \emph{parabolic} if it has a unique fixed point in~$\de\Delta$, \emph{hyperbolic} if it has exactly two fixed points in~$\de\Delta$. It is easy to check that these cases are mutually exclusive and exhaustive.
\end{definition}

We end this brief introduction to the Poincar\'e distance by recalling two facts relating its geometry to the Euclidean geometry of the plane (see, e.g., \cite[Lemma~1.1.5 and  (1.1.11)]{Abatebook}):

\begin{proposition}
Let $\zeta_0\in\Delta$ and $r>0$. Then the ball $B_\Delta(\zeta_0,r)\subset\Delta$ for the Poincar\'e distance of center $\zeta_0$ and radius $r$ is the Euclidean ball with center
\[
\frac{1-(\tanh r)^2}{1-(\tanh r)^2|\zeta_0|^2} \zeta_0
\]
and radius 
\[
\frac{(1-|\zeta_0|^2)\tanh r}{1-(\tanh r)^2|\zeta_0|^2}\;.
\]
\end{proposition}

\begin{proposition} 
\label{th:1.geod}
Let $\zeta_0=r\E^{\I\theta}\in\Delta$. Then the geodesic for the Poincar\'e metric connecting~$0$ to~$\zeta_0$ is the Euclidean radius $\sigma\colon[0,k_\Delta(0,\zeta_0)]\to\Delta$ given by
\[
\sigma(t)= (\tanh t) \E^{\I\theta}\;.
\]
In particular, $k_\Delta\bigl(0,(\tanh t)\E^{\I\theta}\bigr)=|t|$ for all $t\in\mathbb{R}$ and $\theta\in\mathbb{R}$.
\end{proposition}

\subsection{The Kobayashi distance in complex manifolds}
\label{subsec:1.2}

Our next aim is to build on any complex manifold a (pseudo)distance enjoying the main properties of the Poincar\'e distance; in particular, we would like to preserve the 1-Lipschitz property of holomorphic maps, that is to generalize to several variables Schwarz-Pick lemma. There are several ways for doing this; historically, the first such generalization has been introduced by Carath\'eodory \cite{Car} in 1926, but the most well-known and most useful has been proposed in 1967 by Kobayashi \cite{Kob1, Kob2}. Here we shall concentrate on the Kobayashi (pseudo)distance; but several other similar metrics and distances have been introduced (see, e.g., \cite{Berg, Hahn, Azu, Dema, Klimek, Sib1, Sib2, Wu2}; see also \cite{Harris} for a general context explaining why in a very precise sense the Carath\'eodory distance is the smallest and the Kobayashi distance is the largest possible invariant distance,
and \cite{Ahlfors} for a different differential geometric approach).  Furthermore, we shall discuss only the Kobayashi \emph{distance}; it is possible to define a Kobayashi metric, which is a complex Finsler metric whose integrated form is exactly the Kobayashi distance, but we shall not need it. It is also possible to introduce a Kobayashi pseudodistance in complex analytic spaces; again, see \cite{Abatebook}, \cite{JarPflug} and \cite{Kobayashibook} for details and much more.

To define the Kobayashi pseudodistance we first introduce an auxiliary function.

\begin{definition}
Let $X$ be a connected complex manifold. The \emph{Lempert function} $\delta_X\colon X\times X\to\mathbb{R}^+\cup\{+\infty\}$ is defined by
\[
\delta_X(z,w)=\inf\bigl\{k_\Delta(\zeta_0,\zeta_1)\bigm|\exists \varphi\in\Hol(\Delta,X): \varphi(\zeta_0)=z, \varphi(\zeta_1)=w\bigr\}
\]
for every $z$,~$w\in X$.
\end{definition}

\begin{remark}
\label{rem:2.Lempert}
Corollary~\ref{th:1.autDelta} yields the following equivalent definition of the Lempert function:
\[
\delta_X(z,w)=\inf\bigl\{k_\Delta(0,\zeta)\bigm|\exists \varphi\in\Hol(\Delta,X): \varphi(0)=z, \varphi(\zeta)=w\bigr\}\;.
\]
\end{remark}

The Lempert function in general (but there are exceptions; see Theorem~\ref{th:1.onedisk} below) does not satisfy the triangular inequality (see, e.g., \cite{Lemp} for an example), and so it is not a distance. But this is a problem easily solved:

\begin{definition}
Let $X$ be a connected complex manifold. The \emph{Kobayashi (pseudo) distance} $k_X\colon X\times X\to\mathbb{R}^+$ is the largest (pseudo)distance bounded above by the Lempert function, that is
\[
k_X(z,w)=\inf\biggl\{\sum_{j=1}^k \delta_X(z_{j-1},z_j)\biggm| k\in\mathbb{N}, z_0=z, z_k=w, z_1,\ldots,z_{k-1}\in X\biggr\}
\]
for all $z$, $w\in X$.
\end{definition}

A few remarks are in order. First of all, it is easy to check that since $X$ is connected then $k_X$ is always finite. Furthermore, it is clearly symmetric, it satisfies the triangle inequality by definition, and $k_X(z,z)=0$ for all $z\in X$. On the other hand, it might well happen that $k_X(z_0,z_1)=0$ for two distinct points $z_0\ne z_1$ of~$X$ (it might even happen that $k_X\equiv 0$; see Proposition~\ref{th:2.examples} below); so $k_X$ in general is only a pseudodistance. Anyway, the definition clearly implies the following generalization of the Schwarz-Pick lemma:

\begin{theorem}
\label{th:2.SPlemma}
Let $X$, $Y$ be two complex manifolds, and $f\in\Hol(X,Y)$. Then
\[
k_Y\bigl(f(z),f(w)\bigr)\le k_X(z,w)
\]
for all $z$, $w\in X$. In particular:
\begin{enumerate}
\item[\rm(i)] if $X$ is a submanifold of $Y$ then $k_Y|_{X\times X}\le k_X$;
\item[\rm(ii)] biholomorphisms are isometries with respect to the Kobayashi pseudodistances.
\end{enumerate}
\end{theorem}

A statement like this is the reason why the Kobayashi (pseudo)distance is said to be an \emph{invariant} distance: it is invariant under biholomorphisms.

Using the definition, it is easy to compute the Kobayashi pseudodistance of a few of interesting manifolds (see, e.g., \cite[Proposition~2.3.4, Corollaries~2.3.6, 2.3.7]{Abatebook}):

\begin{proposition}
\label{th:2.examples}
\begin{enumerate}
\item[\rm(i)] The Poincar\'e distance is the Kobayashi distance of the unit disk~$\Delta$.
\item[\rm(ii)] The Kobayashi distances of $\mathbb{C}^n$ and of the complex projective space $\mathbb{P}^n(\mathbb{C})$ vanish identically.
\item[\rm(iii)] For every $z=(z_1,\ldots,z_n)$, $w=(w_1,\ldots,w_n)\in\Delta^n$ we have
\[
k_{\Delta^n}(z,w)=\max_{j=1,\ldots,n}\{k_\Delta(z_j,w_j)\}\;.
\]
\item[\rm(iv)] The Kobayashi distance of the unit ball $B^n\subset\mathbb{C}^n$ coincides with the classical Bergman distance; in particular, if $O\in\mathbb{C}^n$ is the origin and $z\in B^n$ then
\[
k_{B^n}(O,z)=\mlog\frac{1+\|z\|}{1-\|z\|}\;.
\]
\end{enumerate}
\end{proposition}

\begin{remark}
As often happens with objects introduced via a general definition, the Kobayashi pseudodistance can seldom be explicitly computed. Besides the cases listed in Proposition~\ref{th:2.examples}, as far as we know there are formulas only for some complex ellipsoids \cite{JPZ}, bounded symmetric domains \cite{JarPflug}, the symmetrized bidisk \cite{AY} and a few other scattered examples. On the other hand, it is possible and important to estimate the Kobayashi distance; see Subsection~\ref{subsect:1.5} below.
\end{remark}

We shall be interested in manifolds where the Kobayashi pseudodistance is a true distance, that is in complex manifolds $X$ such that $k_X(z,w)>0$ as soon as $z\ne w$. 

\begin{definition}
A connected complex manifold $X$ is \emph{(Kobayashi) hyperbolic} if $k_X$ is a true distance. In this case, if $z_0\in X$ and $r>0$ we shall denote by $B_X(z_0,r)$ the ball for~$k_X$ of center~$z_0$ and radius~$r$; we shall call $B_X(z_0,r)$ a \emph{Kobayashi ball}. More generally, if $A\subseteq X$ and $r>0$ we shall put
$B_X(A,r)=\bigcup_{z\in A} B_X(z,r)$.
\end{definition}

In hyperbolic manifolds the Kobayashi distance induces the topology of the manifold. More precisely (see, e.g., \cite[Proposition~2.3.10]{Abatebook}):

\begin{proposition}[Barth, \cite{Bar1}]
\label{th:2.hyp}
A connected complex manifold $X$ is hyperbolic if and only if $k_X$ induces the manifold topology on~$X$. 
\end{proposition}

To give a first idea of how one can work with the Kobayashi distance, we 
describe two large classes of examples of hyperbolic manifolds:

\begin{proposition}[Kobayashi, \cite{Kob1, Kob2}]
\label{th:2.exhypman}
\begin{enumerate}
\item[\rm(i)] A submanifold of a hyperbolic manifold is hyperbolic. In particular, bounded domains in $\mathbb{C}^n$ are hyperbolic.
\item[\rm(ii)] Let $\pi\colon\tilde X\to X$ be a holomorphic covering map. Then $X$ is hyperbolic if and only if $\tilde X$ is. In particular, a Riemann surface is hyperbolic if and only if it is Kobayashi hyperbolic.
\end{enumerate}
\end{proposition}

\begin{proof}
(i) The first assertion follows immediately from Theorem~\ref{th:2.SPlemma}.(i). For the second one,
we remark that the unit ball $B^n$ is hyperbolic by Proposition~\ref{th:2.examples}.(iv). Then Theorem~\ref{th:2.SPlemma}.(ii) implies that all balls are hyperbolic; since a bounded domain is contained in a ball, the assertion follows.

(ii) First of all we claims that 
\begin{equation}
k_X(z_0,w_0)=\inf\bigl\{k_{\tilde X}(\tilde z_0,\tilde w)\bigm | w\in\pi^{-1}(w_0)\bigr\}\;,
\label{eq:2.cover}
\end{equation}
for any $z_0$,~$w_0\in X$,
where $\tilde z_0$ is any element of~$\pi^{-1}(z_0)$. Indeed, first of all 
Theorem~\ref{th:2.SPlemma} immediately implies that
\[
k_X(z_0,w_0)\le \inf\bigl\{k_{\tilde X}(\tilde z_0,\tilde w)\bigm | w\in\pi^{-1}(w_0)\bigr\}\;.
\]
Assume now, by contradiction, that there is $\eps>0$ such that
\[
k_X(z_0,w_0)+\eps\le k_{\widetilde X}\bigl(\tilde z_0,\tilde w\bigr)
\]
for all $\tilde w\in\pi^{-1}(w_0)$. Choose $z_1,\ldots,z_k=w_0\in X$ with $z_k=w$ such that
\[
\sum_{j=1}^k \delta_X(z_{j-1},z_j)<k_X(z_0,w_0)+\eps/2\;.
\]
By Remark~\ref{rem:2.Lempert}, we can find $\phe_1,\ldots,\phe_k\in\Hol(\Delta,X)$ 
and $\zeta_1,\ldots,\zeta_k\in\Delta$ such that $\phe_j(0)=z_{j-1}$, $\phe_j(\zeta_j)=z_j$ for all $j=1,
\ldots,k$ and
\[
\sum_{j=1}^k k_\Delta(0,\zeta_j)<k_X(z_0,w_0)+\eps\;.
\]
Let $\tilde\phe_1,\ldots,\tilde\phe_k\in\Hol(\Delta,\tilde X)$ be the liftings of~$\phe_1,\ldots,\phe_k$
chosen so that $\tilde\phe_1(0)=\tilde z_0$ and $\tilde\phe_{j+1}(0)=\tilde\phe_j(\zeta_j)$ for $j=1,\ldots,k-1$, and set $\tilde w_0=\tilde\phe_k(\zeta_k)\in\pi^{-1}(w_0)$. Then
\[
k_{\tilde X}(\tilde z_0,\tilde w_0)\le\sum_{j=1}^k \delta_{\tilde X}\bigl(\tilde\phe_j(0),
\tilde\phe_j(\zeta_j)\bigr)\le \sum_{j=1}^k k_\Delta(0,\zeta_j)<k_X(z_0,w_0)+\eps\le
k_{\tilde X}(\tilde z_0,\tilde w_0)\;,
\]
contradiction.

Having proved (\ref{eq:2.cover}), let us assume 
that $\tilde X$ is hyperbolic. If there are $z_0$, $w_0\in X$ such that $k_X(z_0,w_0)=0$, then for any~$\tilde z_0\in\pi^{-1}(z_0)$ there is a 
sequence $\{\tilde w_\nu\}\subset\pi^{-1}(w_0)$ such that $k_{\tilde X}(
\tilde z_0,\tilde w_\nu)\to0$ as~$\nu\to+\infty$. 
Then $\tilde w_\nu\to\tilde z_0$ (Proposition~\ref{th:2.hyp}) and so
$\tilde z_0\in\pi^{-1}(w_0)$, that is~$z_0=w_0$.

Conversely, assume $X$ hyperbolic. Suppose $\tilde z_0$,~$\tilde 
w_0\in\tilde X$ are so that $k_{\tilde X}(\tilde z_0,\tilde w_0)=0$; 
then $k_X\bigl(\pi(\tilde z_0),\pi(\tilde w_0)\bigr)=0$ and so
$\pi(\tilde z_0)=\pi(\tilde w_0)=z_0$. Let~$\widetilde U$ be a connected
neighborhood of~$\tilde z_0$ such that $\pi|_{\widetilde U}$~is a 
biholomorphism between~$\widetilde U$ and the (connected component containing~$z_0$ of the) for the Kobayashi ball $B_X(z_0,\eps)$ 
of center~$z_0$ and radius $\eps>0$~small 
enough; this can be done because of Proposition~\ref{th:2.hyp}. Since
$k_{\widetilde X}(\tilde z_0,\tilde w_0)=0$, we can find $\phe_1,\ldots,\phe_k\in\Hol(\Delta,\tilde X)$ 
and $\zeta_1,\ldots,\zeta_k\in\Delta$ with $\phe_1(0)=\tilde z_0$, $\phe_j(\zeta_j)=\phe_{j+1}(0)$ for $j=1,\ldots,k-1$ and $\phe_k(\zeta_k)=\tilde w_0$ such that 
\[
\sum_{j=1}^k k_\Delta(0,\zeta_j)<\eps\;.
\]
Let $\sigma_j$ be the radial segment in~$\Delta$
joining~$0$ to~$\zeta_j$; by Proposition~\ref{th:1.geod} the $\sigma_j$ are geodesics for the Poincar\'e metric. The arcs~$\phe_j\circ\sigma_j$ 
in~$\widetilde X$ connect to form a continuous curve~$\sigma$ from~$\tilde z_0$ 
to~$\tilde w_0$. Now the maps $\pi\circ\phe_j\in\Hol(\Delta,X)$ 
are non-expanding; therefore every point of the curve~$\pi\circ\sigma$ 
should belong to~$B_X(z_0,\eps)$. But then~$\sigma$ is contained in~$\widetilde 
U$, and this implies $\tilde z_0=\tilde w_0$.

The final assertion on Riemann surfaces follows immediately because hyperbolic Riemann surfaces can be characterized as the only Riemann surfaces whose universal covering is the unit disk.
\qed 
\end{proof}

It is also possible to prove the following (see, e.g., \cite[Proposition~2.3.13]{Abatebook}):

\begin{proposition}
\label{th:2.prodhyp}
Let $X_1$ and $X_2$ be connected complex manifolds. Then $X_1\times X_2$ is hyperbolic if and only if both $X_1$ and $X_2$ are hyperbolic.
\end{proposition}

\begin{remark}
The Kobayashi pseudodistance can be useful even when it is degenerate. For instance, the classical Liouville theorem (a bounded entire function is constant) is an immediate consequence, thanks to Theorem~\ref{th:2.SPlemma}, of the vanishing of the Kobayashi pseudodistance of $\mathbb{C}^n$ and the fact that bounded domains are hyperbolic.
\end{remark}

A technical fact we shall need later on is the following:

\begin{lemma} 
\label{th:1.somma}
Let $X$ be a hyperbolic manifold, and choose $z_0\in X$ 
and $r_1$,~$r_2>0$. Then
\[
B_X\bigl(B_X(z_0,r_1),r_2\bigr)=B_X(z_0,r_1+r_2)\;.
\]
\end{lemma}

\begin{proof} 
The inclusion $B_D\bigl(B_D(z_0,r_1),r_2\bigr)\subseteq B_D(z_0,r_1+r_2)$ 
follows immediately from the triangular inequality. For the converse, let 
$z\in B_D(z_0,r_1+r_2)$, and set $3\eps=r_1+r_2-k_X(z_0,z)$. Then there are
$\phe_1,\ldots,\phe_m\in\Hol(\Delta,X)$ and $\zeta_1,\ldots,\zeta_m\in\Delta$
so that $\phe_1(0)=z_0$, $\phe_j(\zeta_j)=\phe_{j+1}(0)$ for $j=1,\ldots,m-1$,
$\phe_m(\zeta_m)=z$ and
\[
\sum_{j=1}^m k_\Delta(0,\zeta_{j})<r_1+r_2-2\eps\;.
\]
Let $\mu\le m$ be the largest integer such that
\[
\sum_{j=1}^{\mu-1}k_\Delta(0,\zeta_{j})<r_1-\eps\;.
\]
Let $\eta_\mu$ be the point on the Euclidean radius in~$\Delta$ passing through $\zeta_{\mu+1}$ (which is a geodesic for the Poincar\'e distance) such that
\[
\sum_{j=1}^{\mu-1}k_\Delta(0,\zeta_j)+k_\Delta(0,\eta_\mu)=r_1
	-\eps\;.
\]
If we set $w=\phe_\mu(\eta_\mu)$, then $k_X(z_0,w)<r_1$ and $k_X(w,z)<r_2$, 
so that 
\[
z\in B_D(w,r_2)\subseteq B_D\bigl(B_D(z_0,r_1),r_2\bigr)\;,
\]
and we are done.
\qed
\end{proof}

A condition slightly stronger than hyperbolicity is the following:

\begin{definition}
A hyperbolic complex manifold $X$ is \emph{complete hyperbolic} if the Kobayashi distance $k_X$ is complete.
\end{definition}

Complete hyperbolic manifolds have a topological characterization (see, e.g., \cite[Proposition~2.3.17]{Abatebook}):

\begin{proposition}
\label{th:2.chyp}
Let $X$ be a hyperbolic manifold. Then $X$ is complete hyperbolic if and only if every closed Kobayashi ball is compact. In particular, compact hyperbolic manifolds are automatically complete hyperbolic.
\end{proposition}

Examples of complete hyperbolic manifolds are contained in the following (see, e.g., \cite[Propositions~2.3.19 and 2.3.20]{Abatebook}):

\begin{proposition}
\label{th:2.exchyp}
\begin{enumerate}
\item[\rm(i)] A homogeneous hyperbolic manifold is complete hyperbolic. In particular, both $B^n$ and $\Delta^n$ are complete hyperbolic.
\item[\rm(ii)] A closed submanifold of a complete hyperbolic manifold is complete hyperbolic.
\item[\rm(iii)] The product of two hyperbolic manifolds is complete hyperbolic if and only if both factors are complete hyperbolic.
\item[\rm(iv)] If $\pi\colon\tilde X\to X$ is a holomorphic covering map, then $\tilde X$ is complete hyperbolic if and only if $X$ is complete hyperbolic.
\end{enumerate}
\end{proposition}

We shall see more examples of complete hyperbolic manifolds later on (Proposition~\ref{th:1.Barth} and Corollary~\ref{th:1.pscomp}). We end this subsection recalling the following important fact (see, e.g., \cite[Theorem~5.4.2]{Kobayashibook}):

\begin{theorem}
\label{th:1.authyp}
The automorphism group $\Aut(X)$ of a hyperbolic manifold $X$ has a natural structure of real Lie group. 
\end{theorem}

\subsection{Taut manifolds}
\label{subsect:1.3}

For our dynamical applications we shall need a class of manifolds which is intermediate between complete hyperbolic and hyperbolic manifolds. To introduce it, we first show that hyperbolicity can be characterized as a precompactness assumption on the space $\Hol(\Delta,X)$. 

If $X$ is a topological space, we shall denote by $X^*=X\cup\{\infty\}$ its one-point (or Alexandroff) compactification; see, e.g., \cite[p. 150]{Kelley} for details.

\begin{theorem}[\cite{Ahyp}]
Let $X$ be a connected complex manifold. Then $X$ is hyperbolic if and only if $\Hol(\Delta,X)$ is relatively compact in the space $C^0(\Delta,X^*)$ of continuous functions from $\Delta$ into the one-point compactification of~$X$. In particular, if $X$ is compact then it is hyperbolic if and only if $\Hol(\Delta,X)$ is compact. Finally, if $X$ is hyperbolic then $\Hol(Y,X)$ is relatively compact in $C^0(Y,X^*)$ for any complex manifold~$Y$.
\end{theorem}

If $X$ is hyperbolic and not compact, the closure of $\Hol(\Delta,X)$ in $C^0(\Delta,X^*)$ might contain continuous maps whose image might both contain $\infty$ and intersect~$X$, exiting thus from the realm of holomorphic maps. Taut manifolds, introduced by Wu \cite{Wutaut}, are a class of (not necessarily compact) hyperbolic manifolds where this problem does not appear,
and (as we shall see) this will be very useful when studying the dynamics of holomorphic self-maps. 

\begin{definition}
A complex manifold $X$ is \emph{taut} if it is hyperbolic and every map in the closure of $\Hol(\Delta,X)$ in $C^0(\Delta,X^*)$ either is in $\Hol(\Delta,X)$ or is the constant map~$\infty$.
\end{definition}

This definition can be rephrased in another way not requiring the one-point compactification.

\begin{definition}
Let $X$ and $Y$ be topological spaces. A sequence $\{f_\nu\}\subset C^0(Y,X)$ is \emph{compactly divergent} if for every pair of compacts $H\subseteq Y$ and $K\subseteq X$ there exists $\nu_0\in\mathbb{N}$ such that $f_\nu(H)\cap K=\emptyset$ for every $\nu\ge\nu_0$. A family $\mathcal{F}\subseteq C^0(Y,X)$ is \emph{normal} if every sequence in $\mathcal{F}$ admits a subsequence which is either uniformly converging on compact subsets or compactly divergent. 
\end{definition}

By the definition of one-point compactification, a sequence in $C^0(Y,X)$ converges in $C^0(Y,X^*)$ to the constant map~$\infty$ if and only if it is compactly divergent. When $X$ and $Y$ are manifolds (more precisely, when they are Hausdorff, locally compact, connected and second countable topological spaces), a subset in $C^0(Y,X^*)$ is compact if and only if it is sequentially compact; therefore we have obtained the following alternative characterization of taut manifolds:

\begin{corollary}
A connected complex manifold $X$ is taut if and only if the family $\Hol(\Delta,X)$ is normal.
\end{corollary}

Actually, it is not difficult to prove (see, e.g., \cite[Theorem~2.1.2]{Abatebook}) that the role of $\Delta$ in the definition of taut manifolds is not essential:

\begin{proposition}
Let $X$ be a taut manifold. Then $\Hol(Y,X)$ is a normal family for every complex manifold~$Y$. 
\end{proposition}

It is easy to find examples of hyperbolic manifolds which are not taut:

\begin{example}
Let $D=\Delta^2\setminus\{(0,0)\}$. Since $D$ is a bounded domain in~$\mathbb{C}^2$, it is hyperbolic. For $\nu\ge 1$ let $\phe_\nu\in\Hol(\Delta,D)$ given by $\phe_\nu(\zeta)=(\zeta,1/\nu)$.
Clearly $\{\phe_\nu\}$ converges as $\nu\to+\infty$ to the map $\phe(\zeta)=(\zeta,0)$, whose 
image is not contained either in~$D$ or in~$\partial D$. In particular, the sequence $\{\phe_\nu\}$ does not admit a subsequence which is compactly divergent or converging to a map with image in~$D$---and thus $D$ is not taut.
\end{example}

On the other hand, complete hyperbolic manifolds are taut. This is a consequence of the famous
Ascoli-Arzel\`a theorem (see, e.g., \cite[p. 233]{Kelley}):

\begin{theorem}[Ascoli-Arzel\`a theorem]
\label{th:2.AA}
Let $X$ be a metric space, and $Y$ a locally compact metric space. Then a family $\mathcal{F}\subseteq C^0(Y,X)$ is relatively compact in $C^0(Y,X)$ if and only if the following two conditions are satisfied:
\begin{enumerate}
\item[\rm(i)] $\mathcal{F}$ is equicontinuous;
\item[\rm(ii)] the set $\mathcal{F}(y)=\{f(y)\mid f\in\mathcal{F}\}$ is relatively compact in~$X$ for every $y\in Y$.
\end{enumerate}
\end{theorem}

Then: 

\begin{proposition}
\label{th:2.chyptaut}
Every complete hyperbolic manifold is taut.
\end{proposition}

\begin{proof}
Let $X$ be a complete hyperbolic manifold, and $\{\phe_\nu\}\subset\Hol(\Delta,X)$ a sequence which is not compactly divergent; we must prove that it admits a subsequence converging in $\Hol(\Delta,X)$. 

Up to passing to a subsequence, we can find a pair of compacts $H\subset\Delta$ and $K\subseteq X$ such that $\phe_\nu(H)\cap K\ne\emptyset$ for all $\nu\in\mathbb{N}$. Fix $\zeta_0\in H$ 
and $z_0\in K$, and set 
$r=\max\{k_X(z,z_0)\mid z\in K\}$. Then for every $\zeta\in\Delta$ and $\nu\in\mathbb{N}$ we have
\[
k_X\bigl(\phe_\nu(\zeta),z_0\bigr)\le k_X\bigl(\phe_\nu(\zeta),\phe_\nu(\zeta_0)\bigr)+
k_X\bigl(\phe_\nu(\zeta_0),z_0\bigr)\le k_\Delta(\zeta,\zeta_0)+r\;.
\]
So $\{\phe_\nu(\zeta)\}$ is contained in the closed Kobayashi ball of center~$z_0$ and radius $k_\Delta(\zeta,\zeta_0)+r$, which is compact since $X$ is complete hyperbolic (Proposition~\ref{th:2.chyp}); as a consequence, $\{\phe_\nu(\zeta)\}$  is relatively compact in~$X$. Furthernore,
since $X$ is hyperbolic, the whole family $\Hol(\Delta,X)$ is equicontinuous (it is 1-Lipschitz with respect to the Kobayashi distances); therefore, by the Ascoli-Arzel\`a theorem, the sequence $\{\phe_\nu\}$ is relatively compact in~$C^0(\Delta,X)$. In particular, it admits a subsequence converging in $C^0(\Delta,X)$; but since, by Weierstrass theorem, $\Hol(\Delta,X)$ is closed in $C^0(\Delta,X)$, the limit belongs to $\Hol(\Delta,X)$, and we are done.
\qed
\end{proof}

Thus complete hyperbolic manifolds provide examples of taut manifolds. However, 
there are taut manifolds which are not complete hyperbolic; an example has been given by Rosay (see \cite{Rosay}). Finally, we have the following equivalent of Proposition~\ref{th:2.exchyp}
(see, e.g., \cite[Lemma~2.1.15]{Abatebook}):

\begin{proposition}
\label{th:2.extaut}
\begin{enumerate}
\item[\rm(i)] A closed submanifold of a taut manifold is taut.
\item[\rm(ii)] The product of two complex manifolds is taut if and only if both factors are taut.
\end{enumerate}
\end{proposition}

Just to give an idea of the usefulness of the taut condition in studying holomorphic self-maps we end this subsection by quoting Wu's generalization of the classical Cartan-Carath\'eodory and Cartan uniqueness theorems (see, e.g., \cite[Theorem~2.1.21 and Corollary 2.1.22]{Abatebook}):

\begin{theorem}[Wu, \cite{Wutaut}]
\label{th:1.CC}
Let $X$ be a taut manifold, and let $f\in\Hol(X,X)$ with a fixed point $z_0\in X$. Then:
\begin{enumerate}
\item[\rm(i)] the spectrum of $\D f_{z_0}$ is contained in~$\overline{\Delta}$;
\item[\rm(ii)] $|\det \D f_{z_0}|\le 1$;
\item[\rm(iii)] $|\det \D f_{z_0}|=1$ if and only if $f\in\Aut(X)$;
\item[\rm(iv)] $\D f_{z_0}=\mathrm{id}$ if and only if $f$ is the identity map;
\item[\rm(v)] $T_{z_0}X$ admits a $\D f_{z_0}$-invariant splitting $T_{z_0}X=L_N\oplus L_U$ such that the spectrum of $\D f_{z_0}|_{L_N}$ is contained in~$\Delta$, the spectrum of $\D f_{z_0}|_{L_U}$ is contained in $\partial\Delta$, and $\D f_{z_0}|_{L_U}$ is diagonalizable.
\end{enumerate}
\end{theorem}

\begin{corollary}[Wu, \cite{Wutaut}]
\label{th:1.Cu}
Let $X$ be a taut manifold, and $z_0\in X$. Then if $f$, $g\in\Aut (X)$ are such that $f(z_0)=g(z_0)$ and $\D f_{z_0}=\D g_{z_0}$ then $f\equiv g$.
\end{corollary}

\begin{proof}
Apply Theorem~\ref{th:1.CC}.(iv) to $g^{-1}\circ f$.
\qed
\end{proof}

\subsection{Convex domains}
\label{subsect:1.4}

In the next two sections we shall be particularly interested in two classes of bounded domains in $\mathbb{C}^n$: convex domains and strongly pseudoconvex domains. Consequently, in this and the next subsection we shall collect some of the main properties of the Kobayashi distance respectively in convex and strongly pseudoconvex domains.

We start with convex domains recalling a few definitions. 

\begin{definition}
\label{def:1.segment}
Given $x$,~$y\in\mathbb{C}^n$ let
\[
[x,y]=\{sx+(1-s)y\in\mathbb{C}^n\mid s\in[0,1]\}\ \hbox{and}\ 
(x,y)=\{sx+(1-s)y\in\mathbb{C}^n\mid s\in(0,1)\}
\]
denote the \emph{closed,} respectively \emph{open, segment} connecting $x$ and $y$. 
A set $D\subseteq\mathbb{C}^n$ is \emph{convex} if $[x,y]\subseteq D$ for all $x$,~$y\in D$;
and \emph{strictly convex} if $(x,y)\subseteq D$ for all $x$,~$y\in\overline{D}$. A convex domain not strictly convex will sometimes be called \emph{weakly convex.}
\end{definition}

An easy but useful observation (whose proof is left to the reader) is:

\begin{lemma} 
\label{th:1.couno}
Let $D\subset\C^n$ be a convex domain. Then:
\begin{description}[(ii)]
\item[\rm (i)] $(z,w)\subset D$ for all $z\in D$ and $w\in\de D$;
\item[\rm (ii)] if $x$, $y\in\de D$ then either $(x,y)\subset\de D$ or $(x,y)\subset D$.
\end{description}
\end{lemma}

This suggests the following

\begin{definition}
\label{def:1.2}
Let $D\subset\C^n$ be a convex domain. Given $x\in\de D$, we put
\[
\hbox{\rm ch}(x)=\{y\in\de D\mid [x,y]\subset\de D\}\;;
\]
we shall say that $x$ is a \emph{strictly convex point} if $\hbox{\rm ch}(x)=\{x\}$. More generally,
given $F\subseteq\de D$ we put
\[
\hbox{\rm ch}(F)=\bigcup_{x\in F}\hbox{\rm ch}(x)\;.
\]
\end{definition}

A similar construction having a more holomorphic character is the following:

\begin{definition}
\label{def:1.3}
Let $D\subset\C^n$ be a convex domain. A \emph{complex supporting functional} at~$x\in\de D$ is a
$\C$-linear map $L\colon\C^n\to\C$ such that $\mathrm{Re}\,L(z)<\mathrm{Re}\,L(x)$ for all $z\in D$. A \emph{complex supporting hyperplane} at $x\in\de D$ is 
an affine complex hyperplane $H\subset\C^n$ of the form
$H=x+\ker L$, where $L$ is a complex supporting functional at~$x$ (the existence of complex supporting functionals and hyperplanes is guaranteed by the Hahn-Banach theorem).
Given $x\in\de D$, we shall denote by
$\hbox{\rm Ch}(x)$ the intersection of~$\overline{D}$ with of all complex supporting hyperplanes at~$x$. Clearly, $\hbox{\rm Ch}(x)$ is a closed convex set containing~$x$; in particular, $\hbox{\rm Ch}(x)\subseteq\hbox{\rm ch}(x)$. If $\hbox{\rm Ch}(x)=\{x\}$ we say that $x$ is a \emph{strictly $\C$-linearly convex point;}
and we say that $D$ is \emph{strictly $\C$-linearly convex} if all points of~$\de D$ are
strictly $\C$-linearly convex. Finally, if $F\subset\de D$ we set
\[
\hbox{\rm Ch}(F)=\bigcup_{x\in F}\hbox{\rm Ch}(x)\;;
\]
clearly, $\hbox{\rm Ch}(F)\subseteq\hbox{\rm ch}(F)$.
\end{definition}

\begin{definition}
Let $D\subset\C^n$ be a convex domain, $x\in\de D$ and $L\colon\C^n\to\C$ a complex supporting functional at~$x$. The \emph{weak peak function} associated to~$L$ is the function $\psi\in\Hol(D,\Delta)$ given by
\[
\psi(z)=\frac{1}{1-\bigl(L(z)-L(x)\bigr)}\;.
\]
Then $\psi$ extends continuously to~$\overline{D}$ with $\psi(\overline{D})\subseteq\overline{\Delta}$, $\psi(x)=1$, and $|\psi(z)|<1$ for all $z\in D$; moreover $y\in\de D$ is 
such that $|\psi(y)|=1$ if and only if $\psi(y)=\psi(x)=1$, and hence if and only if $L(y)=L(x)$.
\end{definition}

\begin{remark}
\label{rem:1.suppf}
If $x\in\de D$ is a strictly convex point of a convex domain $D\subset\C^n$ then it is possible to find a complex supporting functional $L$ at~$x$ so that $\mathrm{Re}\,L(z)<\mathrm{Re}\,L(x)$ for all $z\in\overline{D}\setminus\{x\}$. In particular, the associated weak peak function $\psi\colon\C^n\to\C$ is a true peak function (see Definition~\ref{def:1.peak} below) in the sense that $|\psi(z)|<1$ for all $z\in\overline{D}\setminus\{x\}$.
\end{remark}

We shall now present three propositions showing how the Kobayashi distance is particularly well-behaved in convex domains. The first result, due to Lempert, shows that in convex domains the definition of Kobayashi distance can be simplified:

\begin{proposition}[Lempert, \cite{Lemp}]
\label{th:1.onedisk} 
Let $D\subset\C^n$ be a convex domain. Then $\delta_D=k_D$.
\end{proposition}

\begin{proof}
First of all, note that $\delta_D(z,w)<+\infty$ for all $z$,~$w\in D$. 
Indeed, let 
\[
\Omega=\{\lambda\in\C\mid(1-\lambda)z+\lambda w\in D\}\;.
\]
Since $D$ is convex, $\Omega$ is a convex domain in $\C$ containing~0 and~1. 
Let $\phi\colon\Delta\to\Omega$ be a biholomorphism such that $\phi(0)=0$; then 
the map $\phe\colon\Delta\to D$ given by 
\[
\phe(\zeta)=\bigl(1-\phi(\zeta)\bigr)z+\phi(\zeta)w
\]
is such that $z$,~$w\in\phe(\Delta)$.

Now, by definition we have $\delta_D(z,w)\ge k_D(z,w)$; to get the reverse inequality it suffices to show that $\delta_D$ satisfies 
the triangular inequality. Take $z_1$,~$z_2$,~$z_3\in D$ 
and fix $\eps>0$. Then there are $\phe_1$,~$\phe_2\in\Hol(\Delta,D)$ and
$\zeta_1$,~$\zeta_2\in\Delta$ such that $\phe_1(0)=z_1$, $\phe_1(\zeta_1)=
\phe_2(\zeta_1)=z_2$, $\phe_2(\zeta_2)=z_3$ and
\begin{eqnarray*}
k_\Delta(0,\zeta_1)&<\delta_D(z_1,z_2)+\eps\;,\\
	k_\Delta(\zeta_1,\zeta_2)&<\delta_D(z_2,z_3)+\eps\;.
\end{eqnarray*}
Moreover, by Remark~\ref{rem:1.homDelta} we can assume that $\zeta_1$ and $\zeta_2$ are real, and that
$\zeta_2>\zeta_1>0$. Furthermore, up to replacing $\phe_j$ by a map $\phe_j^r$ 
defined by $\phe^r_j(\zeta
)=\phe_j(r\zeta)$ for $r$ close enough to 1, we can also assume that $\phe_j$ 
is defined and continuous on $\overline{\Delta}$ (and this for $j=1$,~$2$).

Let $\lambda\colon\C\setminus\{\zeta_1,\zeta_1^{-1}\}\to\C$ be given by
\[
\lambda(\zeta)=\frac{(\zeta-\zeta_2)(\zeta-\zeta_2^{-1})}{(\zeta-\zeta_1)
	(\zeta-\zeta_1^{-1})}\;.
\]
Then $\lambda$ is meromorphic in $\C$, and in a neighborhood of $\overline\Delta$ the 
only pole is the simple pole at~$\zeta_1$. Moreover, $\lambda(0)=1$,
$\lambda(\zeta_2)=0$ and $\lambda(\partial\Delta)\subset[0,1]$. Then define 
$\phi\colon\overline{\Delta}\to\C^n$ by
\[
\phi(\zeta)=\lambda(\zeta)\phe_1(\zeta)+\bigl(1-\lambda(\zeta)\bigr)
	\phe_2(\zeta)\;.
\]
Since $\phe_1(\zeta_1)=\phe_2(\zeta_1)$, it turns out that $\phi$ is holomorphic on $\Delta$; moreover, $\phi(0)=z_1$, $\phi(\zeta_2)=z_3$ and $\phi(\partial\Delta)\subset\overline{D}$.
We claim that this implies that~$\phi(\Delta)\subset D$. Indeed, otherwise there would be $\zeta_0\in
\Delta$ such that $\phi(\zeta_0)=x_0\in\partial D$. Let $L$ be a complex supporting functional at~$x_0$, and $\psi$ the associated weak peak function.
Then we would have $|\psi\circ\phi|\le 1$ on~$\partial\Delta$ 
and $|\psi\circ\phi(\zeta_0)|=1$; thus, by the maximum principle,
$|\psi\circ\phi|\equiv1$, i.e., $\phi(\Delta)\subset\partial D$, whereas~$\phi(0)
\in D$, contradiction.

So $\phi\in\Hol(\Delta,D)$. In particular, then,
\[
\delta_D(z_1,z_3)\le k_\Delta(0,\zeta_2)=k_\Delta(0,\zeta_1)+k_\Delta(\zeta_1,
	\zeta_2)\le\delta_D(z_1,z_2)+\delta_D(z_2,z_3)+2\eps\;,
\]
and the assertion follows, since $\eps$ is arbitrary. 
\qed
\end{proof}

Bounded convex domains, being bounded, are hyperbolic. But actually more is true:

\begin{proposition}[Harris, \cite{Harris}]
\label{th:1.Barth} 
Let $D\subset\subset\C^n$ be a bounded convex 
domain. Then $D$ is complete hyperbolic.
\end{proposition}

\begin{proof} 
We can assume $O\in D$. By Proposition~\ref{th:2.chyp}, it suffices to show that all the closed 
Kobayashi balls $\overline{B_D(O,r)}$ of center $O$ are compact. Let $\{z_\nu\}\subset
\overline{B_D(O,r)}$; we must find a subsequence converging to a point of $D$. 
Clearly, we may suppose that $z_\nu\to w_0\in\overline D$ as~$\nu\to+\infty$, for 
$D$~is bounded. 

Assume, by contradiction, that~$w_0\in\partial D$, and let $L\colon\C^n\to\C$ be a complex supporting functional at~$w_0$; 
in particular, 
$L(w_0)\ne0$ (because~$O\in D$). 
Set $H=\{\zeta\in\C\mid\mathrm{Re}\, L(\zeta w_0)<\mathrm{Re}\,L(w_0)\}$; 
clearly $H$~is a half-plane of~$\C$, and the linear map $\pi\colon\C^n\to\C$ 
given by $\pi(z)=L(z)/L(w_0)$ sends~$D$ into~$H$. In particular
\[
r\ge k_D(0,z_\nu)\ge k_H\bigl(0,\pi(z_\nu)\bigr)\;.
\]
Since $H$~is complete hyperbolic, by Proposition~\ref{th:2.chyp} the closed 
Kobayashi balls in~$H$ are compact; therefore, up to a subsequence
$\{\pi(z_\nu)\}$ tends to a point of~$H$. On the other hand, $\pi(z_\nu)\to
\pi(w_0)=1\in\partial H$, and this is a contradiction. 
\qed
\end{proof}

\begin{remark}
\label{rem:1:hypcvx}
There are unbounded convex domains which are not hyperbolic; for instance, $\C^n$ itself. However, unbounded hyperbolic convex domains are automatically complete hyperbolic, because Harris (see \cite{Harris}) proved that a convex domain is hyperbolic if and only if it is biholomorphic to a bounded convex domain. Furthermore, Barth (see \cite{Bar2}) has shown that an unbounded convex domain is hyperbolic if and only if it contains no complex lines.
\end{remark}

Finally, the convexity is reflected by the shape of Kobayashi balls. To prove this (and also because they will be useful later) we shall need a couple of estimates:

\begin{proposition}[\cite{Lemp}, \cite{KS}, \cite{KKR}]
\label{th:1.convball} 
Let $D\subset\C^n$ be a 
convex domain. Then:
\begin{enumerate}
\item[\rm(i)] if $z_1$, $z_2$, $w_1$, $w_2\in D$ and $s\in[0,1]$ then
\[
k_D\bigl(sz_1+(1-s)z_2, sw_1+(1-s)w_2\bigr)\le\max\{k_D(z_1,w_1),k_D(z_2,w_2)\}\;;
\]
\item[\rm(ii)] if $z$, $w\in D$ and $s$, $t\in[0,1]$ then
\[
k_D\bigl(sz+(1-s)w,tz+(1-t)w\bigr)\le k_D(z,w)\;.
\]
\end{enumerate}
\end{proposition}
 
\begin{proof} 
Let us start by proving (i). Without loss of generality we can assume that $k_D(z_2,w_2)\le 
k_D(z_1,w_1)$. Fix~$\eps>0$; by Proposition~\ref{th:1.onedisk}, there 
are~$\phe_1$,~$\phe_2\in\Hol(\Delta, D)$ and~$\zeta_1$,~$\zeta_2\in\Delta$ such 
that $\phe_j(0)=z_j$, $\phe_j(\zeta_j)=w_j$ and $k_\Delta(0,\zeta_j)<k_D(z_j,w_j)
+\eps$, for~$j=1$,~2; moreover, we may assume $0\le\zeta_2\le\zeta_1<1$ 
and $\zeta_1>0$. Define $\psi\colon\Delta\to D$ by
\[
\psi(\zeta)=\phe_2\biggl(\frac{\zeta_2}{\zeta_1}\,\zeta\!\biggr)\;,
\]
so that $\psi(0)=z_2$ and $\psi(\zeta_1)=w_2$, 
and~$\phi_s\colon\Delta\to\C^n$ by
\[
\phi_s(\zeta)=s \phe_1(\zeta)+(1-s)\psi(\zeta)\;.
\]
Since $D$~is convex, $\phi_s$ maps~$\Delta$
into~$D$; furthermore, $\phi_s(0)=sz_1+(1-s)z_2$ and $\phi_s(\zeta_1)=sw_1+(1-s)w_2$. Hence
\begin{eqnarray*}
k_D\bigl(sz_1+(1-s)z_2,sw_1+(1-s)w_2\bigr)&=&k_D\bigl(\phi_s(0),\phi_s(\zeta
	_1)\bigr)\\
	&\le& k_\Delta(0,\zeta_1)<k_D(z_1,w_1)+\eps\;,
\end{eqnarray*}
and (i) follows because $\eps$ is arbitrary.

Given $z_0\in D$, we obtain a particular case of (i) by setting $z_1=z_2=z_0$:
\begin{equation}
k_D\bigl(z_0,sw_1+(1-s)w_2\bigr)\le\max\{k_D(z_0,w_1),k_D(z_0,w_2)\}
\label{eq:1.convuno}
\end{equation}
for all  $z_0$, $w_1$, $w_2\in D$ and $s\in[0,1]$.

To prove (ii), put $z_0=sz+(1-s)w$; then two applications of (\ref{eq:1.convuno})
yield
\begin{eqnarray*}
k_D\bigl(sz+(1-s)w,tz+(1-t)w\bigr)&\le&\max\bigl\{k_D\bigl(sz+(1-s)w,z\bigr),
k_D\bigl(sz+(1-s)w,w\bigr)\bigr\}\\
&\le& k_D(z,w)\;,
\end{eqnarray*}
and we are done.
\qed
\end{proof}

\begin{corollary}
\label{th:1.convexballs}
Closed Kobayashi balls in a hyperbolic convex domain are compact and convex.
\end{corollary}

\begin{proof}
The compactness follows from Propositions~\ref{th:2.chyp} and~\ref{th:1.Barth} (and Remark~\ref{rem:1:hypcvx} for unbounded hyperbolic convex domains); 
the convexity follows from (\ref{eq:1.convuno}). 
\qed
\end{proof}

\subsection{Strongly pseudoconvex domains}
\label{subsect:1.5}

Another important class of domains where the Kobayashi distance has been studied in detail is given by strongly pseudoconvex domains. In particular, in strongly pseudoconvex domains it is possible to estimate the Kobayashi distance by means of the Euclidean distance from the boundary. 

To recall the definition of strongly pseudoconvex domains, and to fix notations useful later, let us first
introduce smoothly bounded domains. For simplicity we shall state the following definitions in~$\mathbb{R}^N$, but they can be easily adapted to~$\mathbb{C}^n$ by using the standard identification $\mathbb{C}^n\simeq\mathbb{R}^{2n}$.

\begin{definition}
\label{def:1.smoothdomain}
A domain $D\subset\mathbb{R}^N$ has \emph{$C^r$~boundary} (or is a \emph{$C^r$~domain}), where $r\in\mathbb{N}\cup\{\infty,\omega\}$ (and $C^\omega$~means 
real analytic), if there is a $C^r$~function $\rho\colon\mathbb{R}^N\to\mathbb{R}$ such 
that:
\begin{enumerate}
\item[(a)] $D=\{x\in\mathbb{R}^N\mid\rho(x)<0\}$;
\item[(b)] $\partial D=\{x\in\mathbb{R}^N\mid\rho(x)=0\}$; and
\item[(c)] $\hbox{grad}\,\rho$ is never vanishing on~$\partial D$.
\end{enumerate}
The function $\rho$~is a \emph{defining function\/} for~$D$. The \emph{outer unit normal vector\/}~$\mathbf{n}_x$ at~$x$ is the unit vector 
parallel to~$-\hbox{grad}\,\rho(x)$. 
\end{definition}

\begin{remark}
it is not difficult to check that if $\rho_1$~is another defining 
function for a domain~$D$ then there is a never vanishing $C^r$~function $\psi
\colon\mathbb{R}^N\to\mathbb{R}^+$ such that 
\begin{equation}
\rho_1=\psi\rho\;.
\label{eq:1.deff}
\end{equation}
\end{remark}
 
If $D\subset\mathbb{R}^N$ is a $C^r$~domain with defining function~$\rho$, then $\partial
D$ is a $C^r$~manifold embedded in~$\mathbb{R}^N$. In particular, for every~$x
\in\partial D$ the tangent space of~$\partial D$ at~$x$ can 
be identified with the kernel of~$\D\rho_x$ (which by (\ref{eq:1.deff}) is independent 
of the chosen defining function). In particular,
$T_x(\partial D)$ is just the hyperplane orthogonal to~$\mathbf{n}_x$.

Using a defining function it is possible to check when a $C^2$-domain is convex.

\begin{definition}
\label{def:1.Hessian}
If $\rho\colon\mathbb{R}^N\to\mathbb{R}$~is a $C^2$~function, the \emph{Hessian}~$H_{\rho,x}$ 
of~$\rho$ at~$x\in\mathbb{R}^N$ is the symmetric bilinear form given by
\[
H_{\rho,x}(v,w)=\sum_{h,k=1}^N\frac{\partial^2\rho}{\partial x_h
	\partial x_k}(x)\,v_h w_k
\]
for every $v$, $w\in\mathbb{R}^N$.
\end{definition}

The following result is well-known (see, e.g, \cite[p.102]{Krantz}):

\begin{proposition}
A $C^2$~domain $D\subset\mathbb{R}^N$ is convex if and only if
for every~$x\in\partial D$ the Hessian~$H_
{\rho,x}$ is positive semidefinite on~$T_x(\partial D)$, where $\rho$~is any 
defining function for~$D$.
\end{proposition}

This suggests the following 

\begin{definition}
\label{def:1.stronglyconvex}
A $C^2$~domain $D\subset\mathbb{R}^N$ is \emph{strongly convex} at~$x\in\partial D$ if for some (and hence any) $C^2$~defining 
function~$\rho$ for~$D$ the Hessian~$H_{\rho,x}$ is positive definite on~$T_x(\partial D)$.
We say that $D$~is \emph{strongly convex} if it 
is so at each point of~$\partial D$.
\end{definition}

\begin{remark} 
It is easy to check that strongly convex $C^2$ domains are strictly convex. Furthermore, 
it is also possible to prove that every strongly convex domain~$D$ has a 
$C^2$~defining function~$\rho$ such that $H_{\rho,x}$~is positive definite 
on the whole of~$\mathbb{R}^N$ for every~$x\in\partial D$ (see, e.g., \cite[p. 101]{Krantz}).
\end{remark}

\begin{remark}
If $D\subset\C^n$ is a convex $C^1$ domain and $x\in\de D$ then the unique (up to a positive multiple) complex supporting functional at~$x$ is given by $L(z)=\langle z,\mathbf{n}_x\rangle$. In particular, $\hbox{\rm Ch}(x)$ coincides with the intersection of the associated complex supporting hyperplane with~$\de D$. But non-smooth points can have more than one complex supporting hyperplanes; this
happens for instance in the polydisk.
\end{remark}

Let us now move to a more complex setting.

\begin{definition}
Let $D\subset\C^n$ be a 
domain with $C^2$~boundary and defining function~$\rho\colon\C^n\to\R$. The 
\emph{complex tangent space}~$T^{\C}_x(\partial D)$ of~$\partial D$ at~$x\in\partial D$ 
is the kernel of~$\partial\rho_x$, that is
\[
T_x^{\C}(\partial D)=\biggl\{v\in\C^n\biggm|\sum_{j=1}^n\frac{\de\rho}{\de z_j}(x)
	\,v_j=0\biggr\}\;.
\]
As usual, $T_x^\C(\de D)$ does not depend on the particular defining 
function. 
The \emph{Levi form\/}~$L_{\rho,x}$ of~$\rho$ at~$x\in\C^n$ is the Hermitian 
form given by
\[
L_{\rho,x}(v,w)=\sum_{h,k=1}^n\frac{\de^2\rho}{\de 
	z_h\de\bar z_k}(x)\,v_h\overline{w_k}
\]
for every $v$, $w\in\C^n$.
\end{definition}

\begin{definition}
A  $C^2$~domain $D\subset\C^n$ is called \emph{strongly pseudoconvex} (respectively, \emph{weakly pseudoconvex}) at a point~$x\in\de D$ if for some 
(and hence all) $C^2$~defining function~$\rho$ for~$D$ the Levi form~$L_{\rho,
x}$ is positive definite (respectively, weakly positive definite) on~$T^\C_x(\de D)$. The domain $D$~is \emph{strongly pseudoconvex} (respectively, \emph{weakly pseudoconvex}) if it is so at each point of~$\de D$.  
\end{definition}

\begin{remark}
If $D$~is strongly pseudoconvex then
there is a defining function~$\rho$ for~$D$ such that the Levi form~$L_{
\rho,x}$ is positive definite on~$\C^n$ for every~$x\in\de D$ (see, e.g., \cite[p. 109]{Krantz}). 
\end{remark}

Roughly speaking, strongly pseudoconvex domains are locally strongly convex. More precisely,
one can prove (see, e.g., \cite[Proposition~2.1.13]{Abatebook}) the following:

\begin{proposition}
\label{th:1.Narasimhan}
A bounded $C^2$~domain $D\subset\subset\C^n$ 
is strongly pseudoconvex if and only if for every $x\in\de D$ there is a neighborhood~$U_x
\subset\C^n$ and a biholomorphism~$\Phi_x\colon U_x\to\Phi_x(U_x)$ such that $\Phi_x(U_x
\cap D)$~is strongly convex.
\end{proposition}

From this one can prove that strongly pseudoconvex domains are taut; but we shall directly prove that they are complete hyperbolic, as a consequence of the boundary estimates we are now going to state.

\begin{definition}
If $M\subset\C^n$ is any subset of~$\C^n$, we shall denote by $d(\cdot,M)\colon\C^n\to\R^+$ the Euclidean distance from~$M$, defined by
\[
d(z,M)=\inf\{\|z-x\|\mid x\in M\}\;.
\]
\end{definition}

To give an idea of the kind of estimates we are looking for, we shall prove an easy lemma:

\begin{lemma}
\label{th:1.bball} 
Let $B_r\subset\C^n$~be the euclidean ball of radius~$r>0$ in~$\C^n$ 
centered at the origin. Then 
\[
\mlog r -\mlog d(z,\partial B_r)\le k_{B_r}(O,z)\le\mlog(2r)
        -\mlog d(z,\partial B_r)
\]
for every $z\in B_r$.
\end{lemma}

\begin{proof} 
We have
\[
k_{B_r}(O,z)=\mlog \frac{1+\|z\|/r}{1-\|z\|/r}\;,
\]
and $d(z,\partial B_r)=r-\|z\|$.
Then, setting $t=\|z\|/r$, we get
\begin{eqnarray*}
\mlog r-\mlog d(z,\partial B_r)&=&\mlog\frac{1}{1-t}\le
\mlog\frac{1+t}{1-t}=k_{B_r}(O,z)\\
&\le&\mlog\frac{2}{1-t}=\mlog(2r)-\mlog d(z,\partial B_r)\;,
\end{eqnarray*}
as claimed.
\qed
\end{proof}

Thus in the ball the Kobayashi distance from a reference point is comparable with one-half of the logarithm of the Euclidean distance from the boundary.
We would like to prove similar estimates in strongly pseudoconvex domains. To do so
we need one more definition.
 
\begin{definition} 
Let $M$ be a compact $C^2$-hypersurface of~$\R^N$, 
and fix an unit normal vector field~$\mathbf{n}$ on~$M$. We shall say that 
$M$~has a \emph{tubular neighborhood} of \emph{radius}~$\eps>0$ 
if the segments $\{x+t\mathbf{n}_x\mid t\in(-\eps,\eps)\}$ are pairwise 
disjoint, and we set
\[
U_\eps=\bigcup_{x\in M}\{x+t\mathbf{n}_x\mid t\in(-\eps,\eps)\}\;.
\]
Note that if $M$ has a tubular neighborhood of radius~$\eps$, then
$d(x+t\mathbf{n}_x,M)=|t|$ for every~$t\in(-\eps,\eps)$ and~$x\in M$; in 
particular, $U_\eps$ is the union of the Euclidean balls $B(x,\eps)$ of center $x\in M$ and radius~$\eps$.
\end{definition}

\begin{remark}
A proof of the existence of 
a tubular neighborhood of radius sufficiently small for any compact 
$C^2$-hypersurface of~$\R^N$ can be found, e.g., in \cite[Theorem 10.19]{Lee}.
\end{remark}

And now, we begin proving the estimates. The upper estimate does not even 
depend on the strong pseudoconvexity:

\begin{theorem}[\cite{Vormoor, Aba86}]
\label{th:1.bestiu} 
Let $D\subset\subset\C^n$ be a bounded $C^2$ domain, and
$z_0\in D$. Then there is a constant $c_1\in\R$
depending only on~$D$ and~$z_0$ such that 
\begin{equation}
k_D(z_0,z)\le c_1-\mlog d(z,\partial D)
\label{eq:1.bestiu}
\end{equation}
for all $z\in D$.
\end{theorem}

\begin{proof} 
Since $D$~is a bounded $C^2$~domain, $\partial D$~admits tubular 
neighborhoods~$U_\eps$ of radius~$\eps<1$ small enough. Put
\[
c_1=\sup\bigl\{k_D(z_0,w)\bigm|w\in D\setminus U_{\eps/4}\bigr\}
        +\max\bigl\{0,\,\mlog\mathrm{diam}(D)\bigr\}\;,
\]
where diam$(D)$ is the Euclidean diameter of~$D$.

There are two cases:
\begin{enumerate}
\item[(i)] $z\in U_{\eps/4}\cap D$. Let $x\in\partial D$ be such that
$\|x-z\|=d(z,\partial D)$. Since $U_{\eps/2}$~is a tubular neighborhood
of~$\partial D$, there exists~$\lambda\in\R$ such that $w=\lambda(x-z)\in
\partial U_{\eps/2}\cap D$ and the euclidean ball~$B$ of center~$w$ and
radius~$\eps/2$ is contained in~$U_\eps\cap D$ and tangent to~$\partial D$
in~$x$. Therefore Lemma~\ref{th:1.bball} yields
\begin{eqnarray*}
k_D(z_0,z)&\le& k_D(z_0,w)+k_D(w,z)\le
        k_D(z_0,w)+k_B(w,z)\\
        &\le& k_D(z_0,w)+\mlog\eps-\mlog d(z,\partial B)\\
        &\le& c_1-\mlog d(z,\partial D)\;,
\end{eqnarray*}
because $w\notin U_{\eps/4}$ (and $\eps<1$).

\item[(ii)] $z\in D\setminus U_{\eps/4}$. Then
\[
k_D(z_0,z)\le c_1-\mlog\mathrm{diam}(D)\le c_1-\mlog d(z,\partial
        D)\;,
\]
because $d(z,\partial D)\le \mathrm{diam}(D)$, and we are done.
\qed
\end{enumerate}
\end{proof}

To prove the more interesting lower estimate, we need to introduce the last definition of this subsection.

\begin{definition}
\label{def:1.peak}
Let $D\subset\C^n$ be a domain in~$\C^n$, and $x\in\de D$. A \emph{peak function} for~$D$ at~$x$ is a holomorphic function $\psi\in\Hol(D,\Delta)$ continuous up to the boundary of~$D$ such that
$\psi(x)=1$ and $|\psi(z)|<1$ for all $z\in\overline{D}\setminus\{x\}$.
\end{definition}

If $D\subset\C^n$ is strongly convex and $x\in\de D$ then by Remark~\ref{rem:1.suppf} there exists a peak function for~$D$ at~$x$.
Since a strongly pseudoconvex domain $D$ is locally strongly convex, using Proposition~\ref{th:1.Narasimhan} one can easily build peak functions defined in a neighborhood of a point of the boundary of~$D$. To prove the more interesting lower estimate on the Kobayashi distance we shall need the non-trivial fact that in a strongly pseudoconvex domain it is possible to build a family of \emph{global} peak functions continuously dependent on the point in the boundary:

\begin{theorem}[Graham, \cite{Graham}]
\label{th:1.Graham} 
Let $D\subset\subset\C^n$ be a strongly 
pseudoconvex $C^2$ domain. Then there exist a neighborhood~$D'$ of~$\overline{D}$ and a continuous function $\Psi\colon\partial D\times D'\to\C$ such that
$\Psi_{x_0}=\Psi(x_0,\cdot)$ is holomorphic in $D'$ and a peak function for $D$ at $x_0$ for each 
$x_0\in\partial D$.
\end{theorem}

With this result we can prove 

\begin{theorem}[\cite{Vormoor, Aba86}]
\label{th:1.bestil} 
Let $D\subset\subset\C^n$ be a bounded strongly pseudoconvex $C62$
domain, and~$z_0\in D$. Then there is a constant $c_2\in\R$ depending only 
on~$D$ and~$z_0$ such that 
\begin{equation}
c_2-\mlog d(z,\partial D)\le k_D(z_0,z)
\label{eq:1.bestil}
\end{equation}
for all $z\in D$.
\end{theorem}

\begin{proof} 
Let~$D'\supset\supset D$ and~$\Psi\colon\partial D\times D'\to\C$ be given 
by Theorem~\ref{th:1.Graham}, and define $\phi\colon\partial D\times\Delta\to\C$~by
\begin{equation}
\phi(x,\zeta)=\frac{1-\overline{\Psi(x,z_0)}}{1-\Psi(x,z_0)}\cdot
        \frac{\zeta-\Psi(x,z_0)}{1-\overline{\Psi(x,z_0)}\zeta}\;.
\label{eq:1.change}
\end{equation}
Then the map $\Phi(x,z)=\Phi_x(z)=\phi\bigl(x,\Psi(x,z)\bigr)$ is defined on a
neighborhood~$\partial D\times D_0$ of~$\partial D\times\overline{D}$
(with~$D_0\subset\subset D'$) and satisfies
\begin{enumerate}
\item[(a)] $\Phi$ is continuous, and $\Phi_x$~is a
holomorphic peak function for~$D$ at~$x$ for any~$x\in\partial D$;
\item[(b)] for every~$x\in\partial D$ we have~$\Phi_x(z_0)=0$.
\end{enumerate}
Now set $U_\eps=\bigcup_{x\in\partial D}P(x,\eps)$, where $P(x,\eps)$
is the polydisk of center~$x$ and polyradius~$(\eps,\ldots,\eps)$.
The family~$\{U_\eps\}$ is a basis for the neighborhoods of~$\partial D$;
hence there exists~$\eps>0$ such that~$U_\eps\subset\subset D_0$ and $U_\eps$
is contained in a tubular neighborhood of~$\partial D$. Then for
any~$x\in\partial D$ and~$z\in P(x,\eps/2)$ the Cauchy estimates yield
\begin{eqnarray*}
|1-\Phi_x(z)|=|\Phi_x(x)-\Phi_x(z)|&\le&\bigg\|\frac{\de\Phi_x}{\de z}
        \bigg\|_{P(x,\eps/2)}\|z-x\|\\
	&\le&\frac{2\sqrt{n}}{\eps}\|\Phi\|_{\de D\times U_\eps}\|z-x\|=M\|z-x\|\;,
\end{eqnarray*}
where $M$ is independent of~$z$ and~$x$; in these formulas $\|F\|_S$ denotes the supremum
of the Euclidean norm of the map~$F$ on the set~$S$. 

Put $c_2=-\mlog M$;
note that $c_2\le\mlog(\eps/2)$, for $\|\Phi\|_{\de D\times U_\eps}\ge1$.
Then we again have two cases:
\begin{enumerate}
\item[(i)] $z\in D\cap U_{\eps/2}$. Choose $x\in\partial D$ so that
$d(z,\partial D)=\|z-x\|$. Since~$\Phi_x(D)\subset\Delta$ and~$\Phi_x(z_0)=0$,
we have
\[
k_D(z_0,z)\ge k_\Delta\bigl(\Phi_x(z_0),\Phi_x(z)\bigr)\ge
        \mlog\frac{1}{1-|\Phi_x(z)|}\;.
\]
Now,
\[
1-|\Phi_x(z)|\le|1-\Phi_x(z)|\le M\|z-x\|=M\,d(z,\partial D)\;;
\]
therefore
\[
k_D(z_0,z)\ge -\mlog M-\mlog d(z,\partial D)= c_2-\mlog
        d(z,\partial D)
\]
as desired.
\item[(ii)] $z\in D\setminus U_{\eps/2}$. Then $d(z,\partial D)\ge\eps/2$;
hence
\[
k_D(z_0,z)\ge0\ge\mlog(\eps/2)-\mlog
        d(z,\partial D)\ge c_2-\mlog d(z,\partial D)\;,
\]
and we are done. 
\qed
\end{enumerate}
\end{proof}

A first consequence is the promised:

\begin{corollary}[Graham, \cite{Graham}]
\label{th:1.pscomp}
Every bounded strongly pseudoconvex $C^2$ domain~$D$ is complete hyperbolic.
\end{corollary}

\begin{proof} 
Take $z_0\in D$, $r>0$ and let $z\in B_D(z_0,r)$. Then (\ref{eq:1.bestil}) 
yields
\[
d(z,\partial D)\ge\exp\bigl(2(c_2-r)\bigr)\;,
\]
where $c_2$ depends only on $z_0$. Then $B_D(z_0,r)$ is relatively compact 
in $D$, and the assertion follows from Proposition~\ref{th:2.chyp}.
\qed
\end{proof}

For dynamical applications we shall also need estimates on the Kobayashi distance
$k_D(z_1,z_2)$ when both $z_1$ and $z_2$ are close to the boundary. The needed estimates 
were proved by Forstneri\v c and Rosay (see \cite{FR}, and \cite[Corollary~2.3.55, Theorem~2.3.56]{Abatebook}):
 
\begin{theorem}[\cite{FR}]
\label{th:1.bestia} 
Let $D\subset\subset\C^n$ be a bounded strongly 
pseudoconvex $C^2$ domain, and choose two points $x_1$,~$x_2\in\de D$ 
with $x_1\ne x_2$. Then there exist~$\eps_0>0$ and~$K\in\R$ such that for any 
$z_1$, $z_2\in D$ with $\|z_j-x_j\|<\eps_0$ for $j=1$,~$2$ we have
\begin{equation}
k_D(z_1,z_2)\ge-\mlog d(z_1,\partial D)-\mlog d(z_2,\partial D)+K\;.
\label{eq:1.bestia}
\end{equation}
\end{theorem}

\begin{theorem}[\cite{FR}]
\label{th:1.bestib}
Let $D\subset\subset\C^n$ be a bounded $C^2$ domain and 
$x_0\in\partial D$. Then there exist~$\eps>0$ and~$C\in\R$ 
such that for all $z_1$,~$z_2\in D$ with $\|z_j-x_0\|<\eps$ for $j=1$,~$2$ we have
\begin{equation}
k_D(z_1,z_2)\le \mlog\left(1+\frac{\|z_1-z_2\|}{d(z_1,\de D)}\right)+
\mlog\left(1+\frac{\|z_1-z_2\|}{d(z_2,\de D)}\right)+C\;.
\label{eq:1.bestib}
\end{equation}
\end{theorem}

We end this section by quoting a theorem, that we shall need in Section~\ref{sec:3}, giving a different way of comparing the Kobayashi geometry and the Euclidean geometry of strongly pseudoconvex domains:

\begin{theorem}[\cite{AS}]
\label{th:1.volume}
Let $D\subset\subset\C^n$ be a strongly pseudoconvex $C^\infty$ domain, and $R>0$. Then there exist
$C_R>0$ depending only on $R$ and $D$ such that
\[
\frac{1}{C_R} d(z_0,\de D)^{n+1}\le \nu\bigl(B_D(z_0,R)\bigr)\le C_R d(z_0,\de D)^{n+1}
\]
for all $z_0\in D$, where $\nu\bigl(B_D(z_0,R)\bigr)$ denotes the Lebesgue volume of the Kobayashi ball $B_D(z_0,R)$.
\end{theorem}
 
\section{Holomorphic dynamics}
\label{sec:2}

In this section we shall describe the dynamics of holomorphic self-maps of taut manifolds, and in particular the dynamics of holomorphic self-maps of convex and strongly pseudoconvex domains. A main tool in this exploration will be provided by the Kobayashi distance.

\begin{definition}
Let $f\colon X\to X$ be a self-map of a set~$X$. Given $k\in\mathbb{N}$, we define the 
$k$-th \emph{iterate}~$f^k$ of~$f$ setting by induction $f^0=\mathrm{id}_X$, $f^1=f$ and
$f^k=f\circ f^{k-1}$. Given $x\in X$, the \emph{orbit} of~$x$ is the set $\{f^k(x)\mid k\in\mathbb{N}\}$. 
\end{definition}

Studying the dynamics of a self-map $f$ means studying the asymptotic behavior of the sequence $\{f^k\}$ of iterates of~$f$; in particular, in principle one would like to know the behavior of all orbits. In general this is too an ambitious task; but as we shall see it can be achieved for holomorphic self-maps of taut manifolds, because the normality condition prevents the occurrence of chaotic behavior.
 
The model theorem for this theory is the famous Wolff-Denjoy theorem (for a proof see, e.g., \cite[Theorem~1.3.9]{Abatebook}):

\begin{theorem}[Wolff-Denjoy, \cite{Wo, De}]
\label{th:2.WD}
Let $f\in\Hol(\Delta,\Delta)\setminus\{\mathrm{id}_\Delta\}$ be a holomorphic self-map of $\Delta$ different from the identity. Assume that $f$ is not an elliptic automorphism. Then the sequence of iterates of $f$ converges, uniformly on compact subsets, to a constant map $\tau\in\overline{\Delta}$.
\end{theorem}

\begin{definition}
Let $f\in\Hol(\Delta,\Delta)\setminus\{\mathrm{id}_\Delta\}$ be a holomorphic self-map of $\Delta$ different from the identity and not an elliptic automorphism. Then the point $\tau\in\overline{\Delta}$ whose existence is asserted by Theorem~\ref{th:2.WD} is the \emph{Wolff point} of~$f$.
\end{definition}

Actually, we can even be slightly more precise, introducing a bit of terminology.

\begin{definition}
Let $f\colon X\to X$ be a self-map of a set~$X$. A \emph{fixed point} of $f$ is a point $x_0\in X$ such that $f(x_0)=x_0$. We shall denote by $\Fix(f)$ the set of fixed points of~$f$. More generally, we shall say that $x_0\in X$ is \emph{periodic} of \emph{period} $p\ge 1$ if $f^p(x_0)=x_0$ and
$f^j(x_0)\ne x_0$ for all $j=1,\ldots,p-1$.
We shall say that $f$ is \emph{periodic} of \emph{period}~$p\ge 1$ if $f^p=\mathrm{id}_X$, that is if all points are periodic of period~$p$. 
\end{definition}

\begin{definition}
Let $f\colon X\to X$ be a continuous self-map of a topological space~$X$. We shall say that a continuous map $g\colon X\to X$ is a \emph{limit map} of~$f$ if there is a subsequence of iterates of~$f$ converging to~$g$ (uniformly on compact subsets). We shall denote by $\Gamma(f)\subset C^0(X,X)$ the set of limit maps of~$f$. If $\mathrm{id}_X\in\Gamma(f)$ we shall say that $f$ is \emph{pseudoperiodic}.
\end{definition}

\begin{example}
\label{ex:2.ellaut}
Let $\gamma_\theta\in\Aut(\Delta)$ be given by $\gamma_\theta(\zeta)=\E^{2\pi\I\theta}\zeta$. It is easy to check that $\gamma_\theta$ is periodic if $\theta\in\mathbb{Q}$, and it is pseudoperiodic (but not periodic) if $\theta\in\R\setminus\mathbb{Q}$.
\end{example}

\begin{definition}
Let $X$ and $Y$ be two sets (topological spaces, complex manifolds, etc.). Two self-maps $f\colon X\to X$ and $g\colon Y\to Y$ are \emph{conjugate} if there exists a bijection (homeomorphism, biholomorphism, etc.) $\psi\colon X\to Y$ such that $f=\psi^{-1}\circ g\circ \psi$. 
\end{definition}

If $f$ and $g$ are conjugate via $\psi$, we clearly have $f^k=\psi^{-1}\circ g^k\circ\psi$ for all $k\in\mathbb{N}$; therefore $f$ and $g$ share the same dynamical properties.

\begin{example}
It is easy to check that any elliptic automorphism of $\Delta$ is (biholomorphically) conjugated to one of the automorphisms $\gamma_\theta$ introduced in Example~\ref{ex:2.ellaut}. Therefore an elliptic automorphism of $\Delta$ is necessarily periodic or pseudoperiodic.
\end{example}

We can now better specify the content of Theorem~\ref{th:2.WD} as follows. Take $f\in\Hol(\Delta,\Delta)$ different from the identity. We have two cases: either $f$ has a fixed point $\tau\in\Delta$ or $\Fix(f)=\emptyset$ (notice that, by the Schwarz-Pick lemma and the structure of the automorphisms of~$\Delta$, the only holomorphic self-map of~$\Delta$ with at least two distinct fixed points is the identity). Then:
\begin{enumerate}
\item[(a)] If $\Fix(f)=\{\tau\}$, then either $f$ is an elliptic automorphism---and hence it is periodic or pseudoperiodc---or the whole sequence of iterates converges to the constant function~$\tau$;
\item[(b)] if $\Fix(f)=\emptyset$ then there exists a unique point $\tau\in\de\Delta$ such that the whole sequence of iterates converges to the constant function~$\tau$.
\end{enumerate}
So there is a natural dichotomy between self-maps with fixed points and self-maps without fixed points. Our aim is to present a (suitable) generalization of the Wolff-Denjoy theorem to taut manifolds in any (finite) dimension. Even in several variables a natural dichotomy will appear; but it will be slightly different.

\subsection{Dynamics in taut manifolds}
\label{subsec:2.1}

Let $X$ be a taut manifold. Then the whole family $\Hol(X,X)$ is normal; in particular, if $f\in\Hol(X,X)$ the sequence of iterates $\{f^k\}$ is normal. This suggests to subdivide the study of the dynamics of self-maps of $X$ in three tasks:
\begin{enumerate}
\item[(a)] to study the dynamics of $f$ when the sequence $\{f^k\}$ is not compactly divergent;
\item[(b)] to find conditions on $f$ ensuring that the sequence $\{f^k\}$ is not compactly divergent;
\item[(c)] to study the dynamics of $f$ when the sequence $\{f^k\}$ is compactly divergent.
\end{enumerate}
So in several variables the natural dichotomy to consider is between maps having a compactly divergent sequence of iterates and maps whose sequence of iterates is not compactly divergent.
If $f$ has a fixed point its sequence of iterates cannot be compactly divergent; so this dichotomy has something to do with the dichotomy discussed in the introduction to this section but, as we shall see, in general they are not the same.

In this subsection we shall discuss tasks (a) and (b). To discuss task (c) we shall need a boundary; we shall limit ourselves to discuss (in the next three subsections) the case of bounded (convex or strongly pseudoconvex) domains in~$\C^n$.

An useful notion for our discussion is the following

\begin{definition}
A \emph{holomorphic retraction} of a complex manifold $X$ is a holomorphic self-map $\rho\in\Hol(X,X)$ such that $\rho^2=\rho$. In particular, $\rho(X)=\Fix(\rho)$. The image of a holomorphic retraction is a \emph{holomorphic retract.}
\end{definition}

The dynamics of holomorphic retraction is trivial: the iteration stops at the second step. On the other had, it is easy to understand why holomorphic retractions might be important in holomorphic dynamics. Indeed, assume that the sequence of iterates $\{f^k\}$ converges to a map~$\rho$.
Then the subsequence $\{f^{2k}\}$ should converge to the same map; but $f^{2k}=f^k\circ f^k$, and thus $\{f^{2k}\}$ converges to $\rho\circ\rho$ too---and thus $\rho^2=\rho$, that is $\rho$ is a holomorphic retraction.

In dimension one, a holomorphic retraction must be either the identity or a constant map, because of the open mapping theorem and the identity principle. In several variables there is instead plenty of non-trivial holomorphic retractions.

\begin{example}
Let $B^2$~be the unit 
euclidean ball in~$\C^2$. The power series
\[
1-\sqrt{1-t}=\sum_{k=1}^\infty c_kt^k
\]
is converging for $|t|<1$ and has $c_k>0$ for all~$k\ge1$. Take $g_k\in\Hol
(B^2,\C)$ such that $|g_k(z,w)|\le c_k$ for all~$(z,w)\in B^2$, and define
$\phi\in\Hol(B^2,\Delta)$ by
\[
\phi(z,w)=z+\sum_{k=1}^\infty g_k(z,w)\,w^{2k}.\;
\]
Then $\rho(z,w)=\bigl(\phi(z,w),0)$ always satisfies $\rho^2=\rho$, and it 
is neither constant nor the identity.
\end{example} 

On the other hans, holomorphic retracts cannot be wild. This has been proven for the first time by Rossi \cite{Rossi}; here we report a clever proof due to H. Cartan \cite{Ca}:  

\begin{lemma}
\label{th:2.ret}
Let $X$ be a complex manifold, and $\rho\colon X\to X$ a 
holomorphic retraction of~$X$. Then the image of $\rho$ is a closed submanifold
of~$X$.
\end{lemma}

\begin{proof} 
Let $M=\rho(X)$ be the image of~$\rho$, and take~$z_0\in M$. Choose an 
open neighborhood~$U$ of~$z_0$ in~$X$ contained in a local chart for~$X$ 
at~$z_0$. Then $V=\rho^{-1}(U)\cap U$ is an open neighborhood of~$z_0$ 
contained in a local chart such that $\rho(V)\subseteq V$. Therefore 
without loss of generality we can assume that $X$~is a bounded domain~$D$ 
in~$\C^n$.

Set $P=\D\rho_{z_0}\colon\C^n\to\C^n$, and define $\phe\colon D\to\C^n$ by
\[
\phe=\mathrm{id}_D+(2P-\mathrm{id}_D)\circ(\rho-P)\;.
\]
Since $\D\phe_{z_0}=\mathrm{id}$, the map $\phe$~defines a local chart in a neighborhood 
of~$z_0$. Now $P^2=P$ and $\rho^2=\rho$; hence
\begin{eqnarray*}
\phe\circ\rho&=&\rho+(2P-\mathrm{id}_D)\circ\rho^2-(2P-\mathrm{id}_D)\circ P\circ
	\rho\\
	&=&P\circ\rho=P+P\circ(2P-\mathrm{id}_D)\circ(\rho-P)=P\circ\phe\;.
\end{eqnarray*}
Therefore in this local chart $\rho$ becomes linear, and $M$~is a 
submanifold near~$z_0$. By the arbitrariness of~$z_0$, the assertion 
follows. 
\qed
\end{proof}

Having the notion of holomorphic retraction, we can immediately explain why holomorphic dynamics is trivial in compact hyperbolic manifolds (for a proof see, e.g., \cite[Theorem~2.4.9]{Abatebook}):

\begin{theorem}[Kaup, \cite{Kaup}]
Let $X$ be a compact hyperbolic manifold, and $f\in\Hol(X,X)$. Then there is $m\in\mathbb{N}$ such that $f^m$ is a holomorphic retraction.
\end{theorem}

So from now on we shall concentrate on non-compact taut manifolds. The basic result describing the dynamics of self-maps whose sequence of iterates is not compactly divergent is the following:

\begin{theorem}[Bedford, \cite{Bed}; Abate, \cite{Ab1}]
\label{th:2.BedAba}
Let $X$ be a taut manifold, and $f\in\Hol(X,X)$. 
Assume that the sequence~$\{f^k\}$ of iterates of~$f$ is not compactly 
divergent. Then there exist a unique holomorphic retraction $\rho\in\Gamma(f)$ onto a submanifold~$M$ of~$X$ such that every limit map $h\in\Gamma(f)$ 
is of the form
\begin{equation}
h=\gamma\circ\rho\;,
\label{eq:2.limit}
\end{equation}
where~$\gamma$ is an automorphism of~$M$. Furthermore, $\phe=f|_M\in\Aut(M)$ and $\Gamma(f)$ is isomorphic to a subgroup of~$\Aut(M)$, the closure of $\{\phe^k\}$ in~$\Aut(M)$.
\end{theorem} 

\begin{proof} 
Since the sequence $\{f^k\}$ of iterates is not compactly divergent, it must contain a subsequence
$\{f^{k_\nu}\}$ converging to $h\in\Hol
(X,X)$. We can also assume that $p_\nu=k_{\nu+1}-k_\nu$ and $q_\nu=p_\nu-k_\nu=
k_{\nu+1}-2k_\nu$ tend to~$+\infty$ as~$\nu\to+\infty$, and that $\{f^{p_\nu}\}$
and $\{f^{q_\nu}\}$ are either converging or compactly divergent. Now we have
\[
\lim_{\nu\to\infty}f^{p_\nu}\bigl(f^{k_\nu}(z)\bigr)=\lim_{\nu\to\infty}f^
	{k_{\nu+1}}(z)=h(z)
\]
for all~$z\in X$; 
therefore $\{f^{p_\nu}\}$ cannot be compactly divergent, and thus converges to 
a map~$\rho\in\Hol(X,X)$ such that 
\begin{equation}
h\circ\rho=\rho\circ h=h\;.
\label{eq:2.rho}
\end{equation}
Next, for all~$z\in X$ we have
\[
\lim_{\nu\to\infty}f^{q_\nu}\bigl(f^{k_\nu}(z)\bigr)=\lim_{\nu\to\infty}f^
	{p_\nu}(z)=\rho(z)\;.
\]
Hence neither $\{f^{q_\nu}\}$ can be compactly divergent, and thus converges to a 
map~$g\in\Hol(X,X)$ such that
\begin{equation}
g\circ h=h\circ g=\rho\;.
\label{eq:2.gi}
\end{equation}
In particular
\[
\rho^2=\rho\circ\rho=g\circ h\circ\rho=g\circ h=\rho\;,
\]
and $\rho$ is a holomorphic retraction of $X$ onto a submanifold $M$. 
Now (\ref{eq:2.rho}) implies $h(X)\subseteq M$. Since $g\circ\rho=\rho\circ g$, we 
have $g(M)\subseteq M$ and (\ref{eq:2.gi}) yields 
\[
g\circ h|_M=h\circ g|_M=\mathrm{id}_M\;;
\]
hence $\gamma=h|_M\in\Aut(M)$ and (\ref{eq:2.rho}) becomes (\ref{eq:2.limit}).

Now, let $\{f^{k'_\nu}\}$ be another subsequence of $\{f^k\}$ converging to a 
map $h'\in\Hol(X,X)$. Arguing as before, we can assume $s_\nu=k'_
\nu-k_\nu$ and $t_\nu=k_{\nu+1}-k'_\nu$ are converging to~$+\infty$ as~$\nu\to+\infty$,
and that $\{f^{s_\nu}\}$ and $\{f^{t_\nu}\}$ converge to holomorphic maps 
$\alpha\in\Hol(X,X)$, respectively $\beta\in\Hol(X,X)$ such that
\begin{equation}
\alpha\circ h=h\circ\alpha=h'\qquad\quad\hbox{and}\qquad\quad\beta\circ
	h'=h'\circ\beta=h\;.
\label{eq:2.ab}
\end{equation}
Then $h(X)=h'(X)$, and so $M$ does not depend on the particular converging 
subsequence.

We now show that $\rho$ itself does not depend on the chosen 
subsequence. Write $h=\gamma_1\circ\rho_1$, $h'=\gamma_2\circ\rho_2$, $\alpha=
\gamma_3\circ\rho_3$ and $\beta=\gamma_4\circ\rho_4$, where $\rho_1$, $\rho_2$, 
$\rho_3$ and $\rho_4$ are holomorphic retractions of $X$ onto $M$, and 
$\gamma_1$, $\gamma_2$, $\gamma_3$ and $\gamma_4$ are automorphisms of $M$. Then
$h\circ h'=h'\circ h$ and $\alpha\circ\beta=\beta\circ\alpha$ together with 
(\ref{eq:2.ab}) become
\begin{equation}
\vcenter{\openup1\jot
	\halign{\hfil$#$\hfil\cr
\gamma_1\circ\gamma_2\circ\rho_2=\gamma_2\circ\gamma_1\circ\rho_1\;,\cr
\gamma_3\circ\gamma_1\circ\rho_1=\gamma_1\circ\gamma_3\circ\rho_3=\gamma_2
	\circ\rho_2\;,\cr
\gamma_4\circ\gamma_2\circ\rho_2=\gamma_2\circ\gamma_4\circ\rho_4=\gamma_1
	\circ\rho_1\;,\cr
\gamma_3\circ\gamma_4\circ\rho_4=\gamma_4\circ\gamma_3\circ\rho_3\;.\cr}}
\label{eq:2.trick}
\end{equation}
Writing $\rho_2$ in function of $\rho_1$ using the first and the second 
equation in (\ref{eq:2.trick}) we find $\gamma_3=\gamma^{-1}_1\circ\gamma_2$. Writing
$\rho_1$ in function of $\rho_2$ using the first and the third equation, we 
get $\gamma_4=\gamma_2^{-1}\circ\gamma_1$. Hence $\gamma_3=\gamma_4^{-1}$ 
and the fourth equation yields $\rho_3=\rho_4$. But then, using the second 
and third equation we obtain
\[
\rho_2=\gamma_3^{-1}\circ\gamma_1^{-1}\circ\gamma_2\circ\rho_2=\rho_3=
	\rho_4=\gamma_4^{-1}\circ\gamma_2^{-1}\circ\gamma_1\circ\rho_1=
	\rho_1\;,
\]
as claimed. 

Next, from $f\circ\rho=\rho\circ f$ it follows immediately that $f(M)\subseteq M$. Put $\phe=f|_M$;
if $f^{p_\nu}\to\rho$ then $f^{p_\nu+1}\to\phe\circ\rho$, and thus $\phe\in\Aut(M)$.

Finally, for each limit point $h=\gamma\circ\rho\in\Gamma(f)$ we have $\gamma^{-1}
\circ\rho\in\Gamma(f)$. Indeed fix a subsequence $\{f^{p_\nu}\}$ converging
to~$\rho$, and a subsequence $\{f^{k_\nu}\}$ converging to~$h$. As usual, we can
assume that $p_\nu-k_\nu\to+\infty$ and $f^{p_\nu-k_\nu}\to h_1=\gamma_1\circ
\rho$ as~$\nu\to+\infty$. Then $h\circ h_1=\rho=h_1\circ h$, that is~$\gamma_1=\gamma^{-1}$.
Hence the association $h=\gamma\circ\rho\mapsto\gamma$ yields an isomorphism between $\Gamma(f)$ and the subgroup of~$\Aut(M)$ obtained as closure of~$\{\phe^k\}$.
\qed
\end{proof}

\begin{definition}
Let $X$ be a taut manifold and $f\in\Hol(X,X)$ such that the sequence $\{f^k\}$ is not compactly divergent. The manifold $M$ whose existence is asserted in the previous theorem is 
the \emph{limit manifold} of the map~$f$, and its dimension is the \emph{limit 
multiplicity}~$m_f$ of~$f$; finally, the holomorphic retraction is the \emph{limit retraction} 
of~$f$. 
\end{definition}

It is also possible to describe precisely the algebraic structure of the group $\Gamma(f)$, because it is compact. This is a consequence of the following theorem (whose proof generalizes an argument due to Ca\l ka \cite{Calka}), that, among other things, says that if a sequence of iterates is not compactly divergent then it does not contain any compactly divergent subsequence, and thus it is relatively compact in $\Hol(X,X)$:

\begin{theorem}[Abate, \cite{Ab2}]
\label{th:2.Calka}
Let $X$ be a taut manifold, and $f\in\Hol(X,X)$. Then the following assertions are equivalent:
\begin{enumerate}
\item[\rm(i)] the sequence of iterates $\{f^k\}$ is not compactly divergent;
\item[\rm(ii)] the sequence of iterates $\{f^k\}$ does not contain any compactly divergent subsequence;
\item[\rm(iii)] $\{f^k\}$ is relatively compact in $\Hol(X,X)$;
\item[\rm(iv)] the orbit of $z\in X$ is relatively compact in~$X$ for all $z\in X$;
\item[\rm(v)] there exists $z_0\in X$ whose orbit is relatively compact in~$X$.
\end{enumerate}
\end{theorem}

\begin{proof}
(v)$\Longrightarrow$(ii). Take $H=\{z_0\}$ and $K=\overline{\{f^k(z_0)\}}$. Then $H$ and $K$ are compact and $f^k(H)\cap K\ne\emptyset$ for all $k\in\mathbb{N}$, and so no subsequence of $\{f^k\}$ can be compactly divergent.

(ii)$\Longrightarrow$(iii). Since $\Hol(X,X)$ is a metrizable topological space, if $\{f^k\}$ is not relatively compact then it admits a subsequence $\{f^{k_\nu}\}$ with no converging subsequences. But then, being $X$ taut, $\{f^{k_\nu}\}$ must contain a compactly divergent subsequence, against~(ii).

(iii)$\Longrightarrow$(iv). The evaluation map $\Hol(X,X)\times X\to X$ is continuous.

(iv)$\Longrightarrow$(i). Obvious.

(i)$\Longrightarrow$(v). Let $M$ be the limit manifold of $f$, and let $\phe=f|_M$. By Theorem~\ref{th:2.BedAba} we know that $\phe\in\Aut(M)$ and that $\mathrm{id}_M\in\Gamma(\phe)$.

Take $z_0\in M$; we would like to prove that $C=\{\phe^k(z_0)\}$ is relatively compact in~$M$ (and hence in~$X$). Choose $\eps_0>0$ so that $B_M(z_0,\eps_0)$ is relatively compact in~$M$; notice that $\phe\in\Aut(M)$ implies that $B_M\bigl(\phe^k(z_0),\eps_0)=\phe^k\bigl(B_M(z_0,\eps_0)\bigr)$ is relatively compact in~$M$ for all $k\in\mathbb{N}$. By Lemma~\ref{th:1.somma} we have
\[
\overline{B_M(z_0,\eps_0)}\subseteq B_M\bigl(B_M(z_0,7\eps_0/8),\eps_0/4\bigr)\;;
\]
hence there are $w_1,\ldots,w_r\in B_M(z_0,7\eps_0/8)$ such that
\[
\overline{B_M(z_0,\eps_0)}\cap C\subset \bigcup_{j=1}^r B_M(w_j,\eps_0/4)\cap C\;,
\]
and we can assume that $B_M(w_j,\eps_0/4)\cap C\ne\emptyset$ for $j=1,\ldots,r$. 

For each $j=1,\ldots,r$ choose $k_j\in\mathbb{N}$ so that $\phe^{k_j}(z_0)\in B_M(w_j,\eps_0/4)$; then
\begin{equation}
B_M(z_0,\eps_0)\cap C\subset\bigcup_{j=1}^r\bigl[B_M\bigl(\phe^{k_j}(z_0),\eps_0/2\bigr)\cap C\bigr]
\label{eq:2.Calka1}
\end{equation}
Since $\mathrm{id}_M\in\Gamma(\phe)$, the set $I=\bigl\{k\in\mathbb{N}\bigm| k_M\bigl(\phe^k(z_0),z_0)<\eps_0/2\bigr)\bigr\}$ is infinite; therefore we can find $k_0\in\mathbb{N}$ such that
\begin{equation}
k_0\ge\max\{1,k_1,\ldots,k_r\}\quad\hbox{and}\quad k_M\bigl(\phe^{k_0}(z_0),z_0\bigr)<\eps_0/2\;.
\label{eq:2.Calkadue}
\end{equation}
Put
\[
K=\bigcup_{k=1}^{k_0}\overline{B_M\bigl(\phe^k(z_0),\eps_0\bigr)}\;;
\]
since, by construction, $K$ is compact, to end the proof it suffices to show that $C\subset K$. 
Take $h_0\in I$; since the set $I$ is infinite, it suffices to show that $\phe^k(z_0)\in K$ for all $0\le k\le h_0$.

Assume, by contradiction, that $h_0$ is the least element of~$I$ such that $\{\phe^k(z_0)\mid 0\le k\le h_0\}$ is not contained in~$K$. Clearly, $h_0>k_0$. Moreover, $k_M\bigl(\phe^{h_0}(z_0),\phe^{k_0}(z_0)\bigr)<\eps_0$ by (\ref{eq:2.Calkadue}); thus
\[
k_M\bigl(\phe^{h_0-j}(z_0),\phe^{k_0-j}(z_0)\bigr)=k_M\bigl(\phe^{h_0}(z_0),\phe^{k_0}(z_0)\bigr)<\eps_0
\]
for every $0\le j\le k_0$. In particular, 
\begin{equation}
\phe^j(z_0)\in K
\label{eq:2.Calkaqua}
\end{equation}
for every $j=h_0-k_0,\ldots,h_0$, and
$\phe^{h_0-k_0}(z_0)\in B_D(z_0,\eps_0)\cap C$. By (\ref{eq:2.Calka1}) we can find $1\le l\le r$ such that $k_M\bigl(\phe^{k_l}(z_0),\phe^{h_0-k_0}(z_0)\bigr)<\eps_0/2$, and so
\begin{equation}
k_M\bigl(\phe^{h_0-k_0-j}(z_0),\phe^{k_l-j}(z_0)\bigr)<\eps_0/2
\label{eq:2.Calkatre}
\end{equation}
for all $0\le j\le\min\{k_l,h_0-k_0\}$. In particular, if $k_l\ge h_0-k_0$ then, by (\ref{eq:2.Calka1}), (\ref{eq:2.Calkaqua}) and (\ref{eq:2.Calkatre}) we have $\phe^j(z_0)\in K$ for all $0\le j\le h_0$, against the choice of~$h_0$. So we must have $k_l<h_0-k_0$; set $h_1=h_0-k_0-k_l$. 
By (\ref{eq:2.Calkatre}) we have $h_1\in I$; therefore, being $h_1<h_0$, we have $\phe^j(z_0)\in K$ for all $0\le j\le h_1$. But (\ref{eq:2.Calkaqua}) and (\ref{eq:2.Calkatre}) imply that $\phe^j(z_0)\in K$ for $h_1\le j\le h_0$, and thus we again have a contradiction.
\qed
\end{proof}

\begin{corollary}[Abate, \cite{Ab2}]
\label{th:2.qr}
Let $X$ be a taut manifold, and $f\in\Hol(X,X)$ such that the sequence of iterates is not compactly divergent. Then $\Gamma(f)$ is isomorphic to a compact abelian group $\mathbb{Z}_q\times\mathbb{T}^r$, where $\mathbb{Z}_q$ is the cyclic group of order $q$ and $\mathbb{T}^r$ is the real torus of dimension~$r$.
\end{corollary}

\begin{proof}
Let $M$ be the limit manifold of~$f$, and put $\phe=f|_M$. 
By Theorem~\ref{th:2.BedAba}, $\Gamma(f)$ is isomorphic to the closed subgroup~$\Gamma$ of $\Aut(M)$ generated by~$\phe$. We known that $\Aut(M)$ is a Lie group, by Theorem~\ref{th:1.authyp}, and that $\Gamma$ is compact, by Theorem~\ref{th:2.Calka}. Moreover it is abelian, being generated by a single element. It is well known that the compact abelian Lie groups are all of the form $A\times\mathbb{T}^r$, where $A$ is a finite abelian group; to conclude it suffices to notice that $A$ must be cyclic, again because $\Gamma$ is generated by a single element.
\qed
\end{proof}

\begin{definition}
Let $X$ be a taut manifold, and $f\in\Hol(X,X)$ such that the sequence of iterates is not compactly divergent. Then the numbers $q$ and $r$ introduced in the last corollary
are respectively the \emph{limit period}~$q_f$ and the \emph{limit rank}~$r_f$ of $f$.
\end{definition}

When $f$ has a periodic point $z_0\in X$ of period $p\ge 1$, it is possible to explicitly compute the limit dimension, the limit period and the limit rank of $f$ using the eigenvalues of $\D f^p_{z_0}$. To do so we need to introduce two notions.

Let $m\in\mathbb{N}$ and $\Theta=(\theta_1,\ldots,\theta_m)\in[0,1)^m$. Up to a permutation, we can assume that $\theta_1,\ldots,\theta_{\nu_0}\in\mathbb{Q}$ and
$\theta_{\nu_0+1},\ldots,\theta_m\notin\mathbb{Q}$ for some $0\le\nu_0\le m$ (where $\nu_0=0$ means $\Theta\in(\R\setminus\mathbb{Q})^m$ and $\nu_0=m$ means $\Theta\in\mathbb{Q}^m$). 

Let $q_1\in\mathbb{N}^*$ be the least positive integer such that $q_1\theta_1,\ldots,q_1\theta_{\nu_0}\in\mathbb{N}$; if $\nu_0=0$ we put $q_1=1$. 
For $i$,~$j\in\{\nu_0+1,\ldots,m\}$ we shall write $i\sim j$ if and only if $\theta_i-\theta_j\in\mathbb{Q}$. Clearly, $\sim$ is an equivalence relation; furthermore if $i\sim j$
then there is a smallest $q_{ij}\in\mathbb{N}^*$ such that $q_{ij}(\theta_i-\theta_j)\in\mathbb{Z}$. Let $q_2\in\mathbb{N}^*$ be the least common multiple of $\{q_{ij}\mid i\sim j\}$; we put $q_2=1$ if $\nu_0=m$ or $i\not\sim j$ for all pairs $(i,j)$.

\begin{definition}
Let $\Theta=(\theta_1,\ldots,\theta_m)\in[0,1)^m$. Then the \emph{period}~$q(\Theta)\in\mathbb{N}^*$ of~$\Theta$ is the least common multiple of the numbers $q_1$ and $q_2$ introduced above.
\end{definition}

Next, for $j=\nu_0+1,\ldots, m$ write $\theta'_j=q(\Theta)\theta_j-\lfloor q(\Theta)\theta_j\rfloor$, where $\lfloor s\rfloor$ is the integer part of $s\in\R$. Since
\[
\theta'_i=\theta'_j\ \Longleftrightarrow\  q(\Theta)(\theta_i-\theta_j)\in\mathbb{Z}\ \Longleftrightarrow\  i\sim j\;,
\]
the set $\Theta'=\{\theta'_{\nu_0+1},\ldots,\theta'_m\}$ contains as many elements as the number of $\sim$-equivalence classes. If this number is $s$, put $\Theta'=\{\theta''_1,\ldots,\theta''_s\}$. Write $i\approx j$ if and only if $\theta''_i/\theta''_j\in\mathbb{Q}$ (notice that $0\notin\Theta'$); clearly $\approx$ is an equivalence relation.

\begin{definition}
Let $\Theta=(\theta_1,\ldots,\theta_m)\in[0,1)^m$. Then the \emph{rank}~$r(\Theta)\in\mathbb{N}$ is the number of $\approx$-equivalence classes. If $\nu_0=m$ then $r(\Theta)=0$.
\end{definition}

If $X$ is a taut manifold and $f\in\Hol(X,X)$ has a fixed point $z_0\in X$, Theorem~\ref{th:1.CC} says that all the eigenvalues of $\D f_{z_0}$ belongs to $\overline{\Delta}$.
Then we can prove the following:

\begin{theorem}[Abate, \cite{Ab2}]
\label{th:2.limitnumb}
Let $X$ be a taut manifold of dimension $n$, and $f\in\Hol(X,X)$ with a periodic point $z_0\in X$ of period $p\ge 1$. Let $\lambda_1,\ldots,\lambda_n\in\overline{\Delta}$ 
be the eigenvalues of $\D(f^p)_{z_0}$, listed accordingly to their multiplicity and so that 
\[
|\lambda_1|=\cdots=|\lambda_m|=1>|\lambda_{m+1}|\ge\cdots\ge|\lambda_n|
\]
for a suitable $0\le m\le n$. For $j=1,\ldots, m$ write $\lambda_j=\E^{2\pi\I\theta_j}$ with $\theta_j\in[0,1)$, and set $\Theta=(\theta_1,\ldots,\theta_m)$. Then
\[
m_f=m\;,\qquad q_f=p\cdot q(\Theta)\qquad\hbox{and}\qquad r_f=r(\Theta)\;.
\]
\end{theorem}

\begin{proof}
Let us first assume that $z_0$ is a fixed point, that is $p=1$.
Let $M$ be the limit manifold of $f$, and $\rho\in\Hol(X,M)$ its limit retraction. As already remarked, by Theorem~\ref{th:1.CC} the set $\mathrm{sp}(\D f_{z_0})$ of eigenvalues of $\D f_{z_0}$ is contained in~$\overline{\Delta}$; furthermore there is a $\D f_{z_0}$-invariant splitting $T_{z_0}X=L_N\oplus L_U$ satisfying the following properties:
\begin{description}[(c)]
\item[(a)] $\mathrm{sp}(\D f_{z_0}|_{L_N})=\mathrm{sp}(\D f_{z_0})\cap\Delta$ and $\mathrm{sp}(\D f_{z_0}|_{L_U})=\mathrm{sp}(\D f_{z_0})\cap\de\Delta$;
\item[(b)] $(\D f_{z_0}|_{L_N})^k\to O$ as $k\to+\infty$;
\item[(c)] $\D f_{z_0}|_{L_U}$ is diagonalizable.
\end{description}
Fix a subsequence $\{f^{k_\nu}\}$ converging to~$\rho$; in particular, $(\D f_{z_0})^{k_\nu}\to\D\rho_{z_0}$ as $\nu\to+\infty$. Since the only possible eigenvalues of $\D\rho_{z_0}$ are $0$ and $1$, properties (b) and (c) imply that $\D\rho_{z_0}|_{L_N}\equiv O$ and $\D\rho_{z_0}|_{L_U}=\mathrm{id}$. In particular, it follows that
$L_U=T_{z_0}M$ and $m_f=\dim T_{z_0}M=\dim L_U=m$, as claimed.

Set $\phe=f|_M\in\Aut(M)$. By Corollary~\ref{th:1.Cu}, the map $\gamma\mapsto\D\gamma_{z_0}$ is an isomorphism between the group of automorphisms of $M$ fixing~$z_0$ and a subgroup of linear transformations of~$T_{z_0}M$. Therefore, since $\D\phe_{z_0}$ is diagonalizable by (c), $\Gamma(\phe)$, and hence $\Gamma(f)$, is isomorphic to the closed subgroup of~$\mathbb{T}^m$ generated by $\Lambda=(\lambda_1,\ldots,\lambda_m)$. So we have to prove that this latter subgroup is isomorphic to $\mathbb{Z}_{q(\Theta)}\times\mathbb{T}^{r(\Theta)}$. Since we know beforehand the algebraic structure of this group (it is the product of a cyclic group with a torus), it will suffice to write it as a disjoint union of isomorphic tori; the number of tori will be the limit period of~$f$, and the rank of the tori will be the limit rank of~$f$.

Up to a permutation, we can find integers $0\le\nu_0<\nu_1<\cdots<\nu_s=m$ such that $\theta_1,\ldots,\theta_{\nu_0}\in\mathbb{Q}$, and the $\sim$-equivalence classes
are 
\[
\{\theta_{\nu_0+1},\ldots,\theta_{\nu_1}\},\ldots,\{\theta_{\nu_{s-1}+1},\ldots,\theta_m\}\;.
\] 
Then, using the notations introduced for defining $q(\Theta)$ and $r(\Theta)$, we have
\[
\Lambda^{q(\Theta)}=(1,\ldots,1,\E^{2\pi\I\theta''_1},\ldots,\E^{2\pi\I\theta''_1},\E^{2\pi\I\theta''_2},\ldots,\E^{2\pi\I\theta''_2},\ldots, \E^{2\pi\I\theta''_s},
\ldots,\E^{2\pi\I\theta''_s})\;.
\]
This implies that it suffices to show that the subgroup generated by 
\[
\Lambda_1=(\E^{2\pi\I\theta''_1},\ldots,\E^{2\pi\I\theta''_s})
\] 
in~$\mathbb{T}^s$ is isomorphic to $\mathbb{T}^{r(\Theta)}$.

Up to a permutation, we can assume that the $\approx$-equivalence classes are 
\[
\{\theta''_1,\ldots,\theta''_{\mu_1}\},\ldots,\{\theta''_{\mu_{r-1}+1},\ldots,\theta''_s\}\;,
\]
for suitable $1\le\mu_1<\cdots<\mu_r=s$, where $r=r(\Theta)$. Now, by definition of $\approx$ we can find natural numbers $p_j\in\mathbb{N}^*$ for $1\le j\le s$ such that
\[
\displaylines{
\E^{2\pi\I p_1\theta''_1}=\cdots=\E^{2\pi\I p_{\mu_1}\theta''_{\mu_1}}\;,\cr
\vdots\cr
\E^{2\pi\I p_{\mu_{r-1}+1}\theta''_{\mu_{r-1}+1}}=\cdots=\E^{2\pi\I p_s\theta''_s}\;,\cr}
\]
and no other relations of this kind can be found among $\theta''_1,\ldots,\theta''_s$.
It follows that $\{\Lambda_1^k\}_{k\in\mathbb{N}}$ is dense in the subgroup of $\mathbb{T}^s$ defined by the equations
\[
\lambda_1^{p_1}=\cdots=\lambda^{p_{\mu_1}},\ldots,\lambda^{p_{\mu_{t-1}+1}}_{\mu_{r-1}+1}=\cdots=\lambda^{p_s}_s\;,
\] 
which is isomorphic to $\mathbb{T}^r$, as claimed.

Now assume that $z_0$ is periodic of period~$p$, and let $\rho_f$ be the limit retraction of~$f$. Since $\rho_f$ is the unique holomorphic retraction in~$\Gamma(f)$,
and $\Gamma(f^p)\subseteq\Gamma(f)$, it follows that $\rho_f$ is the limit retraction of~$f^p$ too. In particular, the limit manifold of $f$ coincides with the limit manifold of~$f^p$, and hence $m_f=m_{f^p}=m$. Finally, $\Gamma(f)/\Gamma(f^p)\equiv\mathbb{Z}_p$, because $f^j(z_0)\ne z_0$ for $1\le j<p$; hence $\Gamma(f)$ and $\Gamma(f^p)$ have the same connected component at the identity (and hence $r_f=r_{f^p}$), and $q_f=pq_{f^p}$ follows by counting the number of connected components in both groups.
\qed
\end{proof}

If $f\in\Hol(X,X)$ has a periodic point then the sequence of iterates is cleraly not compactly divergent. The converse is in general false, as shown by the following example:

\begin{example}
Let $D\subset\subset\C^2$ be given by 
\[
D=\bigl\{(z,w)\in\C^2\bigm||z|^2+|w|^2+|w|^{-2}<3\bigr\}\;.
\] 
The domain $D$~is strongly pseudoconvex domain, thus taut, but not simply connected. Given $\theta\in\R$ and $\eps=\pm1$, define $f\in\Hol(D,D)$ by
\[
f(z,w)=(z/2,\E^{2\pi\I\theta}w^\eps)\;.
\]
Then the sequence of iterates of $f$ is never compactly divergent, but $f$ has no periodic points as soon as $\theta\notin\mathbb{Q}$. Furthermore,
the limit manifold of~$f$ is the annulus 
\[
M=\bigl\{(0,w)\in\C^2\bigm||w|^2+|w|^{-2}<3\bigr\}\;,
\] 
the limit retraction is $\rho(z,w)=(0,w)$, and suitably choosing~$\eps$ 
and~$\theta$ we can obtain as~$\Gamma(f)$ any compact abelian subgroup 
of~$\Aut(M)$. 
\end{example}

It turns out that self-maps without periodic points but whose sequence of iterates is not compactly divergent can exist only when the topology of the manifold is complicated enough. Indeed, using deep results on the actions of real tori on manifolds, it is possible to prove the following

\begin{theorem}[Abate, \cite{Ab2}]
\label{th:2.topology}
Let $X$ be a taut manifold with finite topological type and such that $H^j(X,\mathbb{Q})=(0)$ for all odd $j$. Take $f\in\Hol(X,X)$. Then the sequence of iterates of $f$ is not compactly divergent if and only if $f$ has a periodic point.
\end{theorem}

When $X=\Delta$ a consequence of the Wolff-Denjoy theorem is that the sequence of iterates of a self-map $f\in\Hol(\Delta,\Delta)$ is not compactly divergent if and only if $f$ has a fixed point, which is an assumption easier to verify than the existence of periodic points.
It turns out that we can generalize this result to convex domains (see also \cite{KS} for a different proof):

\begin{theorem}[Abate, \cite{Ab1}]
\label{th:2.convfix}
Let $D\subset\subset\C^n$ be a bounded convex domain. Take $f\in\Hol(D,D)$. Then the sequence of iterates of $f$ is not compactly divergent if and only if $f$ has a fixed point.
\end{theorem} 

\begin{proof} 
One direction is obvious; conversely, assume that $\{f^k\}$ is not 
compactly divergent, and let $\rho\colon D\to M$ be the limit retraction. 
First of all, note that $k_M=k_D|_{M\times M}$. In fact
\[ 
k_D(z_1,z_2)\le k_M(z_1,z_2)=k_M\bigl(\rho
	(z_1),\rho(z_2)\bigr)\le k_D(z_1,z_2)
\]
for every $z_1$, $z_2\in M$.
In particular, a Kobayashi ball in~$M$ is nothing but the intersection of a 
Kobayashi ball of~$D$ with~$M$.

Let $\phe=f|_M$, and denote by $\Gamma$ the closed subgroup of~$\Aut(M)$ 
generated by~$\phe$; we know, by Corollary~\ref{th:2.qr}, that $\Gamma$~is  
compact. Take $z_0\in M$; then the orbit 
\[
\Gamma(z_0)=\bigl\{\gamma(z_0)\bigm|\gamma\in\Gamma\bigr\}
\]
is compact and contained in~$M$. Let
\[
\mathcal{C}=\Bigl\{\overline{B_D(w,r)}\Bigm|w\in M,\,r>0\ \hbox{and}\ \overline
	{B_D(w,r)}\supset\Gamma(z_0)\Bigr\}\;.
\]
Every $\overline{B_D(w,r)}$ is
compact and convex (by Corollary~\ref{th:1.convexballs}); therefore, $C=\bigcap\mathcal{C}$~is 
a not empty compact convex subset of~$D$. We claim that~$f(C)\subset C$.

Let $z\in C$; we have to show that $f(z)\in\overline{B_D(w,r)}$ for every~$w\in M$
and~$r>0$ such that~$\overline{B_D(w,r)}\supset\Gamma(z_0)$. Now, $\overline{B_D(\phe^{-1}
(w),r)}\in\mathcal{C}$: in fact
\[
\overline{B_D(\phe^{-1}(w),r)}\cap M=\phe^{-1}\bigl(\overline{B_D(w,r)}\cap M\bigr)
        \supset\phe^{-1}\bigl(\Gamma(z_0)\bigr)=\Gamma(z_0)\;.
\]
Therefore $z\in\overline{B_D(\phe^{-1}(w),r)}$ and
\[
k_D\bigl(w,f(z)\bigr)=k_D\Bigl(f\bigl(\phe^{-1}(w)\bigr),f(z)\Bigr)\le
        k_D\bigl(\phe^{-1}(w),z\bigr)\le r\;,
\]
that is $f(z)\in\overline{B_D(w,r)}$, as we want.

In conclusion, $f(C)\subset C$; by Brouwer's theorem, $f$~must have a fixed 
point in~$C$. 
\qed
\end{proof}

The topology of convex domains is particularly simple: indeed, convex domains are topologically contractible, that is they have a point as (continuous) retract of deformation. Using very deep properties of the Kobayashi distance in strongly pseudoconvex domains, outside of the scope of these notes, Huang has been able to generalize Theorem~\ref{th:2.convfix} to topologically contractible strongly pseudoconvex domains:

\begin{theorem}[Huang, \cite{Huang}]
\label{th:2.Huang}
Let $D\subset\subset\C^n$ be a bounded topologically contractible strongly pseudoconvex $C^3$ domain. Take $f\in\Hol(D,D)$. Then the sequence of iterates of $f$ is not compactly divergent if and only if $f$ has a fixed point.
\end{theorem} 

This might suggest that such a statement might be extended to taut manifolds (or at least to taut domains) topologically contractible. Surprisingly, this is not true: 

\begin{theorem}[Abate-Heinzner, \cite{AH}]
\label{th:2.AH}
There exists a bounded domain $D\subset\subset\C^8$ which is taut, homeomorphic to~$\C^8$ (and hence topologically contractible), pseudoconvex, and strongly pseudoconvex at all points of~$\de D$ but one, where a finite cyclic group acts without fixed points.
\end{theorem}

This completes the discussion of tasks (a) and (b). In the next two subsections we shall describe how it is possible to use the Kobayashi distance to deal with task~(c).

\subsection{Horospheres and the Wolff-Denjoy theorem}
\label{subsec:2.2}

When $f\in\Hol(\Delta,\Delta)$ has a fixed point $\zeta_0\in\Delta$, the Wolff-Denjoy theorem is an easy consequence of the Schwarz-Pick lemma. Indeed if $f$ is an automorphism the statement is clear; if it is not an automorphism, then $f$ is a strict contraction of any Kobayashi ball centered at~$\zeta_0$, and thus the orbits must converge to the fixed point~$\zeta_0$. When $f$ has no fixed points, this argument fails because there are no $f$-invariant Kobayashi balls. Wolff had the clever idea of replacing Kobayashi balls by a sort of balls ``centered" at points in the boundary, the \emph{horocycles,} and he was able to prove the existence of $f$-invariant horocycles---and thus to complete the proof of the Wolff-Denjoy theorem. 

This is the approach we shall follow to prove a several variable version of the Wolff-Denjoy theorem in strongly pseudoconvex domains, using the Kobayashi distance to define a general notion of multidimensional analogue of the horocycles, the horospheres. This notion, introduced in \cite{Ab1}, is behind practically all known generalizations of the Wolff-Denjoy theorem; and it has found other applications as well (see, e.g., the survey paper \cite{AbLind} and other chapters in this book).

\begin{definition}
Let $D\subset\subset\C^n$ be a bounded domain. Then the \emph{small horosphere} of \emph{center}~$x_0\in\de D$, \emph{radius}~$R>0$ and \emph{pole}~$z_0\in D$ is the set
\[
E_{z_0}(x_0,R)=\bigl\{z\in D\bigm|\limsup_{w\to x_0}[k_D(z,w)-k_D(z_0,w)]<\mlog R
\bigr\}\;;
\] 
the \emph{large horosphere} of \emph{center}~$x_0\in\de D$, \emph{radius}~$R>0$ and \emph{pole}~$z_0\in D$ is the set
\[
E_{z_0}(x_0,R)=\bigl\{z\in D\bigm|\liminf_{w\to x_0}[k_D(z,w)-k_D(z_0,w)]<\mlog R
\bigr\}\;.
\] 
\end{definition}
The rationale behind this definition is the following. A Kobayashi ball of center $w\in D$ and radius~$r$ is the set of $z\in D$ such that $k_D(z,w)<r$. If we let $w$ go to a point in the boundary $k_D(z,w)$ goes to infinity (at least when $D$ is complete hyperbolic), and so we cannot use it to define subsets of~$D$. We then renormalize $k_D(z,w)$ by subtracting the distance $k_D(z_0,w)$ from a reference point~$z_0$. By the triangular inequality the difference $k_D(z,w)-k_D(z_0,w)$ is bounded by $k_D(z_0,z)$; thus we can consider the liminf and the limsup as $w$ goes to~$x_0\in\de D$ (in general, the limit does not exist; an exception is given by strongly convex $C^3$ domains, see \cite[Corollary~2.6.48]{Abatebook}), and the sublevels provide some sort of balls centered at points in the boundary.

The following lemma contains a few elementary properties of the horospheres, which are an immediate consequence of the definition (see, e.g., \cite[Lemmas~2.4.10 and~2.4.11]{Abatebook}):

\begin{lemma}
\label{th:2.horoeasy}
Let $D\subset\subset\C^n$~be a bounded domain of~$\C^n$, and choose $z_0\in D$ and
$x\in\partial D$. Then:
\begin{enumerate}
\item[\rm(i)] for every $R>0$ we have $E_{z_0}(x, R)\subset F_{z_0}(x,R)$;
\item[\rm(ii)] for every $0<R_1<R_2$ we have $E_{z_0}(x, R_1)\subset E_{z_0}(x, R_2)$
and $F_{z_0}(x, R_1)\subset F_{z_0}(x,R_2)$;
\item[\rm(iii)] for every $R>1$ we have $B_D(z_0,\mlog R)\subset E_{z_0}(x,R)$;
\item[\rm(iv)] for every $R<1$ we have $F_{z_0}(x,R) \cap B_D(z_0,-\mlog R)=\emptyset$;
\item[\rm(v)] $\bigcup\limits_{R>0}E_{z_0}(x,R) =\bigcup\limits_{R>0}F_{z_0}(x,R)=D$
and $\bigcap\limits_{R>0}E_{z_0}(x,R)=\bigcap\limits_{R>0} F_{z_0}(x,R)=\emptyset$;
\item[\rm(vi)] if $\gamma\in\Aut(D)\cap C^0(\overline D,\overline D)$, then for every $R>0$
\[
\phe\bigl(E_{z_0}(x,R)\bigr)=E_{\phe(z_0)}\bigl(\phe(x),R\bigr)\qquad\hbox{
and}\qquad \phe\bigl(F_{z_0}(x,R)\bigr)= F_{\phe(z_0)}\bigl(\phe(x),R\bigr)\;;
\]
\item[\rm(vii)] if $z_1 \in D$, set
\[
\mlog L = \limsup_{w\to x}\bigl[k_D(z_1,w)-k_D(z_0,w)\bigr]\;.
\]
Then for every $R>0$ we have $E_{z_1}(x,R)\subseteq E_{z_0}(x,LR)$ and
$F_{z_1}(x,R)\subseteq F_{z_0}(x,LR)$.
\end{enumerate}
\end{lemma}

It is also easy to check that the horospheres with pole at the origin in~$B^n$ (and thus in~$\Delta$) coincide with the classical horospheres:

\begin{lemma}
\label{th:2.horoball}
If $x\in\de B^n$ and $R>0$ then
\[
E_O(x,R)=F_O(x,r)=\left\{z\in B^n\biggm| \frac{|1-\langle z,x\rangle|^2}{1-\|z\|^2}<R\right\}\;.
\]
\end{lemma}

\begin{proof}
If $z\in B^n\setminus\{O\}$, let $\gamma_z\colon B^n\to\C^n$ be given by
\begin{equation}
\gamma_z(w)=\frac{z-P_z(w)-(1-\|z\|^2)^{1/2}\bigl(w-P_z(w)\bigr)}{1-\langle w,z\rangle}\;,
\label{eq:2.autBn}
\end{equation}
where $P_z(w)=\frac{\langle w,z\rangle}{\langle z,z\rangle}z$ is the orthogonal projection on~$\C z$;
we shall also put $\gamma_O=\mathrm{id}_{B^n}$. It is easy to check that $\gamma_z(z)=O$, that $\gamma_z(B^n)\subseteq B^n$ and that $\gamma_z\circ\gamma_z=\mathrm{id}_{B^n}$; in particular, $\gamma_z\in\Aut(B^n)$. Furthermore,
\[
1-\|\gamma_z(w)\|^2=\frac{(1-\|z\|^2)(1-\|w\|^2)}{|1-\langle w,z\rangle|^2}\;.
\]
Therefore for all $w\in B^n$ we get
\begin{eqnarray*}
k_{B^n}(z,w)-k_{B^n}(O,w)&=&k_{B^n}\bigl(O,\gamma_z(w)\bigr)-k_{B^n}
	(O,w)\\
	&=&\mlog{\biggl(\frac{1+\|\gamma_z(w)\|}{1+\|w\|}\cdot\frac{1-\|w\|}
	{1-\|\gamma_z(w)\|}\biggr)}\\
	&=&\log\frac{1+\|\gamma_z(w)\|}{1+\|w\|}+\mlog\frac{|1-\langle w,z\rangle|^2}{1-\|z\|^2}\;.
\end{eqnarray*}
Letting $w\to x$ we get the assertion, because $\|\gamma_z(x)\|=1$. 
\qed
\end{proof}

Thus in $B^n$ small and large horospheres coincide. Furthermore, the horospheres with pole at the origin are ellipsoids tangent to~$\de B^n$ in~$x$, because an easy computation yields
\[
E_O(x,R)=\biggl\{z\in B^n \biggm| \frac{\|P_x(z)-(1-r)x\|^2}{r^2}+\frac{\|z-P_x(z)\|^2}{r}<1
\biggr\}\;,
\]
where $r=R/(1+R)$. In particular if $\tau\in\de\Delta$ we have
\[
E_0(\tau,R)=\bigl\{\zeta\in\Delta\bigm| |\zeta-(1-r)\tau|^2<r^2\bigr\}\;,
\]
and so a horocycle is an Euclidean disk internally tangent to~$\de\Delta$ in~$\tau$.

Another domain where we can explicitly compute the horospheres is the polydisk; in this case large and small horospheres are actually different (see, e.g., \cite[Proposition~2.4.12]{Abatebook}):

\begin{proposition}
\label{th:2.horopolydisk}
Let $x\in\partial\Delta^n$ and $R>0$. Then
\[
\vcenter{\halign{$\hfil#$&$#=\Biggl\{z\in\Delta^n\hfil$&${}\Biggm|\hfil#$&
        $\biggl\{{\displaystyle\frac{|x_j-z_j|^2}{1-|z_j|^2}}\biggm|
        |x_j|=1\biggr\}<R\Biggr\}$#\hfil\cr
        E_O(x,R)&{}&\max\limits_j&\thinspace;\cr
\noalign{\vskip2\jot}
        F_O(x,R)&{}&\min\limits_j&\thinspace.\cr}}
\]
\end{proposition}

The key in the proof of the classical Wolff-Denjoy theorem is the 

\begin{theorem}[Wolff's lemma, \cite{Wo}]
\label{th:2.Wolff}
Let $f\in\Hol(\Delta,\Delta)$ without fixed points. Then there exists a unique $\tau\in\de\Delta$ such that
\begin{equation}
f\bigl(E_0(\tau,R)\bigr)\subseteq E_0(\tau,R)
\label{eq:2.unocinque}
\end{equation}
for all $R>0$.
\end{theorem}

\begin{proof}
For the uniqueness, assume that (\ref{eq:2.unocinque}) holds for two distinct points 
$\tau$,~$\tau_1\in\partial\Delta$. Then we can construct two horocycles, 
one centered at~$\tau$ and the other centered at~$\tau_1$, tangent to 
each other at a point of~$\Delta$. By (\ref{eq:2.unocinque}) this point would be a 
fixed point of~$f$, contradiction.

For the existence, pick a sequence $\{r_\nu\}\subset(0,1)$ with $r_\nu
\to1$, and set $f_\nu=r_\nu f$. Then $f_\nu(\Delta)$ is relatively compact 
in~$\Delta$; by Brouwer's theorem each~$f_\nu$ has a fixed 
point~$\eta_\nu\in\Delta$. Up to a subsequence, we can assume $\eta_\nu\to \tau\in
\overline{\Delta}$. If $\tau$ were in~$\Delta$, we would have
\[
f(\tau)=\lim_{\nu\to\infty}f_\nu(\eta_\nu)=\lim_{\nu\to\infty}\eta_\nu=\tau\;,
\]
which is impossible; therefore $\tau\in\partial\Delta$.

Now, by the Schwarz-Pick lemma we have $k_\Delta\bigl(f_\nu(\zeta),\eta_\nu\bigr)\le k_\Delta(\zeta,\eta_\nu)$ for all $\zeta\in\Delta$; recalling the formula for the Poincar\'e distance we get
\[
1-\biggl|\frac{f_\nu(\zeta)-\eta_\nu}{1-\overline{\eta_\nu}f_\nu(\zeta)}\biggr|^2\ge1-\biggl|
	\frac{\zeta-\eta_\nu}{1-\overline{\eta_\nu}\zeta}\biggr|^2\;,
\]
or, equivalently,
\[
\frac{|1-\overline{\eta_\nu}f_\nu(\zeta)|^2}{1-|f_\nu(\zeta)|^2}\le\frac{|1-\overline{\eta_\nu}\zeta|^2}{
	1-|\zeta|^2}\;.
\]
Taking the limit as $\nu\to\infty$ we get
\[
\frac{|1-\overline{\tau}f(\zeta)|^2}{1-|f(\zeta)|^2}\le\frac{|1-\overline{\tau}\zeta|^2}{
	1-|\zeta|^2}\;,
\]
and the assertion follows.
\qed
\end{proof}

With this result it is easy to conclude the proof of the Wolff-Denjoy theorem. Indeed, if $f\in\Hol(\Delta,\Delta)$ has no fixed points we already know that the sequence of iterates is compactly divergent, which means that the image of any limit $h$ of a converging subsequence is contained in~$\de\Delta$. By the maximum principle, the map~$h$ must be constant; and by Wolff's lemma this constant must be contained in~$\overline{E_0(\tau,R)}\cap\de\Delta=\{\tau\}$. So every converging subsequence of $\{f^k\}$ must converge to the constant~$\tau$; and this is equivalent to saying that the whole sequence of iterates converges to the constant map~$\tau$.

\begin{remark}
Let me make more explicit the final argument used here, because we are going to use it often.
Let $D\subset\subset\C^n$ be a bounded domain; in particular, it is (hyperbolic and) relatively compact inside an Euclidean ball $B$, which is complete hyperbolic and hence taut. Take now $f\in\Hol(D,D)$. Since $\Hol(D,D)\subset\Hol(D,B)$, the sequence of iterates $\{f^k\}$ is normal 
in $\Hol(D,B)$; but since $D$ is relatively compact in~$B$, it cannot contain subsequences compactly divergent in~$B$. Therefore $\{f^k\}$ is relatively compact in $\Hol(D,B)$; and since the latter is a metrizable topological space, to prove that $\{f^k\}$ converges in $\Hol(D,B)$ it suffices to prove that all converging subsequences of $\{f^k\}$ converge to the same limit (whose image
will be contained in $\overline{D}$, clearly).
\end{remark}

The proof of the Wolff-Denjoy theorem we described is based on two ingredients: the existence of a $f$-invariant horocycle, and the fact that a horocycle touches the boundary in exactly one point. 
To generalize this argument to several variables we need an analogous of Theorem~\ref{th:2.Wolff} for our multidimensional horopsheres, and then we need to know how the horospheres touch the boundary.

There exist several multidimensional versions of Wolff's lemma; we shall present three of them (Theorems~\ref{th:2.WolffAba}, \ref{th:2.Wolffconv} and \ref{th:2.Wolffdue}). To state the first one we need a definition.

\begin{definition}
Let $D\subset\C^n$ be a domain in~$\C^n$. We say that $D$ has \emph{simple boundary} if every $\phe\in\Hol(\Delta,\C^n)$ such that $\phe(\Delta)\subseteq\overline{D}$ and $\phe(\Delta)\cap\de D\ne\emptyset$ is constant.
\end{definition}

\begin{remark}
\label{rem:2.sb}
It is easy to prove (see, e.g., \cite[Proposition~2.1.4]{Abatebook}) that if $D$ has simple boundary and $Y$ is any complex manifold then every $f\in\Hol(Y,\C^n)$ such that $f(Y)\subseteq\overline{D}$ and $f(Y)\cap\de D\ne\emptyset$ is constant.
\end{remark}

\begin{remark}
\label{rem:2.simpbound}
By the maximum principle, every domain $D\subset\C^n$ admitting a peak function at each point of its boundary is simple. For instance, strongly pseudoconvex domain (Theorem~\ref{th:1.Graham}) and (not necessarily smooth) strictly convex domains (Remark~\ref{rem:1.suppf}) have simple boundary.
\end{remark}

Then we are able to prove the following 

\begin{theorem}[Abate, \cite{Ab2}]
\label{th:2.WolffAba}
Let $D\subset\subset\C^n$ be a complete hyperbolic bounded domain with simple boundary, and take $f\in\Hol(D,D)$ with compactly divergent sequence of iterates. Fix $z_0\in D$. Then there exists $x_0\in\de D$ such that
\[
f^p\bigl(E_{z_0}(x_0,R)\bigr)\subseteq F_{z_0}(x_0,R)
\]
for all $p\in\mathbb{N}$ and $R>0$.
\end{theorem}

\begin{proof}
Since $D$ is complete hyperbolic and $\{f^k\}$ is compactly divergent, we have $k_D\bigl(z_0,f^k(z_0)\bigr)\to+\infty$ as $k\to+\infty$. Given $\nu\in\mathbb{N}$, let $k_\nu$ be the largest~$k$ such that 
$k_D\bigl(z_0,f^k(z_0)\bigr)\le\nu$. In particular for every $p>0$ we have
\begin{equation}
k_D\bigl(z_0,f^{k_\nu}(z_0)\bigr)\le \nu<k_D\bigl(z_0,f^{k_\nu+p}(z_0)\bigr)\;.
\label{eq:2.W1}
\end{equation}
Since $D$ is bounded, up to a subsequence we can assume that $\{f^{k_\nu}\}$ converges
to a holomorphic $h\in\Hol(D,\C^n)$. But $\{f^k\}$ is compactly divergent; therefore $h(D)\subset\de D$ and so $h\equiv x_0\in\de D$, because $D$ has simple boundary (see Remark~\ref{rem:2.sb}).

Put $w_\nu=f^{k_\nu}(z_0)$. We have $w_\nu\to x_0$; as a consequence for every $p>0$ we have
$f^p(w_\nu)=f^{k_\nu}\bigl(f^p(z_0)\bigr)\to x_0$ and
\[
\limsup_{\nu\to+\infty}\bigl[k_D(z_0,w_\nu)-k_D\bigl(z_0,f^p(w_\nu)\bigr)\bigr]\le 0
\]
by (\ref{eq:2.W1}). Take $z\in E_{z_0}(x_0,R)$; then we have
\begin{eqnarray*}
\liminf_{w\to x_0}\bigl[k_D\bigl(f^p(z),w\bigr)-k_D(z_0,w)\bigr]&\le& 
\liminf_{\nu\to +\infty}\bigl[k_D\bigl(f^p(z),f^p(w_\nu)\bigr)-k_D\bigl(z_0,f^p(w_\nu)\bigr)\bigr]\\
&\le&\liminf_{\nu\to+\infty}\bigl[k_D(z,w_\nu)-k_D\bigl(z_0,f^p(w_\nu)\bigr)\bigr]\\
&\le&\limsup_{\nu\to+\infty}\bigl[k_D(z,w_\nu)-k_D(z_0,w_\nu)\bigr]\\
&&\quad+\limsup_{\nu\to+\infty}\bigl[k_D(z_0,w_\nu)-k_D\bigl(z_0,f^p(w_\nu)\bigr)\bigr]\\
&\le&\limsup_{\nu\to+\infty}\bigl[k_D(z,w_\nu)-k_D(z_0,w_\nu)\bigr]<\mlog R\;,
\end{eqnarray*}
that is $f^p(z)\in F_{z_0}(x_0,R)$, and we are done.
\qed
\end{proof}

The next step consists in determining how the large horospheres touch the boundary. 
The main tools here are the boundary estimates proved in Subsection~\ref{subsect:1.5}:

\begin{theorem}[Abate, \cite{Ab1}]
\label{th:2.horo}
Let $D\subset\subset\C^n$ be a bounded strongly pseudoconvex domain. Then
\[
\overline{E_{z_0}(x_0,R)}\cap\de D=\overline{F_{z_0}(x_0,R)}\cap\de D=\{x_0\}
\]
for every $z_0\in D$, $x_0\in\de D$ and $R>0$.
\end{theorem}

\begin{proof} 
We begin by proving that $x_0$~belongs to the closure of~$E_{z_0}(x_0,R)$. Let 
$\eps>0$ be given by Theorem~\ref{th:1.bestib}; then, recalling Theorem~\ref{th:1.bestil}, for 
every $z$,~$w\in D$ with $\|z-x_0\|$, $\|w-x_0\|<\eps$ we have
\[
k_D(z,w)-k_D(z_0,w)\le\mlog\biggl(1+\frac{\|z-w\|}{d(z,\de D)}\biggr)
	+\mlog\bigl[d(w,\de D)+\|z-w\|\bigr]+K\;,
\]
for a suitable constant $K\in\R$ depending only on~$x_0$ and~$z_0$. In 
particular, as soon as $\|z-x\|<\eps$ we get
\begin{equation}
\limsup_{w\to x}[k_D(z,w)-k_D(z_0,w)]\le\mlog\biggl(1+\frac{\|z-x\|}{ d(z,
	\de D)}\biggr)+\mlog\|z-x\|+K\;.
\label{eq:2.ntE}
\end{equation}
So if we take a sequence $\{z_\nu\}\subset D$ converging to~$x_0$ so that
$\{\|z_\nu-x_0\|/d(z_\nu,\de D)\}$ is bounded (for instance, a sequence 
converging non-tangentially to~$x_0$), then for every~$R>0$ we have $z_\nu
\in E_{z_0}(x_0,R)$ eventually, and thus $x_0\in\overline{E_{z_0}(x_0,R)}$. 

To conclude the proof, we have to show that $x_0$~is the only boundary point
belonging to the closure of~$F_{z_0}(x_0,R)$. Suppose, by contradiction, that there
exists $y\in\partial D\cap\overline{F_{z_0}(x_0,R)}$ with~$y\ne x_0$; then we can find a
sequence~$\{z_\mu\}\subset F_{z_0}(x_0,R)$ with~$z_\mu\to y$.

Theorem~\ref{th:1.bestia} provides us with~$\eps>0$ and~$K\in\R$ associated to the
pair~$(x_0,y)$; we may assume $\|z_\mu-y\|<\eps$ for all $\mu\in\mathbb{N}$. 
Since~$z_\mu\in F_{z_0}(x_0,R)$, we have
\[
\liminf_{w\to x}\bigl[k_D(z_\mu,w)-
        k_D(z_0,w)\bigr]<\mlog R
\]
for every $\mu\in\mathbb{N}$;
therefore for each~$\mu\in\mathbb{N}$ we can find a sequence~$\{w_{\mu\nu}\}\subset
D$ such that~$\lim\limits_{\nu\to\infty}w_{\mu\nu}=x_0$ and
\[
\lim_{\nu\to\infty}\bigl[k_D(z_\mu,w_{\mu\nu})-k_D(z_0,w_{\mu\nu})\bigr]<
        \mlog R\;.
\]
Moreover, we can assume $\|w_{\mu\nu}-x\|<\eps$ and~$k_D(z_\mu,w_{\mu\nu})-
k_D(z_0,w_{\mu\nu})<\mlog R$ for all~$\mu$,~$\nu\in\mathbb{N}$.

By Theorem~\ref{th:1.bestia} for all~$\mu$,~$\nu\in\mathbb{N}$ we have
\begin{eqnarray*}
\mlog R&>&k_D(z_\mu,w_{\mu\nu})- k_D(z_0,w_{\mu\nu})\\
        &\ge&-\mlog d(z_\mu,\partial D)-\mlog d(w_{\mu\nu},\partial D)-
        k_D(z_0,w_{\mu\nu})-K\;.
\end{eqnarray*}
On the other hand, Theorem~\ref{th:1.bestiu} yields~$c_1>0$ 
(independent of~$w_{\mu\nu}$) such that
\[
k_D(z_0,w_{\mu\nu})\le c_1-\mlog
        d(w_{\mu\nu},\partial D)
\]
for every $\mu$, $\nu\in\mathbb{N}$.
Therefore
\[
\mlog R>-\mlog d(z_\mu,\partial D)-K-c_1
\]
for every $\mu\in\mathbb{N}$,
and, letting $\mu$ go to infinity, we get a contradiction. 
\qed
\end{proof}

We are then able to prove a Wolff-Denjoy theorem for strongly pseudoconvex domains:

\begin{theorem}[Abate, \cite{Ab2}]
\label{th:2.WDA}
Let $D\subset\subset\C^n$ be a strongly pseudoconvex $C^2$ domain. Take $f\in\Hol(D,D)$ with compactly divergent sequence of iterates. Then $\{f^k\}$ converges to a constant map $x_0\in\de D$.
\end{theorem}

\begin{proof}
Fix $z_0\in D$, and
let $x_0\in\de D$ be given by Theorem~\ref{th:2.WolffAba}.
Since $D$ is bounded, it suffices to prove that every subsequence of $\{f^k\}$ converging in $\Hol(D,\C^n)$ actually converges to the constant map~$x_0$.

Let $h\in\Hol(D,\C^n)$ be the limit of a subsequence of iterates. Since $\{f^k\}$ is compactly divergent, we must have $h(D)\subset\de D$. Hence
Theorem~\ref{th:2.WolffAba} implies that
\[
h\bigl(E_{z_0}(x_0,R)\bigr)\subseteq \overline{F_{z_0}(x_0,R)}\cap\de D
\]
for any $R>0$; since (Theorem~\ref{th:2.horo}) $\overline{F_{z_0}(x_0,R)}\cap\de D=\{x_0\}$ 
we get $h\equiv x_0$, and we are done.
\qed
\end{proof}

\begin{remark}
The proof of Theorem~\ref{th:2.WDA} shows that we can get such a statement in any complete hyperbolic domain with simple boundary satisfying Theorem~\ref{th:2.horo}; and the proof of the latter theorem shows that what is actually needed are suitable estimates on the boundary behavior of the Kobayashi distance. Using this remark, it is possible to extend Theorem~\ref{th:2.WDA} to some classes of weakly pseudoconvex domains; see, e.g., Ren-Zhang \cite{RZ} and Khanh-Thu \cite{KT}.
\end{remark}
 
 \subsection{Strictly convex domains}
 \label{subsec:2.3}
 
The proof of Theorem~\ref{th:2.WDA} described in the previous subsection depends in an essential way on the fact that the boundary of the domain ~$D$ is of class at least~$C^2$. Recently, Budzy\'nska \cite{Bud} (see also \cite{BKR}) found a way to prove Theorem~\ref{th:2.WDA} in \emph{strictly convex} domains without any assumption on the smoothness of the boundary; in this subsection we shall describe a simplified approach due to Abate and Raissy \cite{AR}. 

The result which is going to replace Theorem~\ref{th:2.horo} is the following:

\begin{proposition}
\label{th:2.uuno} 
Let $D\subset\C^n$ be a hyperbolic convex domain, $z_0\in D$, $R>0$
and $x\in\de D$. Then we have $[x,z]\subset \overline{F_{z_0}(x,R)}$ for 
all $z\in\overline{F_{z_0}(x,R)}$. Furthermore, 
\begin{equation}
x\in\bigcap_{R>0} \overline{F_{z_0}(x,R)}\subseteq\hbox{\rm ch}(x)\;.
\label{eq:uinter}
\end{equation}
In particular, if $x$ is a strictly convex point then $\bigcap\limits_{R>0} \overline{F_{z_0}(x,R)}=\{x\}$.
\end{proposition}

\begin{proof}
Given $z\in F_{z_0}(x,R)$, choose a sequence $\{w_\nu\}\subset D$ converging to~$x$
and such that the limit of $k_D(z,w_\nu)-k_D(z_0,w_\nu)$ exists and is less than~$\mlog R$. 
Given $0<s<1$, let $h_\nu^s\colon D\to D$ be defined by
\[
h_\nu^s(w)=sw+(1-s)w_\nu
\]
for every $w\in D$;
then $h_\nu^s(w_\nu)=w_\nu$.
In particular,
\[
\limsup_{\nu\to+\infty}\bigl[k_D\bigl(h_\nu^s(z),w_\nu)-k_D(z_0,w_\nu)\bigr]
\le \lim_{\nu\to+\infty}\bigl[k_D(z,w_\nu)-k_D(z_0,w_\nu)\bigr]<\mlog R\;.
\]
Furthermore we have
\[
\bigl|k_D\bigl(sz+(1-s)x,w_\nu\bigr)-k_D\bigl(h^s_\nu(z),w_\nu\bigr)\bigr|
\le k_D\bigl(sz+(1-s)w_\nu,sz+(1-s)x\bigr)\to 0
\]
as $\nu\to+\infty$. Therefore
\begin{eqnarray*}
\liminf_{w\to x}\bigl[k_D&&\!\!\!\!\bigl(sz+(1-s)x,w\bigr)-k_D(z_0,w)\bigr]\\
&&\le\limsup_{\nu\to+\infty}\bigl[k_D\bigl(sz+(1-s)x,w_\nu\bigr)-k_D(z_0,w_\nu)\bigr]\\
&&\le
\limsup_{\nu\to+\infty}\bigl[k_D\bigl(h_\nu^s(z),w_\nu\bigr)-k_D(z_0,w_\nu)\bigr]\\
&&\qquad+\lim_{\nu\to+\infty}\bigl[k_D\bigl(sz+(1-s)x,w_\nu\bigr)-k_D\bigl(h_\nu^s(z),w_\nu\bigr)\bigr]\\
&&<\mlog R\;,
\end{eqnarray*}
and thus $sz+(1-s)x\in F_{z_0}(x,R)$. Letting $s\to 0$ we also get $x\in\overline{F_{z_0}(x,R)}$,
and we have proved the first assertion for $z\in F_{z_0}(x,R)$. If $z\in\de F_{z_0}(x,R)$,
it suffices to apply what we have just proved to a sequence in~$F_{z_0}(x,R)$ approaching~$z$.

In particular we have thus shown that $x\in\bigcap_{R>0} \overline{F_{z_0}(x,R)}$.
Moreover this intersection is contained in~$\de D$, by Lemma~\ref{th:2.horoeasy}.
Take $y\in\bigcap_{R>0} \overline{F_{z_0}(x,R)}$ different from~$x$. Then
the whole
segment $[x,y]$ must be contained in the intersection, and thus in~$\de D$; hence $y\in\hbox{\rm ch}(x)$,
and we are done.
\qed
\end{proof}

We can now prove a Wolff-Denjoy theorem in strictly convex domains without any assumption on the regularity of the boundary:

\begin{theorem}[Budzy\'nska, \cite{Bud}; Abate-Raissy, \cite{AR}]
\label{th:2.WDAR}
Let $D\subset\subset\C^n$ be a bounded strictly convex domain, and take $f\in\Hol(D,D)$ without fixed points. Then the sequence of iterates $\{f^k\}$ converges to a constant map $x\in\de D$.
\end{theorem}

\begin{proof}
Fix $z_0\in D$, and let $x\in\de D$ be given by Theorem~\ref{th:2.WolffAba}, that can be applied because strictly convex domains are complete hyperbolic (by Proposition~\ref{th:1.Barth}) and have simple boundary (by Remark~\ref{rem:2.simpbound}). So, since $D$ is bounded, it suffices to prove that every converging subsequence of $\{f^k\}$ converges to the constant map~$x$.

Assume that $\{f^{k_\nu}\}$ converges to a holomorphic map $h\in\Hol(D,\C^n)$. Clearly, $h(D)\subset\overline{D}$; since the sequence of iterates is compactly divergent (Theorem~\ref{th:2.convfix}), we have $h(D)\subset\de D$; since $D$ has simple boundary, it follows that $h\equiv y\in\de D$. So we have to prove that $y=x$.

Take $R>0$, and choose $z\in E_{z_0}(x,R)$. Then Theorem~\ref{th:2.WolffAba} yields
$y=h(z)\in\overline{F_{z_0}(x,R)}\cap\de D$. Since this holds for all $R>0$ we get
$y\in\bigcap_{R>0}\overline{F_{z_0}(x,R)}$, and Proposition~\ref{th:2.uuno} yields the assertion.
\qed
\end{proof}

\subsection{Weakly convex domains}
\label{subsec:2.4}

The approach leading to Theorem~\ref{th:2.WDAR} actually yields results for weakly convex domains too, even though we cannot expect in general the convergence to a constant map.

\begin{example}
Let $f\in\Hol(\Delta^2,\Delta^2)$ be given by
\[
f(z,w)=\left(\frac{z+1/2}{1+z/2},w\right)\;.
\]
Then it is easy to check that the sequence of iterates of $f$ converges to the non-constant map $h(z,w)=(1,w)$.
\end{example}

The first observation is that we have a version of Theorem~\ref{th:2.WolffAba} valid in all convex domains, without the requirement of simple boundary:

\begin{theorem}[\cite{Ab1}]
\label{th:2.Wolffconv}
Let $D\subset\subset\C^n$ be a bounded convex
domain, and take a map $f\in\Hol(D,D)$ without fixed points. Then there 
exists~$x\in\partial D$ such that 
\[
f^k\bigl(E_{z_0}(x,R)\bigr)\subset F_{z_0}(x,R)
\]
for every~$z_0\in D$,~$R>0$ and~$k\in\mathbb{N}$.
\end{theorem}

\begin{proof} 
Without loss of generality we can assume that $O\in D$. For $\nu>0$ let $f_\nu\in\Hol(D,D)$ be given by
\[
f_\nu(z)=\left(1-\frac{1}{\nu}\right)f(z)\;;
\]
then $f_\nu(D)$ is relatively compact in~$D$ and $f_\nu\to f$ as $\nu\to+\infty$.
By Brouwer's theorem, every~$f_\nu$ has a fixed
point~$w_\nu\in D$. Up to a subsequence, we may assume that
$\{w_\nu\}$~converges to a point~$x\in\overline{D}$. If~$x\in D$, then
\[
f(x)=\lim_{\nu\to\infty}f_\nu(w_\nu)=\lim_{\nu\to\infty}w_\nu=x\;,
\]
impossible; therefore~$x\in\partial D$.

Now fix $z\in E_{z_0}(x,R)$ and $k\in\mathbb{N}$. We have
\[
\bigl|k_D\bigl(f_\nu^k(z),w_\nu\bigr)-k_D\bigl(f^k(z),w_\nu\bigr)\bigr|\le
        k_D\bigl(f_\nu^k(z),f^k(z)\bigr)\longrightarrow 0
\]
as~$\nu\to+\infty$.
Since $w_\nu$ is a fixed point of~$f_\nu^k$ for every~$k\in\mathbb{N}$, we then get
\begin{eqnarray*}
\liminf_{w\to x}\bigl[k_D(f^k(z),w)-k_D(z_0,w)\bigr]&\le&\liminf_{\nu\to+\infty}
        \bigl[k_D(f^k(z),w_\nu)-k_D(z_0,w_\nu)\bigr]\\
        &\le&\limsup_{\nu\to+\infty}\bigl[k_D\bigl(f^k_\nu(z),w_\nu\bigr)-k_D(z_0,w_\nu)\bigr]\\
        &&\quad+\lim_{\nu\to+\infty}\bigl[k_D\bigl(f^k(z),w_\nu\bigr)-k_D\bigl(f^k_\nu(z),w_\nu\bigr)\bigr]\\
        &\le&\limsup_{\nu\to+\infty}\bigl[k_D\bigl(z,w_\nu\bigr)-k_D(z_0,w_\nu)\bigr]\\
        &\le&\limsup_{w\to x}\bigl[k_D\bigl(z,w\bigr)-k_D(z_0,w)\bigr]<\mlog R\;,
\end{eqnarray*}
and $f^k(z)\in F_{z_0}(x,R)$.
\qed
\end{proof}

When $D$ has $C^2$ boundary this is enough to get a sensible Wolff-Denjoy theorem, because of the following result:

\begin{proposition}[\cite{AR}]
\label{th:2.aggiunta}
Let $D\subset\subset\C^n$ be a bounded convex domain with $C^2$ boundary,
and $x\in\de D$. Then for every $z_0\in D$ and $R>0$ we have
\[
\overline{F_{z_0}(x,R)}\cap\de D\subseteq \hbox{\rm Ch}(x)\;.
\]
In particular, if $x$ is a strictly $\C$-linearly convex point then $\overline{F_{z_0}(x,R)}\cap\de D=\{x\}$.
\end{proposition}

To simplify subsequent statements, let us introduce a definition.

\begin{definition}
Let $D\subset\C^n$ be a hyperbolic convex domain, and $f\in\Hol(D,D)$ without fixed points.
The \emph{target set} of $f$ is defined as
\[
T(f)=\bigcup_h h(D)\subseteq\de D\;,
\]
where the union is taken with respect to all the holomorphic maps $h\in\Hol(D,\C^n)$ obtained as limit of a subsequence of iterates of~$f$. We have $T(f)\subseteq\de D$ because the sequence of iterates $\{f^k\}$ is compactly divergent.
\end{definition}

As a consequence of Proposition~\ref{th:2.aggiunta} we get:

\begin{corollary}[\cite{AR}]
\label{th:2.aggiuntobis}
Let $D\subset\subset\C^n$ be a $C^2$ bounded convex domain,
and $f\in\Hol(D,D)$ without fixed points. Then there exists 
$x_0\in\de D$ such that 
\[
T(f)\subseteq\hbox{\rm Ch}(x_0)\;.
\]
In particular, if $D$ is strictly
$\C$-linearly convex then the sequence of iterates $\{f^k\}$ converges to the constant map~$x_0$.
\end{corollary}

\begin{proof}
Let $x_0\in\de D$ be given by Theorem~\ref{th:2.Wolffconv}, and
fix $z_0\in D$. Given $z\in D$, choose $R>0$ such that $z\in E_{z_0}(x_0,R)$. 
If $h\in\Hol(D,\C^n)$ is the limit of a subsequence of iterates
then Theorem~\ref{th:2.Wolffconv} and Proposition~\ref{th:2.aggiunta} yield
\[
h(z)\in\overline{F{z_0}(x,R)}\cap\de D\subset\hbox{\rm Ch}(x_0)\;,
\]
and we are done.
\qed
\end{proof}

\begin{remark}
Zimmer \cite{Zi} has proved Corollary~\ref{th:2.aggiuntobis} for bounded convex domains 
with $C^{1,\alpha}$ boundary. We conjecture that it should hold for strictly $\C$-linearly convex domains without smoothness assumptions on the boundary. 
\end{remark}

Let us now drop any smothness or strict convexity condition on the boundary. In this general
context, an useful result is the following:

%

\begin{lemma}
\label{th:2.AV} 
Let $D\subset\C^n$ be a convex domain. Then for every connected complex manifold $X$ and every holomorphic map
$h\colon X\to\C^n$ such that $h(X)\subset\overline{D}$ and $h(X)\cap\de D\ne\emptyset$ we have 
\[
h(X)\subseteq\bigcap_{z\in X}\mathrm{Ch}\bigl(h(z)\bigr)\subseteq\de D\;.
\]
\end{lemma}

\begin{proof}
Take $x_0=h(z_0)\in h(X)\cap\de D$, and let $\psi$ be the weak peak function associated to a complex supporting functional
at~$x_0$. Then $\psi\circ h$ is a holomorphic function
with modulus bounded by~1 and such that $\psi\circ h(z_0)=1$; by the maximum principle
we have $\psi\circ h\equiv 1$, and hence $L\circ h\equiv L(x_0)$. In particular, $h(X)\subseteq\de D$.

Since this holds for all complex supporting hyperplanes at~$x_0$ we have shown that $h(X)\subseteq \mathrm{Ch}\bigl(h(z_0)\bigr)$; but since we know that $h(X)\subseteq\de D$ we can 
repeat the argument for any $z_0\in X$, and we are done.
\qed
\end{proof}

We can then prove a weak Wolff-Denjoy theorem:

\begin{proposition}
\label{th:2.WDwc}
Let $D\subset\subset\C^n$ be a bounded convex domain, and $f\in\Hol(D,D)$ without fixed points. 
Then there exists $x\in\de D$ such that for any $z_0\in D$ we have
\begin{equation}
T(f)\subseteq\bigcap_{R>0}\mathrm{Ch}\bigl(\overline{F_{z_0}(x,R)}\cap\de D\bigr)\;.
\label{eq:2.intF}
\end{equation}
\end{proposition}

\begin{proof}
Let $x\in\de D$ be given by Theorem~\ref{th:2.Wolffconv}. Choose $z_0\in D$ and $R>0$,
and take $z\in E_{z_0}(x,R)$. Let $h\in\Hol(D,\C^n)$ be obtained as limit of a subsequence of iterates of $f$.
Arguing as usual we know that $h(D)\subseteq\de D$; therefore
Theorem~\ref{th:2.Wolffconv} yields $h(z)\in\overline{F_{z_0}(x,R)}\cap\de D$. 
Then Lemma~\ref{th:2.AV} yields 
\[
h(D)\subseteq\mathrm{Ch}\bigl(h(z)\bigr)\subseteq\mathrm{Ch}\bigl(\overline{F_{z_0}(x,R)}\cap\de D\bigr)\;.
\]
Since $z_0$ and $R$ are arbitrary, we get the assertion.
\qed
\end{proof}

\begin{remark}
Using Lemma~\ref{th:2.horoeasy} it is easy to check that the intersection in (\ref{eq:2.intF}) is independent of the choice of~$z_0\in D$.
\end{remark}

Unfortunately, large horospheres can be too large. For instance, take $(\tau_1,\tau_2)\in\de\Delta\times\de\Delta$. Then Proposition~\ref{th:2.horopolydisk} says that the horosphere of center $(\tau_1,\tau_2)$ in the bidisk are given by
\[
F_O\bigl((\tau_1,\tau_2),R\bigr)=E_0(\tau_1,R)\times\Delta\cup \Delta\times E_0(\tau_2,R)\;,
\]
where $E_0(\tau,R)$ is the horocycle of center~$\tau\in\de\Delta$ and radius $R>0$ in the unit disk~$\Delta$, and a not difficult computation shows that 
\[
\mathrm{Ch}\bigl(\overline{F_O\bigl((\tau_1,\tau_2),R\bigr)}\cap\de\Delta^2\bigr)=\de\Delta^2\;,
\]
making the statement of Proposition~\ref{th:2.WDwc} irrelevant. So to get an effective statement we need to replace large horospheres with smaller sets. 

Small horospheres might be too small; as shown by Frosini \cite{Fro}, there are holomorphic self-maps of the polydisk with no invariant
small horospheres. We thus need another kind of horospheres, defined by Kapeluszny, Kuczumow and Reich \cite{KKR1}, and studied in detail by Budzy\'nska \cite{Bud}.
To introduce them we begin with a definition:

\begin{definition}
\label{def:3.1}
Let $D\subset\subset\C^n$ be a bounded domain, and $z_0\in D$. A sequence $\mathbf{x}=\{x_\nu\}\subset D$ converging to~$x\in\de D$ is
a \emph{horosphere sequence} at~$x$ if
the limit of $k_D(z,x_\nu)-k_D(z_0,x_\nu)$ as $\nu\to+\infty$ exists for all~$z\in D$.
\end{definition}

\begin{remark}
\label{rem:3.2}
It is easy to see that the notion of horosphere sequence does not depend on the point~$z_0$.
\end{remark}

Horosphere sequences always exist. This follows from a topological lemma:

\begin{lemma}[\cite{Reich}]
\label{th:2.Reich}
Let $(X, d)$ be a separable metric space, and for each $\nu\in\mathbb{N}$ let $a_\nu\colon X\to\R$ be a 1-Lipschitz map, i.e., $|a_{\nu}(x)-a_\nu(y)| \le d(x, y)$ for all $x$,~$y\in X$. If for each $x\in X$ the sequence $\{a_\nu(x)\}$ is bounded, then there exists a subsequence $\{a_{\nu_j}\}$ of $\{a_\nu\}$ such that $\lim_{j\to\infty} a_{\nu_j} (x)$ exists for each $x\in X$.
\end{lemma}

\begin{proof} 
Take a countable sequence $\{x_j\}_{j\in\mathbb{N}}\subset X$ dense in $X$. Clearly, the sequence $\{a_{\nu}(x_0)\}\subset\R$ admits a convergent subsequence $\{a_{\nu,0}(x_0)\}$. Analogously, the sequence $\{a_{\nu,0}(x_1)\}$ admits a convergent subsequence $\{a_{\nu,1}(x_1)\}$. Proceeding in this way, we get a countable family of subsequences $\{a_{\nu,k}\}$ of the sequence $\{a_\nu\}$ such that for each $k\in\mathbb{N}$ the limit $\lim_{\nu\to\infty} a_{\nu,k} (x_j)$ exists for $j = 0,\ldots, k$. We claim that setting $a_{\nu_j}=a_{j,j}$ the subsequence $\{a_{\nu_j}\}$ is as desired. Indeed, given $x\in X$ and $\eps>0$ we can find $x_h$ such that $d(x,x_h)< \eps/2$, and then we have
\begin{eqnarray*}
0 &\le& \limsup_{j\to\infty} a_{\nu_j}(x) - \liminf_{j\to\infty} a_{\nu_j} (x)\\
&=& \bigl[\limsup_{j\to\infty}\bigl(a_{\nu_j}(x) - a_{\nu_j} (x_h)\bigr) + \lim_{j\to\infty} a_{\nu_j} (x_h)\bigr]\\
&&\quad - 
\bigl[\liminf_{j\to\infty}\bigl(a_{\nu_j}(x) - a_{\nu_j} (x_h)\bigr) + \lim_{j\to\infty} a_{\nu_j} (x_h)\bigr]\\
&\le& 2d(x, x_h) < \eps\;.
\end{eqnarray*}
Since $\eps$ was arbitrary, it follows that the limit $\lim_{j\to\infty} a_{\nu_j}(x)$ exists, as required. 
\qed
\end{proof}

Then: 

\begin{proposition}[\cite{BKR}]
\label{th:2.horseq}
Let $D\subset\subset\C^n$ be a bounded convex domain, and $x\in\de D$. Then every sequence
$\{x_\nu\}\subset D$ converging to~$x$ contains a subsequence which is a horosphere sequence at~$x$.
\end{proposition}

\begin{proof}
Let $X=D\times D$ be endowed with the distance 
\[
d\bigl((z_1,w_1),(z_2,w_2)\bigr)=k_D(z_1,z_2)+k_D(w_1,w_2)
\] 
for all $z_1$, $z_2$, $w_1$, $w_2\in D$. 

Define $a_\nu\colon X\to\R$ by setting $a_\nu(z,w)=k_D(w,x_\nu)-k_D(z,x_\nu)$. The triangular inequality shows that each $a_\nu$ is 1-Lipschitz,
and for each $(z,w)\in X$ the sequence $\{a_\nu(z,w)\}$ is bounded by $k_D(z,w)$. Lemma~\ref{th:2.Reich} then yields a subsequence $\{x_{\nu_j}\}$ such that
$\lim_{j\to\infty} a_{\nu_j}(z,w)$ exists for all $z$, $w\in D$, and this exactly means that $\{x_{\nu_j}\}$ is a horosphere sequence.
\qed
\end{proof}

We can now introduce a new kind of horospheres.

\begin{definition}
\label{def:3.2}
Let $D\subset\subset\C^n$ be a bounded convex domain. Given $z_0\in D$,
let $\mathbf{x}=\{x_\nu\}$ be a horosphere sequence at~$x\in\de D$, and take $R>0$.
Then the \emph{sequence horosphere} $G_{z_0}(x,R,\mathbf{x})$ is defined as
\[
G_{z_0}(x,R,\mathbf{x})=\bigl\{z\in D\bigm|\lim_{\nu\to+\infty}\bigl[k_D(z,x_\nu)-k_D(z_0,x_\nu)\bigr]
<\mlog R\bigr\}\;.
\]
\end{definition}

The basic properties of sequence horospheres are contained in the following:

\begin{proposition}[\cite{KKR1, Bud, BKR}]
\label{th:2.newhor}
Let $D\subset\subset\C^n$ be a bounded convex domain. Fix $z_0\in D$, and let $\mathbf{x}=\{x_\nu\}\subset D$ be a horosphere sequence
at~$x\in\de D$. Then:
\begin{enumerate}
\item[\rm (i)] $E_{z_0}(x,R)\subseteq G_{z_0}(x,R,\mathbf{x})\subseteq F_{z_0}(x,R)$ for all $R > 0$; 
\item[\rm (ii)] $G _{z_0}(x,R,\mathbf{x})$ is nonempty and convex for all $R>0$;
\item[\rm (iii)] $\overline{G_{z_0}(x,R_1,\mathbf{x})}\cap D\subset G_{z_0} (x,R_2, \mathbf{x})$ for all $0<R_1<R_2$;
\item[\rm (iv)] $B_D(z_0,\mlog R)\subset G_{z_0}(x, R, \mathbf{x})$ for all $R>1$;
\item[\rm (v)] $B_D(z_0,-\mlog R)\cap G_{z_0}(x, R, \mathbf{x})=\emptyset$ for all $0<R<1$;
\item[\rm (vi)] $\bigcup\limits_{R>0}G_{z_0}(x,R,\mathbf{x})=D$ and
$\bigcap\limits_{R>0}G_{z_0}(x,R,\mathbf{x})=\emptyset$.
\end{enumerate}
\end{proposition}

\begin{remark}
\label{rem:3.4}
If $\mathbf{x}$ is a horosphere sequence at~$x\in\de D$ then it is not difficult to check that the family $\{G_z(x,1,\mathbf{x})\}_{z\in D}$ and the family $\{G_{z_0}(x,R,\mathbf{x})\}_{R>0}$
with $z_0\in D$ given, coincide. 
\end{remark}

Then we have the following version of Theorem~\ref{th:2.Wolff}:

\begin{theorem}[\cite{Bud, AR}] 
\label{th:2.Wolffdue}
Let $D\subset\subset\C^n$ be a convex domain, and let 
$f\in\Hol(D,D)$ without fixed points. Then there exists
$x\in\de D$ and a horosphere sequence $\mathbf{x}$ at~$x$ such that
\[
f\bigl(G_{z_0}(x,R,\mathbf{x})\bigr)\subseteq G_{z_0}(x,R,\mathbf{x})
\]
for every $z_0\in D$ and $R>0$.
\end{theorem}

\begin{proof}
As in the proof of  Theorem~\ref{th:2.Wolffconv}, for $\nu>0$ put
$f_\nu=(1-1/\nu)f\in\Hol(D,D)$; then $f_\nu\to f$ as $\nu\to+\infty$, each $f_\nu$ has a
fixed point $x_\nu\in D$, and up to a subsequence we can assume that $x_\nu\to x\in\de D$.
Furthermore, by Proposition~\ref{th:2.horseq} up to a subsequence we can also assume that $\mathbf{x}=\{x_\nu\}$ is a horosphere
sequence at~$x$.

Now, for every $z\in D$ we have
\[
\bigl|k_D\bigl(f(z),x_\nu\bigr)-k_D\bigl(f_\nu(z),x_\nu\bigr)\bigr|\le
k_D\bigl(f_\nu(z),f(z)\bigr)\to 0
\] 
as $\nu\to+\infty$. Therefore if $z\in G_{z_0}(x,R,\mathbf{x})$ we get
\begin{eqnarray*}
\lim_{\nu\to+\infty}\bigl[k_D\bigl(f(z),x_\nu\bigr)&-&k_D(z_0,x_\nu)\bigr]\\
&\le&\lim_{\nu\to+\infty}\bigl[k_D\bigl(f_\nu(z),x_\nu\bigr)-k_D(z_0,x_\nu)\bigr]\\
&&\qquad+\limsup_{\nu\to+\infty}\bigl[k_D\bigl(f(z),x_\nu\bigr)-k_D\bigl(f_\nu(z),x_\nu\bigr)\bigr]\\
&\le&\lim_{\nu\to+\infty}\bigl[k_D(z,x_\nu)-k_D(z_0,x_\nu)\bigr]<\mlog R
\end{eqnarray*}
because $f_\nu(x_\nu)=x_\nu$ for all $\nu\in\mathbb{N}$, and we are done.
\qed
\end{proof}

Putting everything together we can prove the following Wolff-Denjoy theorem
for (not necessarily strictly or smooth) convex domains:

\begin{theorem}[\cite{AR}]
\label{th:2.dpqua}
Let $D\subset\subset\C^n$ be a bounded convex domain,
and $f\in\Hol(D,D)$ without fixed points. Then there exist $x\in \de D$ and a horosphere
sequence $\mathbf{x}$ at~$x$ such that for any $z_0\in D$ we have
\[
T(f)\subseteq \bigcap_{z\in D}\hbox{\rm Ch}\bigl(\overline{G_z(x,1,\mathbf{x})}\cap\de D\bigr)=\bigcap_{R>0}\hbox{\rm Ch}\bigl(\overline{G_{z_0}(x,R,\mathbf{x})}\cap\de D\bigr)\;.
\]
\end{theorem}

\begin{proof}
The equality of the intersections is a
consequence of Remark~\ref{rem:3.4}. Then the assertion follows from Theorem~\ref{th:2.Wolffdue} and Lemma~\ref{th:2.AV} as in the proof of Proposition~\ref{th:2.WDwc}.
\qed
\end{proof}

To show that this statement is actually better than Proposition~\ref{th:2.WDwc} let us consider the case of the polydisk.

\begin{lemma}
\label{th:2.horoseqpoly}
Let $\mathbf{x}=\{x_\nu\}\subset\Delta^n$ be a horosphere sequence converging to
$\xi=(\xi_1,\ldots,\xi_n)\in\de\Delta^n$. Then for every $1\le j\le n$ such that $|\xi_j|=1$ the limit
\begin{equation}
\alpha_j:=\lim_{\nu\to+\infty}\min_h\left\{\frac{1-|(x_\nu)_h|^2}{1-|(x_{\nu})_j|^2}\right\}\le 1
\label{eq:2.alpha}
\end{equation}
exists, and we have
\[
G_O(\xi,R,\mathbf{x})=\left\{z\in\Delta^n\biggm| \max_j\left\{\alpha_j\frac{|\xi_j-z_j|^2}{1- |z_j|^2}\biggm| |\xi_j|=1\right\}  < R\right\}=\prod_{j=1}^n E_j\;,
\]
where
\[
E_j=\left\{\begin{array}{ll}
\Delta\quad&\hbox{if $|\xi_j|<1$,}\\
E_0(\xi_j,R/\alpha_j)\quad&\hbox{if $|\xi_j|=1$.}
\end{array}\right.
\]
\end{lemma}

\begin{proof}
Given $z=(z_1,\ldots,z_n)\in\Delta^n$, let $\gamma_z\in\Aut(\Delta^n)$ be defined by
\[
\gamma_z(w)=\left(\frac{w_1-z_1}{1-\overline{z_1}w_1},\ldots,\frac{w_n-z_n}{1-\overline{z_n}w_n}
\right)\;,
\]
so that $\gamma_z(z)=O$. Then
\[
k_{\Delta^n}(z,x_\nu)-k_{\Delta^n}(O,x_\nu)=k_{\Delta^n}\bigl(O,\gamma_z(x_\nu)\bigr)-
k_{\Delta^n}(O,x_\nu)\;.
\]
Now, writing $\n z\n=\max_j\{|z_j|\}$ we have
\[
k_{\Delta^n}(O,z)=\max_j\{k_\Delta(0,z_j)\}=\max_j\left\{\mlog\frac{1+|z_j|}{1-|z_j|}\right\}
=\mlog\frac{1+\n z\n}{1-\n z\n}\;,
\]
and hence
\[
k_{\Delta^n}(z,x_\nu)-k_{\Delta^n}(O,x_\nu)=\log\biggl(\frac{1+\n \gamma_z(x_\nu)\n}{
        1+\n x_\nu\n}\biggr)+\mlog\biggl(\frac{1-\n x_\nu\n^2}{1-\n\gamma_z(x_\nu)\n^2}
	\biggr)\;.
\]
Since $\n \gamma_z(\xi)\n=\n \xi\n=1$, we just have to study the behavior of the
second term, that we know has a limit as $\nu\to+\infty$ because $\mathbf{x}$ is a horosphere sequence. Now
\[
\displaylines{1-\n x_\nu\n^2=\min_h\bigl\{1-|(x_\nu)_h|^2\bigr\};\cr
1-\n \gamma_z(x_\nu)\n^2=\min_j\biggl\{\frac{1-|z_j|^2}{|1-\overline{z_j}(x_\nu)_j|^2}
        (1-|(x_\nu)_j|^2)\biggr\}\;.\cr}
\]
Therefore
\[
\frac{1-\n x_\nu\n^2}{1-\n \gamma_z(x_\nu)\n^2}=\max_j\min_h\biggl\{\frac{1-|(x_\nu)_h|^2}{1-
	|(x_\nu)_j|^2}\,\cdot\,\frac{|1-\overline{z_j}(x_\nu)_j|^2}{1-|z_j|^2}\biggr\}\;.
\]
Taking the limit as $\nu\to+\infty$ we get
\begin{equation}
\lim_{\nu\to+\infty}\frac{1-\n x_\nu\n^2}{1-\n \gamma_z(x_\nu)\n^2}
=\max_j\left\{\frac{|1-z_j\overline{\xi_j}|^2}{1-|z_j|^2}\lim_{\nu\to+\infty}\min_h\left\{
\frac{1-|(x_\nu)_h|^2}{1-|(x_\nu)_j|^2}\right\}\right\}\;.
\label{eq:2.hDuno}
\end{equation}
In particular, we have shown that the limit in (\ref{eq:2.alpha}) exists, and it is bounded by~1
(it suffices to take $h=j$). Furthermore, if $|\xi_j|<1$ then $\alpha_j=0$; so (\ref{eq:2.hDuno})
becomes
\[
\lim_{\nu\to+\infty}\frac{1-\n x_\nu\n^2}{1-\n \gamma_z(x_\nu)\n^2}
=\max\left\{\alpha_j\frac{|1-z_j\overline{\xi_j}|^2}{1-|z_j|^2}\biggm| |\xi_j|=1\right\}\;,
\]
and the lemma follows.
\qed
\end{proof}

Now, a not too difficult computation shows that
\[
\mathrm{Ch}(\xi)=\bigcap_{|\xi_j|=1}\{\eta\in\de\Delta^n\mid \eta_j=\xi_j\}
\]
for all $\xi\in\de\Delta^n$. As a consequence, 
\[
\mathrm{Ch}\bigl(\overline{G_O(\xi,R,\mathbf{x})}\cap\de\Delta^n\bigr)=
\bigcup_{j=1}^n \overline\Delta\times\cdots\times C_j(\xi)\times\cdots\times\overline\Delta\;,
\]
where
\[
C_j(\xi)=\left\{\begin{array}{ll}
\{\xi_j\}\quad&\hbox{if $|\xi_j|=1$,}\cr
\de\Delta\quad&\hbox{if $|\xi_j|<1$.}\cr
\end{array}\right.
\]
Notice that the right-hand sides do not depend either on $R$ or on the horosphere sequence $\mathbf{x}$, but only on~$\xi$.

So Theorem~\ref{th:2.dpqua} in the polydisk assumes the following form:

\begin{corollary} 
\label{th:2.dpquapd}
Let $f\in\Hol(\Delta^n,\Delta^n)$ be without fixed points. Then there exists $\xi\in \de\Delta^n$ such that
\begin{equation}
T(f)\subseteq \bigcup_{j=1}^n \overline\Delta\times\cdots\times C_j(\xi)\times\cdots\times\overline\Delta\;.
\label{eq:2.union}
\end{equation}
\end{corollary}

Roughly speaking, this is the best one can do, in the sense that while it might be true (for instance in the bidisk; see Theorem~\ref{th:2.Herve} below) that the image of a limit point of the sequence of iterates of~$f$ is always contained in just one of the sets appearing in the right-hand side of~(\ref{eq:2.union}), it is impossible to determine a priori in which one it is contained on the basis of the point~$\xi$ only;
it is necessary to know something more about the map~$f$. Indeed, Herv\'e has proved the following:

\begin{theorem}[Herv\'e, \cite{Herve}]
\label{th:2.Herve}
Let $F=(f,g)\colon\Delta^2\to\Delta^2$ be a holomorphic self-map of the bidisk, and write $f_w=f(\cdot,w)$ and $g_z=g(z,\cdot)$. Assume that $F$ has no fixed points in~$\Delta^2$. Then one and only one of the following cases occurs:
\begin{enumerate}
\item[\rm (i)] if $g(z,w)\equiv w$ (respectively, $f(z,w)\equiv z$) then the sequence of iterates of~$F$ converges uniformly on compact sets to $h(z,w)=(\sigma,w)$, where $\sigma$ is the common Wolff point of the $f_w$'s (respectively, to $h(z,w)=(z,\tau)$, where $\tau$ is the common Wolff point
of the $g_z$'s);
\item[\rm (ii)] if $\Fix(f_w)=\emptyset$ for all $w\in\Delta$ and $\Fix(g_z)=\{y(z)\}\subset\Delta$ for all $z\in\Delta$ (respectively, if $\Fix(f_w)=\{x(w)\}$ and $\Fix(g_z)=\emptyset$) then 
$T(f)\subseteq\{\sigma\}\times\overline\Delta$, where $\sigma\in\de\Delta$ is the common Wolff point of the~$f_w$'s (respectively, $T(f)\subseteq\overline\Delta\times\{\tau\}$, where $\tau$ is the common Wolff point
of the $g_z$'s);
\item[\rm (iii)] if $\Fix(f_w)=\emptyset$ for all $w\in\Delta$ and $\Fix(g_z)=\emptyset$ for all $z\in\Delta$ then either $T(f)\subseteq\{\sigma\}\times\overline\Delta$ or
$T(f)\subseteq\overline\Delta\times\{\tau\}$, where $\sigma\in\de\Delta$ is the common Wolff point of the~$f_w$'s, and $\tau\in\de\Delta$ is the common Wolff point of the~$g_z$;
\item[\rm (iv)] if $\Fix(f_w)=\{x(w)\}\subset\Delta$ for all $w\in\Delta$ and $\Fix(g_z)=\{y(z)\}\subset\Delta$
for all $z\in\Delta$ then there are $\sigma$,~$\tau\in\de D$ such that the sequence of iterates converges to the constant map $(\sigma,\tau)$.
\end{enumerate}
\end{theorem}

All four cases can occur: see \cite{Herve} for the relevant examples.

\section{Carleson measures and Toeplitz operators}
\label{sec:3}

In this last section we shall describe a completely different application of the Kobayashi distance to complex analysis. To describe the problem we would like to deal with we need a few definitions.

\begin{definition}
We shall denote by $\nu$ the Lebesgue measure in~$\C^n$. If $D\subset\subset\C^n$ is a bounded domain and $1\le p\le\infty$, we shall denote by $L^p(D)$ the usual space of measurable $p$-integrable
complex-valued functions on~$D$, with the norm
\[
\|f\|_p=\left[\int_D |f(z)|^p\,\D\nu(z)\right]^{1/p}
\]
if $1\le p<\infty$, while $\|f\|_\infty$ will be the essential supremum of~$|f|$ in~$D$. Given $\beta\in\R$, we shall also consider the \emph{weighted $L^p$-spaces} $L^p(D,\beta)$, which are the $L^p$ spaces with respect to the measure $\delta^\beta\nu$, where $\delta\colon D\to\R^+$ is the Euclidean distance from the boundary: $\delta(z)=d(z,\de D)$. The norm in $L^p(D,\beta)$ is given by
\[
\|f\|_{p,\beta}=\left[\int_D |f(z)|^p\delta(z)^\beta\,\D\nu(z)\right]^{1/p}
\] 
for $1\le p<\infty$, and by $\|f\|_{\beta,\infty}=\|f\delta^\beta\|_\infty$ for $p=\infty$.
\end{definition}

\begin{definition}
Let $D\subset\subset\C^n$ be a bounded domain in~$\C^n$, and $1\le p\le\infty$. The \emph{Bergman space} $A^p(D)$ is the Banach space $A^p(D)=L^p(D)\cap\Hol(D,\C)$ endowed with the norm $\|\cdot\|_p$. More generally, given $\beta\in\R$ the \emph{weighted Bergman space} $A^p(D,\beta)$ is the Banach space $A^p(D,\beta)=L^p(D,\beta)\cap\Hol(D,\C)$ endowed with the norm
$\|\cdot\|_{p,\beta}$. 
\end{definition}

The Bergman space $A^2(D)$ is a Hilbert space; this allows us to introduce one of the most studied objects in complex analysis.

\begin{definition}
Let $D\subset\subset\C^n$ be a bounded domain in~$\C^n$. The \emph{Bergman projection}
is the orthogonal projection $P\colon L^2(D)\to A^2(D)$.
\end{definition}

It is a classical fact (see, e.g., \cite[Section~1.4]{Krantz} for proofs) that the Bergman projection is an integral operator: it exists a function $K\colon D\times D\to\C$ such that
\begin{equation}
Pf(z)=\int_D K(z,w)f(w)\,\D\nu(w)
\label{eq:3.Bergproj}
\end{equation}
for all $f\in L^2(D)$. It turns out that $K$ is holomorphic in the first argument, $K(w,z)=\overline{K(z,w)}$ for all $z$, $w\in D$, and it is a \emph{reproducing kernel} for $A^2(D)$ in the sense that
\[
f(z)=\int_D K(z,w)f(w)\,\D\nu(w)
\] 
for all $f\in A^2(D)$.

\begin{definition}
Let $D\subset\subset\C^n$ be a bounded domain in~$\C^n$. The function $K\colon D\times D\to\C$ satisfying (\ref{eq:3.Bergproj}) is the \emph{Bergman kernel} of~$D$.
\end{definition}

\begin{remark}
\label{rem:3.BK}
It is not difficult to show (see again, e.g., \cite[Section~1.4]{Krantz}) that $K(\cdot,w)\in A^2(D)$ for all $w\in D$, and
that 
\[
\|K(\cdot,w)\|^2_2=K(w,w)>0\;.
\]
\end{remark}

A classical result in complex analysis says that in strongly pseudoconvex domains the Bergman projection can be extended to all~$L^p$ spaces:

\begin{theorem}[Phong-Stein, \cite{PS}]
\label{th:3.PS}
Let $D\subset\subset\C^n$ be a strongly pseudoconvex domain with $C^\infty$ boundary, and $1\le p\le\infty$. Then the formula $(\ref{eq:3.Bergproj})$ defines a continuous operator~$P$ from $L^p(D)$ to~$A^p(D)$. Furthermore, for any $r>p$ there is $f\in L^p(D)$ such that $Pf\notin A^r(D)$.
\end{theorem}

Recently, \v Cu\v ckovi\'c and McNeal posed the following question: does there exist a natural operator, somewhat akin to the Bergman projection, mapping $L^p(D)$ into $A^r(D)$ for some $r>p$?
To answer this question, they considered Toeplitz operators.

\begin{definition}
Let $D\subset\subset\C^n$ be a strongly pseudoconvex domain with $C^\infty$ boundary. Given a measurable function $\psi\colon D\to\C$, the \emph{multiplication operator} of \emph{symbol}~$\psi$ is simply defined by $M_\psi(f)=\psi f$. Given $1\le p\le\infty$, a symbol $\psi$ is \emph{$p$-admissible} if $M_\psi$ sends~$L^p(D)$ into itself; for instance, a $\psi\in L^\infty(D)$ is $p$-admissible for all~$p$. If $\psi$ is $p$-admissible, the \emph{Toeplitz operator} $T_\psi\colon L^p(D)\to A^p(D)$ of \emph{symbol}~$\psi$ is defined by $T_\psi=P\circ M_\psi$, that is
\[
T_\psi(f)(z)=P(\psi f)(z)=\int_D K(z,w)f(w)\psi(w)\,\D\nu(w)\;.
\]
\end{definition}

\begin{remark}
More generally, if $A$ is a Banach algebra, $B\subset A$ is a Banach subspace, $P\colon A\to B$ is a projection and $\psi\in A$, the Toeplitz operator~$T_\psi$ of symbol~$\psi$ is defined by $T_\psi(f)=P(\psi f)$. Toeplitz operators are a much studied topic in functional analysis; see, e.g., \cite{Up}.
\end{remark}

Then \v Cu\v ckovi\'c and McNeal were able to prove the following result:

\begin{theorem}[\v Cu\v ckovi\'c-McNeal, \cite{CMcN}]
\label{th:3.CMcN}
Let $D\subset\subset\C^n$ be a strongly pseudoconvex domain with $C^\infty$ boundary. 
If $1<p<\infty$ and $0\le\beta<n+1$ are such that
\begin{equation}
\frac{n+1}{n+1-\beta}<\frac{p}{p-1}
\label{eq:3.CMcN}
\end{equation}
then the Toeplitz operator $T_{\delta^\beta}$ maps continuously $L^p(D)$ in $A^{p+G}(D)$, where
\[
G=\frac{p^2}{\frac{n+1}{\beta}-p}\;.
\]
\end{theorem}

\v Cu\v ckovi\'c and McNeal also asked whether the gain $G$ in integrability is optimal; they were able to positively answer to this question only for $n=1$. The positive answer in higher dimension has been given by Abate, Raissy and Saracco \cite{ARS}, as a corollary of their study of a larger class of Toeplitz operators on strongly pseudoconvex domains. This study, putting into play another 
important notion in complex analysis, the one of Carleson measure, used as essential tool the Kobayashi distance; in the next couple of sections we shall describe the gist of their results.

\subsection{Definitions and results}
\label{subsec:3.1}

In this subsection and the next $D$ will always be a bounded strongly pseudoconvex domain with $C^\infty$ boundary. 
We believe that the results might be generalized to other classes of domains with $C^\infty$ boundary (e.g., finite type domains), and possibly to domains with less smooth boundary, but we will not pursue this subject here.

Let us introduce the main player in this subject.

\begin{definition}
\label{def:3.Toep}
Let $D\subset\subset\C^n$ be a strongly pseudoconvex domain with $C^\infty$ boundary, and $\mu$ a finite positive Borel measure on~$D$. Then the \emph{Toeplitz operator}~$T_\mu$ of \emph{symbol}~$\mu$ is defined by
\[
T_\mu(f)(z)=\int_D K(z,w)f(w)\,\D\mu(w)\;,
\]
where $K$ is the Bergman kernel of~$D$. 
\end{definition}

For instance, if $\psi$ is an admissible symbol then the Toeplitz operator $T_\psi$ defined above is the Toeplitz operator $T_{\psi\nu}$ according to Definition~\ref{def:3.Toep}. 

In Definition~\ref{def:3.Toep} we did not specify domain and/or range of the Topelitz operator~$\mu$ because the main point of the theory we are going to discuss is exactly to link properties of the measure~$\mu$ with domain and range of~$T_\mu$.

Toeplitz operators associated to measures have been extensively studied on the unit disk~$\Delta$ and on the unit ball~$B^n$ (see, e.g., \cite{Li}, \cite{LL}, \cite{Kapt}, \cite{Zhu} and references therein); but \cite{ARS} has been one of the first papers studying them in strongly pseudoconvex domains.

The kind of measure we shall be interested in is described in the following

\begin{definition}
Let $D\subset\subset\C^n$ be a strongly pseudoconvex domain with $C^\infty$ boundary, $A$ a Banach space of complex-valued functions on~$D$, and $1\le p\le\infty$. We shall say that 
a finite positive Borel measure~$\mu$ on~$D$ is a \emph{$p$-Carleson measure} for $A$ 
if $A$ embeds continuously into~$L^p(\mu)$, that is if there exists $C>0$ such that
\[
\int_D |f(z)|\,\D\mu(z)\le C\|f\|^p_A
\]
for all $f\in A$, where $\|\cdot\|_A$ is the norm in~$A$.
\end{definition} 

\begin{remark}
When the inclusion $A\hookrightarrow L^p(\mu)$ is compact, $\mu$ is called \emph{vanishing Carleson measure}. Here we shall discuss vanishing Carleson measures only in the remarks.\end{remark}

Carleson measures for the Hardy spaces $H^p(\Delta)$ were introduced by Carleson \cite{Carl} to solve the famous corona problem. We shall be interested in Carleson measures for the weighted Bergman spaces $A^p(D,\beta)$; they have been studied by many authors when $D=\Delta$ or $D=B^n$ (see, e.g., \cite{Lueck}, \cite{DW}, \cite{Zhu} and references therein), 
but more rarely when $D$ is a strongly pseudoconvex domain (see, e.g., \cite{CM} and \cite{AS}). 

The main point here is to give a geometric characterization of which measures are Carleson. 
To this aim we introduce the following definition, bringing into play the Kobayashi distance.

\begin{definition}
Let $D\subset\subset\C^n$ be a strongly pseudoconvex domain with $C^\infty$ boundary, and $\theta>0$. We shall say that a finite positive Borel measure~$\mu$ on~$D$ is \emph{$\theta$-Carleson} if there exists $r>0$ and $C_r>0$ such that
\begin{equation}
\mu\bigl(B_D(z_0, r)\bigr)\le C_r \nu\bigl(B_D(z_0, r)\bigr)^\theta
\label{eq:3.thetaC}
\end{equation}
for all $z_0\in D$. We shall see that if (\ref{eq:3.thetaC}) holds for some $r>0$ 
then it holds for all $r>0$. 
\end{definition}

\begin{remark}
There is a parallel vanishing notion: we say that $\mu$ is \emph{vanishing $\theta$-Carleson} if
there exists $r>0$ such that
\[
\lim_{z_0\to\de D}\frac{\mu\bigl(B_D(z_0, r)\bigr)}{\nu\bigl(B_D(z_0, r)\bigr)^\theta}=0\;.
\]
\end{remark}

For later use, we recall two more definitions.

\begin{definition}
Let $D\subset\subset\C^n$ be a strongly pseudoconvex domain with $C^\infty$ boundary. Given $w\in D$, the \emph{normalized Bergman kernel} in~$w$ is given by
\[
k_w(z)=\frac{K(z,w)}{\sqrt{K(w,w)}}\;.
\]
Remark~\ref{rem:3.BK} shows that $k_w\in A^2(D)$ and $\|k_w\|_2=1$ for all $w\in D$.
\end{definition}

\begin{definition}
Let $D\subset\subset\C^n$ be a strongly pseudoconvex domain with $C^\infty$ boundary, and
$\mu$ a finite positive Borel measure on~$D$. The \emph{Berezin transform} of~$\mu$ is the function $B\mu\colon D\to\R^+$ defined by
\[
B\mu(z)=\int_D |k_z(w)|^2\,\D\mu(w)\;.
\]
\end{definition}

Again, part of the theory will describe when the Berezin transform of a measure is actually defined.

We can now state the main results obtained in \cite{ARS}:

\begin{theorem}[Abate-Raissy-Saracco, \cite{ARS}]
\label{th:3.ToepCarl}
Let $D\subset\subset\C^n$ be a strongly pseudoconvex domain with $C^\infty$ boundary, $1<p< r<\infty$ and
$\mu$ a finite positive Borel measure on~$D$. Then $T_\mu$ maps $A^p(D)$ into $A^r(D)$ if and only if $\mu$ is a $p$-Carleson measure for $A^p\bigl(D,(n+1)(\frac{1}{p}-\frac{1}{r})\bigr)$.
\end{theorem}

\begin{theorem}[Abate-Raissy-Saracco, \cite{ARS}]
\label{th:3.Carltheta}
Let $D\subset\subset\C^n$ be a strongly pseudoconvex domain with $C^\infty$ boundary, $1<p<\infty$ and $\theta\in\bigl(1-\frac{1}{n+1},2\bigr)$. Then 
 a finite positive Borel measure $\mu$ on~$D$ is a $p$-Carleson measure for $A^p\bigl(D,(n+1)(\theta-1)\bigr)$ if and only if $\mu$ is a $\theta$-Carleson measure.
 \end{theorem}
 
 \begin{theorem}[Abate-Raissy-Saracco, \cite{ARS}]
\label{th:3.thetaBerez}
Let $D\subset\subset\C^n$ be a strongly pseudoconvex domain with $C^\infty$ boundary, and $\theta>0$. Then 
 a finite positive Borel measure $\mu$ on~$D$ is $\theta$-Carleson if and only the Berezin transform $B\mu$ exists and $\delta^{(n+1)(1-\theta)}B\mu\in L^\infty(D)$.
 \end{theorem}
 
 \begin{remark}
 This is just a small selection of the results contained in \cite{ARS}. There one can find statements also for $p=1$ or $p=\infty$, for other values of~$\theta$, and on the mapping properties of Toeplitz operators on weighted Bergman spaces. Furthermore, there it is also shown that $T_\mu$ is a compact operator from $A^p(D)$ into $A^r(D)$ if and only if $\mu$ is a vanishing $p$-Carleson measure for $A^p\bigl(D,(n+1)(\frac{1}{p}-\frac{1}{r})\bigr)$; that $\mu$ is a vanishing $p$-Carleson measure for $A^p\bigl(D,(n+1)(\theta-1)\bigr)$ if and only if $\mu$ is a vanishing $\theta$-Carleson measure; and that $\mu$ is a vanishing $\theta$-Carleson measure if and only if 
 $\delta^{(n+1)(1-\theta)}(z)B\mu(z)\to 0$ as $z\to\de D$.
  \end{remark}
 
\begin{remark}
The condition ``$p$-Carleson" is independent of any radius $r>0$, while the condition ``$\theta$-Carleson" does not depend on $p$. Theorem~\ref{th:3.Carltheta} thus implies that if $\mu$ satisfies (\ref{eq:3.thetaC}) for some $r>0$ then it satisfies the same condition (with possibly different constants) for all $r>0$; and that if $\mu$ is 
$p$-Carleson for $A^p\bigl(D,(n+1)(\theta-1)\bigr)$ for some $1<p<\infty$ then it is $p$-Carleson for $A^p\bigl(D,(n+1)(\theta-1)\bigr)$ for all $1<p<\infty$.
 \end{remark}
 
 In the next subsection we shall describe the proofs; we end this subsection showing why these results give a positive answer to the question raised by \v Cu\v ckovi\'c and McNeal.
 
Assume that $T_{\delta^\beta}$ maps $L^p(D)$ (and hence $A^p(D)$) into $A^{p+G}(D)$. By Theorem~\ref{th:3.ToepCarl} $\delta^\beta\mu$ must be a $p$-Carleson measure for 
$A^p\bigl(D,(n+1)(\frac{1}{p}-\frac{1}{p+G})\bigr)$. By Theorem~\ref{th:3.Carltheta} this can happen
if and only if $\delta^\beta\nu$ is a $\theta$-Carleson measure, where 
\begin{equation}
\theta=1+\frac{1}{p}-\frac{1}{p+G}\;;
\label{eq:3.thetaG}
\end{equation}
notice that $1\le\theta<2$ because $p>1$ and $G\ge 0$. So we need to understand when $\delta^\beta\nu$ is $\theta$-Carleson. For this we need the following

\begin{lemma}
\label{th:3.lemma5}
Let $D\subset\subset\C^n$ be a strongly pseudoconvex domain with $C^2$ boundary, Then there exists $C>0$ such that for every $z_0\in D$
and $r>0$ one has
\[
\forall{z\in B_D(z_0, r)}\ \ \ \ 
C\E^{2r}\, \delta(z_0)\ge \delta(z)\ge \frac{\E^{-2r}}{C}\delta(z_0)\;.
\]
\end{lemma}

\begin{proof} 
Let us fix $w_0\in D$. Then Theorems~\ref{th:1.bestiu} and~\ref{th:1.bestil} yield $c_0$, $C_0>0$ such that
\begin{eqnarray*}
c_0-\mlog \delta(z)&\le& k_D(w_0,z)\le k_D(z_0,z)+k_D(z_0,w_0)\\
&\le&
r+C_0-\mlog \delta(z_0)\;,
\end{eqnarray*}
for all $z\in B_D(z_0, r)$, and hence
\[
e^{2(c_0-C_0)}\delta(z_0)\le \E^{2r}\delta(z)\;.
\]
The left-hand inequality is obtained in the same way reversing the roles of~$z_0$ and~$z$.
\qed
\end{proof}

\begin{corollary}
\label{th:3.deltaCarl}
Let $D\subset\subset\C^n$ be a strongly pseudoconvex domain with $C^2$ boundary, 
Given $\beta>0$, put $\nu_\beta=\delta^\beta\nu$. Then $\nu_\beta$ is $\theta$-Carleson
if and only if $\beta\ge(n+1)(\theta-1)$. 
\end{corollary}

\begin{proof}
Using Lemma~\ref{th:3.lemma5} we find that 
\begin{eqnarray*}
\frac{\E^{-2r}}{C} \delta(z_0)^{\beta}\nu\bigl(B_D(z_0, r)\bigr)&\le&\nu_\beta\bigl(B_D(z_0, r)\bigr)=\int_{B_D(z_0,r)}\delta(z)^\beta\,\D\nu(z)\\
&\le& C\E^{2r} \delta(z_0)^{\beta}\nu\bigl(B_D(z_0, r)\bigr)
\end{eqnarray*}
for all $z_0\in D$. Therefore $\nu_\beta$ is $\theta$-Carleson if and only if 
\[
\delta(z_0)^\beta\le C_1\nu\bigl(B_D(z_0,r)\bigr)^{\theta-1}
\]
for some $C_1>0$. Recalling Theorem~\ref{th:1.volume} we see that this is equivalent to requiring
$\beta\ge (n+1)(\theta-1)$, and we are done.
\qed
\end{proof}

In our case, $\theta$ is given by (\ref{eq:3.thetaG}); therefore $\beta\ge(n+1)(\theta-1)$ if and only if 
\[
\beta\ge (n+1)\left(\frac{1}{p}-\frac{1}{p+G}\right)\;.
\]
Rewriting this in term of $G$ we get
\[
G\le \frac{p^2}{\frac{n+1}{\beta}-p}\;,
\]
proving that the exponent in Theorem~\ref{th:3.CMcN} is the best possible, as claimed.
Furthermore, $G>0$ if and only if
\[
\frac{\beta}{n+1}<\frac{1}{p}\Leftrightarrow 1-\frac{\beta}{n+1}>1-\frac{1}{p}
\Leftrightarrow \frac{n+1}{n+1-\beta}<\frac{p}{p-1}\;,
\]
and we have also recovered condition (\ref{eq:3.CMcN}) of Theorem~\ref{th:3.CMcN}. 

Corollary~\ref{th:3.deltaCarl} provides examples of $\theta$-Carleson measures. A completely different class of examples is provided by Dirac masses distributed along uniformly discrete sequences. 

\begin{definition}
Let $(X,d)$ be a metric space.  A sequence $\Gamma=\{x_j\}\subset
X$ is \emph{uniformly discrete} if there exists $\eps>0$
such that $d(x_j,x_k)\ge\eps$ for all $j\ne k$. 
\end{definition}

Then it is possible to prove the following result:

\begin{theorem}[\cite{ARS}]
\label{th:3.dtre} 
Let $D\subset\subset\C^n$ be a bounded strongly pseudoconvex domain with $C^\infty$ boundary, considered as a metric space with the Kobayashi distance, and choose $1-\frac{1}{n+1}<\theta<2$. Let $\Gamma=\{z_j\}_{j\in\mathbb{N}}$ be a sequence
in~$D$. Then $\Gamma$ is a finite union of uniformly discrete sequences if and only if $\sum_j \delta(z_j)^{(n+1)\theta}\delta_{z_j}$ is a $\theta$-Carleson measure, where $\delta_{z_j}$ is the Dirac measure in~$z_j$.
\end{theorem}

\subsection{Proofs}
\label{subsec:3.2}

In this section we shall prove Theorems~\ref{th:3.ToepCarl}, \ref{th:3.Carltheta} and~\ref{th:3.thetaBerez}. To do so we shall need a few technical facts on the Bergman kernel and on the Kobayashi distance. To simplify statements and proofs, let us introduce the following notation.

\begin{definition}
Let $D\subset\C^n$ be a domain. Given two non-negative functions $f$, ~$g\colon D\to\R^+$ we shall write $f\preceq g$ or $g\succeq f$
to say that there is $C>0$ such that $f(z)\le C g(z)$ for all $z\in D$. The constant $C$ is 
independent of~$z\in D$, but it might depend on other parameters ($r$, $\theta$, etc.). 
\end{definition}

%

The first technical fact we shall need is an integral estimate on the Bergman kernel:

\begin{theorem}[\cite{Li}, \cite{McS}, \cite{ARS}]
\label{th:3.thm}
Let $D\subset\subset\C^n$ be a strongly pseudoconvex domain with $C^\infty$ boundary. Take $p\ge 1$ and $\beta>-1$. Then  
\[
\int_D|K(w,z_0)|^p\delta(w)^\beta\,\D\nu(w)\preceq
\cases{
 \delta(z_0)^{\beta-(n+1)(p-1)}&if $-1<\beta<(n+1)(p-1)$,\cr
|\log\delta(z_0)|& if $\beta=(n+1)(p-1)$,\cr
1& if $\beta>(n+1)(p-1)$,\cr}
\]
for all $z_0\in D$.
\end{theorem}

In particular, we have the following estimates on the weighted norms of the Bergman kernel and of the normalized Bergman kernel (see, e.g., \cite{ARS}):

\begin{corollary}
\label{th:3.coruno}
Let $D\subset\subset\C^n$ be a strongly pseudoconvex domain with $C^\infty$ boundary. Take $p>1$ and $-1<\beta<(n+1)(p-1)$. Then 
\[
\|K(\cdot,z_0)\|_{p,\beta}\preceq  \delta(z_0)^{\frac{\beta}{p}-\frac{n+1}{p'}}\quad\mathrm{and}\quad
\|k_{z_0}\|_{p,\beta}\preceq\delta(z_0)^{\frac{n+1}{2}+\frac{\beta}{p}-\frac{n+1}{p'}}
\]
for all $z_0\in D$, where $p'>1$ is the conjugate exponent of~$p$.
\end{corollary}

We shall also need a statement relating the Bergman kernel with Kobayashi balls. 

\begin{lemma}[\cite{Li}, \cite{AS}]
\label{th:3.lemma2}
Let $D\subset\subset\C^n$ be a strongly pseudoconvex domain with $C^\infty$ boundary.
Given $r>0$ there is $\delta_r>0$ such that if $\delta(z_0)<\delta_r$ then
\[
\forall\,z\in B_D(z_0,r)\qquad \min\{|K(z,z_0)|,|k_{z_0}(z)|^2\}\succeq \delta(z_0)^{-(n+1)}\;.
\]
\end{lemma}

\begin{remark}
\label{rem:3.nuova}
Notice that Lemma~\ref{th:3.lemma2} implies the well-known estimate
\[
K(z_0,z_0)\succeq \delta(z_0)^{-(n+1)}\;,
\]
which is valid for all $z_0\in D$.
\end{remark}

The next three lemmas involve instead the Kobayashi distance only.

\begin{lemma}[\cite{AS}]
\label{th:3.lemma3} 
Let $D\subset\subset{\bf C} ^n$ be a strongly pseudoconvex bounded  domain with $C^2$ boundary. Then for every $0<r<R$ there exist $m\in\mathbb{N}$ and a sequence 
$\{z_k\}\subset D$ of points such that
$D=\bigcup_{k=0}^\infty B_D(z_k, r)$
and no point of $D$ belongs to more than $m$ of the balls $B_D(z_k,R)$.
\end{lemma}

\begin{proof}
Let $\{B_j\}_{j\in\mathbb{N}}$ be a sequence of Kobayashi balls of radius $r/3$ covering~$D$. We can extract a subsequence $\{\Delta_k=B_D(z_k, r/3)\}_{k\in \mathbb{N}}$ of disjoint balls in the following way: set $\Delta_1=B_1$. Suppose we have already chosen $\Delta_1,\ldots,\Delta_l$. We define $\Delta_{l+1}$ as the first ball in the sequence $\{B_j\}$ which is disjoint from $\Delta_1\cup\cdots\cup\Delta_l$. In particular, by construction every $B_j$
must intersect at least one~$\Delta_k$. 

We now claim that $\{B_D(z_k, r)\}_{k\in\mathbb{N}}$ is a covering of $D$. Indeed, let $z\in D$. Since 
$\{B_j\}_{j\in\mathbb{N}}$ is a covering of~$D$, there is $j_0\in\mathbb{N}$ so that $z\in B_{j_0}$. As remarked above, we get $k_0\in\mathbb{N}$ so that $B_{j_0}\cap\Delta_{k_0}\neq\emptyset$. 
Take $w\in B_{j_0}\cap\Delta_{k_0}$. Then
\[
k_D(z,z_{k_0})\leq k_D(z,w)+k_D(w,z_{k_0})< r\;,
\]
and $z\in B_D(z_{k_0}, r)$.

To conclude the proof we have to show that there is $m=m_r\in\mathbb{N}$ so that
each point $z\in D$ belongs to at most $m$ of the balls $B_D(z_k,R)$. Put $R_1=R+r/3$. Since $z\in B_D(z_k, R)$ is equivalent to $z_k\in B_D(z, R)$, we have that
$z\in B_D(z_k, R)$ implies $B_D(z_k,r/3)\subseteq B_D(z,R_1)$. Furthermore, Theorem~\ref{th:1.volume} 
and Lemma~\ref{th:3.lemma5} yield
\[
\nu\bigl(B_D(z_k, r/3)\bigr)\succeq \delta(z_k)^{n+1}\succeq \delta(z)^{n+1}
\]
when $z_k\in B_D(z,R)$. Therefore, since the balls
$B_D(z_k, r/3)$ are pairwise disjoint, using again Theorem~\ref{th:1.volume}  we get
\[
\hbox{\rm card}\{k\in\mathbb{N}\mid z\in B_D(z_k, R)\}\le\frac{\nu\bigl(B_D(z,R_1)
\bigr)}{\nu\bigl(B_D(z_k, r/3)\bigr)}\preceq 1\;,
\]
and we are done.
\qed
\end{proof}

\begin{lemma}[\cite{AS}]
\label{th:3.lemma4a}
Let $D\subset\subset{\bf C} ^n$ be a strongly pseudoconvex bounded  domain with $C^2$ boundary, and $r>0$. Then 
\[
\chi(z_0)\preceq \frac{1}{\nu\bigl(B_D(z_0,r)\bigr)}\int_{B_D(z_0,r)}\chi\,\D\nu
\]
for all $z_0\in D$ and all non-negative plurisubharmonic functions $\chi\colon D\to{\bf R} ^+$.
\end{lemma}

\begin{proof}
Let us first prove the statement when $D$ is an Euclidean ball $B$ of radius $R>0$. Without loss of generality we can assume that $B$ is centered at the
origin. Fix $z_0\in B$, let $\gamma_{z_0/R}\in\Aut(B^n)$ be given by (\ref{eq:2.autBn}), and let $\Phi_{z_0}\colon B^n\to B$ be defined by $\Phi_{z_0}=R\gamma_{z_0/R}$; in particular,
$\Phi_{z_0}$ is a biholomorphism with $\Phi_{z_0}(O)=z_0$, and thus
$\Phi_{z_0}\bigl(B_{B^n}(O,\hat r)\bigr)=B_B(z_0,\hat r)$. Furthermore (see \cite[Theorem~2.2.6]{Ru})
\[
|\hbox{Jac}_{\mathbb{R}} \Phi_{z_0}(z)|=R^{2n}\left(\frac{R^2-\|z_0\|^2}{
|R-\langle z,z_0\rangle|^2}\right)^{n+1}\ge \frac{R^{n-1}}{ 4^{n+1}}
\, d(z_0,\partial  B)^{n+1}\;,
\]
where $\hbox{Jac}_{\mathbb{R}} \Phi_{z_0}$ denotes the (real) Jacobian determinant of~$\Phi_{z_0}$.
It follows that
\begin{eqnarray*}
\int_{B_B(z_0, r)}\chi\,\D\nu&=&\int_{B_{B^n}(O, r)}(\chi\circ\Phi_{z_0})
|\hbox{\rm Jac}_{\mathbb{R}} \,\Phi_{z_0}|\,\D\nu\\
&\ge& \frac{R^{n-1}}{ 4^{n+1}}\,
d(z_0,\partial  B)^{n+1}\int_{B_{B^n}(O, r)}
(\chi\circ\Phi_{z_0})\,\D\nu\;.
\end{eqnarray*}
Using \cite[1.4.3 and 1.4.7.(1)]{Ru} we obtain
\[
\int_{B_{B^n}(O,r)}
(\chi\circ\Phi_{z_0})\,\D\nu=2n
\int_{\partial  B^n}\D\sigma(x)\frac{1}{ 2\pi}
\int_0^{\tanh r}\int_0^{2\pi}\chi\circ\Phi_{z_0}(t\E^{\I\theta}x)t^{2n-1}\D t\,\D\theta\;,
\]
where $\sigma$ is the area measure on~$\partial  B^n$ normalized so that $\sigma(\partial  B^n)=1$.
Now, $\zeta\mapsto \chi\circ\Phi_{z_0}(\zeta x)$ is subharmonic 
on~$(\tanh r)\Delta=\{|\zeta|<\tanh r\}\subset\C $ for any $x\in\partial  B^n$, since $\Phi_{z_0}$ is holomorphic and $\chi$ is plurisubharmonic. Therefore \cite[Theorem~1.6.3]{Ho} yields
\[
\frac{1}{ 2\pi}\int_0^{\tanh r}\!\!\int_0^{2\pi}\chi\circ\Phi_{z_0}(t\E^{\I\theta}x)t^{2n-1}\D t\,\D\theta
\ge \chi(z_0)\int_0^{\tanh r} \!t^{2n-1}\,\D t=\frac{1}{ 2n}(\tanh r)^{2n}\chi(z_0)\;.
\]
So 
\[
\int_{B_{B^n}(O, r)}(\chi\circ\Phi_{z_0})\,\D\nu\ge
(\tanh r)^{2n}\chi(z_0)\;,
\]
and the assertion follows from Theorem~\ref{th:1.volume}.

Now let $D$ be a generic strongly pseudoconvex domain. Since $D$ has $C^2$ boundary, there exists a radius $\eps>0$ such that for every $x\in\partial  D$ the euclidean ball $B_x(\eps)$ of radius~$\eps$ internally tangent to $\partial  D$ at $x$ is contained in $D$. 

Let $z_0\in D$. If $\delta(z_0)<\eps$, let $x\in\partial  D$ be such that
$\delta(z_0)=\|z_0-x\|$; in particular, $z_0$ belongs to the
ball $B=B_x(\eps)\subset D$. If $\delta(z_0)\ge\eps$, let $B\subset D$ be the 
Euclidean ball of center~$z_0$ and radius $\delta(z_0)$. In both cases
we have $\delta(z_0)=d(z_0,\partial  B)$; moreover,
the decreasing property of the Kobayashi distance yields $B_D(z_0,r) \supseteq 
B_B(z_0,r)$ for all $r>0$.

Let $\chi$ be a non-negative plurisubharmonic function. Then
Theorem~\ref{th:1.volume} and the assertion for a ball imply
\begin{eqnarray*}
\int_{B_D(z_0,r)}\chi\, d\nu&\ge& \int_{B_B(z_0,r)}\chi\, d\nu \succeq \nu\bigl(B_B(z_0,r)\bigr)\chi(z_0)\\
&\succeq&  d(z_0,\de B)^{n+1}\chi(z_0)=\delta(z_0)^{n+1}\chi(z_0)\\
&\succeq& \nu\bigl(B_D(z_0,r)\bigr)\chi(z_0)\;,
\end{eqnarray*}
and we are done.
\qed
\end{proof}

\begin{lemma}[\cite{AS}] 
\label{th:3.lemma4} 
Let $D\subset\subset{\bf C} ^n$ be a strongly pseudoconvex bounded  domain with $C^2$ boundary. Given $0<r<R$ we have
\[
\forall{z_0\in D\quad\forall z\in B_D(z_0,r)}\qquad \chi(z)\preceq
\frac{1}{\nu\bigl(B_D(z_0, r)\bigr)}\int_{B_D(z_0, R)}\chi\,\D\nu
\]
for every nonnegative plurisubharmonic function $\chi\colon D\to\mathbb{R}^+$.
\end{lemma}

\begin{proof}
Let $r_1= R- r$; by the triangle inequality,
$z\in B_D(z_0, r)$ yields $B_D(z,r_1)\subseteq B_D(z_0, R)$. Lemma~\ref{th:3.lemma4a}
then implies
\begin{eqnarray*}
\chi(z)&\preceq& \frac{1}{\nu(B_D(z,r_1))}\int_{B_D(z,r_1)}\chi\,\D\nu\\
&\le& \frac{1}{\nu(B_D(z,r_1))}\int_{B_D(z_0, R)}\chi\,\D\nu
 = \frac{\nu(B_D(z_0,r))}{\nu(B_D(z,r_1))}
\cdot\frac{1}{\nu(B_D(z_0,r))}\int_{B_D(z_0,R)}\chi\,\D\nu
\end{eqnarray*}
for all $z\in B_D(z_0, r)$.
Now Theorem~\ref{th:1.volume} and Lemma~\ref{th:3.lemma5} yield
\[
\frac{\nu(B_D(z_0,r))}{ \nu(B_D(z,r_1))} \preceq 1
\]
for all $z\in B_D(z_0,r)$, and so
\[
\chi(z)\preceq\frac{1 }{\nu\bigl(B_D(z_0,r)
\bigr)}\int_{B_D(z_0,R)}\chi\,\D\nu
\]
as claimed.
\qed
\end{proof}

Finally, the linking between the Berezin transform and Toeplitz operators is given by the following

\begin{lemma}
\label{th:3.prop}
Let $\mu$ be a finite positive Borel measure on a bounded domain
$D\subset\subset\C^n$. Then
\begin{equation}
B\mu(z)=\int_D (T_\mu k_z)(w) \overline{k_z(w)}\,\D\nu(w)
\label{eq:3.BerTop}
\end{equation}
for all $z\in D$.
\end{lemma}

\begin{proof}
Indeed using Fubini's theorem and the reproducing property of the Bergman kernel we have
\begin{eqnarray*}
B\mu (z)
&=&\int_D \frac{|K(x,z)|^2}{K(z, z)}  \, \D\mu(x)\\
&=&\int_D \frac{K(x,z)}{K(z, z)} K(z,x) \, \D\mu(x)\\
&=&\int_D \frac{K(x,z)}{K(z, z)}\left( \int_D K(w,x) K(z,w)\, \D\nu(w) \right)\, \D\mu(x)\\
&=&\int_D \left( \int_D \frac{K(x,z)}{\sqrt{K(z, z)}} K(w,x) \, \D\mu(x)\right) \frac{\overline{K(w,z)}}{\sqrt{K(z, z)}}\, \D\nu(w)\\
&=&\int_D \left( \int_D  K(w,x) k_{z}(x) \, \D\mu(x)\right) \overline{k_{z}(w)}\, \D\nu(w)\\
&=&\int_D (T_\mu  k_z)(w)\overline{k_z(w)}\,\D\nu(w)\;,
\end{eqnarray*}
as claimed.
\qed
\end{proof}

We can now prove Theorems~\ref{th:3.ToepCarl}, \ref{th:3.Carltheta} and~\ref{th:3.thetaBerez}.

\begin{proof}[of Theorem~\ref{th:3.Carltheta}]
Assume that $\mu$ is a $p$-Carleson measure for $A^p\bigl(D,(n+1)(\theta-1)\bigr)$, and fix $r>0$; we need
to prove that $\mu\bigl(B_D(z_0,r)\bigr)\preceq\nu\bigl(B_D(z_0,r)\bigr)^\theta$ 
for all $z_0\in D$. 

First of all, it suffices to prove the assertion for $z_0$ close to the boundary, because both $\mu$ and $\nu$ are finite measures. So we can assume $\delta(z_0)<\delta_r$, where $\delta_r$ is given by Lemma~\ref{th:3.lemma2}. Since, by Corollary~\ref{th:3.coruno}, $k_{z_0}^2\in A^p\bigl(D,(n+1)(\theta-1)\bigr)$, we have
\begin{eqnarray*}
\frac{1}{\delta(z_0)^{(n+1)p}}\mu\bigl(B_D(z_0,r)\bigr)&\preceq&
\int_{B_D(z_0,r)}|k_{z_0}(w)|^{2p}\,\D\mu(w)\le\int_D|k_{z_0}(w)|^{2p}\,\D\mu(w)\\
&\preceq& \int_D |k_{z_0}(w)|^{2p}\delta(w)^{(n+1)(\theta-1)}\,\D\nu(w)\\
&\preceq& \delta(z_0)^{(n+1)p}\int_D|K(w,z_0)|^{2p}\delta(w)^{(n+1)(\theta-1)}\,\D\nu(w)\\
&\preceq& \delta(z_0)^{(n+1)(\theta-p)}
\end{eqnarray*}
by Theorem~\ref{th:3.thm}, that we can apply because $1-\frac{1}{n+1}<\theta<2$. Recalling Theorem~\ref{th:1.volume} we see that $\mu$ is $\theta$-Carleson.

Conversely, assume that $\mu$ is $\theta$-Carleson for some $r>0$, and let $\{z_k\}$ be the sequence given 
by Lemma~\ref{th:3.lemma3}. Take $f\in A^p\bigl(D,(n+1)(\theta-1)\bigr)$. First of all
\[
\int_D|f(z)|^p\,\D\mu(z)\le\sum_{k\in\mathbb{N}}\int_{B_D(z_k,r)}|f(z)|^p\,\D\mu(z)\;.
\]
Choose $R>r$. Since $|f|^p$ is plurisubharmonic, by Lemma~\ref{th:3.lemma4} we get
\begin{eqnarray*}
\int_{B_D(z_k, r)}|f(z)|^p\,\D\mu(z)&\preceq&\frac{1}{\nu\bigl(B_D(z_k,r)\bigr)}
\int_{B_D(z_k, r)}\left[\int_{B_D(z_k,R)}|f(w)|^p\,\D\nu(w)\right]\D\mu(z)\\
&\preceq& \nu\bigl(B_D(z_k, r)\bigr)^{\theta-1}\int_{B_D(z_k, R)}|f(w)|^p\,\D\nu(w)
\end{eqnarray*}
because $\mu$ is $\theta$-Carleson. Recalling Theorem~\ref{th:1.volume} and Lemma~\ref{th:3.lemma5} we get
\begin{eqnarray*}
\int_{B_D(z_k, r)}|f(z)|^p\,\D\mu(z)&\preceq& \delta(z_k)^{(n+1)(\theta-1)}\int_{B_D(z_k, R)}|f(w)|^p\,\D\nu(w)\\
&\preceq& \int_{B_D(z_k, R)}|f(w)|^p\delta(w)^{(n+1)(\theta-1)}\,\D\nu(w)\;.
\end{eqnarray*}
Since, by Lemma~\ref{th:3.lemma3}, there is $m\in\mathbb{N}$ such that at most $m$ of the 
balls $B_D(z_k,R)$ intersect, we get
\[
\int_D|f(z)|^p\,\D\mu(z)\preceq \int_D |f(w)|^p\delta(w)^{(n+1)(\theta-1)}\,\D\nu(w)\;,
\]
and so we have proved that $\mu$ is $p$-Carleson for $A^p\bigl(D,(n+1)(\theta-1)\bigr)$.
\qed
\end{proof}

We explicitly remark that the proof of the implication ``$\theta$-Carleson implies $p$-Carleson for $A^p\bigl(D,(n+1)(\theta-1)\bigr)$" works for all $\theta>0$, 
and actually gives the following 

\begin{corollary}
\label{th:3.cordue}
Let $D\subset\subset\C^n$ be a bounded strongly pseudoconvex domain with $C^2$ boundary, $\theta>0$,
and $\mu$ a $\theta$-Carleson measure on~$D$. Then
\[
\int_D \chi(z)\,d\mu(z)\preceq  \int_D \chi(w)\delta(w)^{(n+1)(\theta-1)}\,\D\nu(w)
\]
for all nonnegative plurisubharmonic functions $\chi\colon D\to\R^+$ such that $\chi\in L^p\bigl(D,(n+1)(\theta-1)\bigr)$.
\end{corollary}

Now we prove the equivalence between $\theta$-Carleson and the condition on the Berezin transform.

\begin{proof}[of Theorem~\ref{th:3.thetaBerez}]
Let us first assume that $\mu$ is $\theta$-Carleson. By Theorem~\ref{th:3.Carltheta} we know that $\mu$ is 2-Carleson for $A^2\bigl(D,(n+1)(\theta-1)\bigr)$.
Fix $z_0\in D$. Then Corollary~\ref{th:3.coruno} yields
\[
B\mu(z_0)=\int_D|k_{z_0}(w)|^2\,\D\mu(w)\preceq \|k_{z_0}\|^2_{2,(n+1)(\theta-1)}\preceq \delta(z_0)^{(n+1)(\theta-1)}\;,
\]
as required.

Conversely, assume that $\delta^{(n+1)(1-\theta)}B\mu\in L^\infty(D)$, and fix $r>0$. Then Lemma~\ref{th:3.lemma2} yields
\begin{eqnarray*}
\delta(z_0)^{(n+1)(\theta-1)}&\succeq& B\mu(z_0)=\int_D|k_{z_0}(w)|^2\,\D\mu(w)\ge \int_{B_D(z_0, r)}|k_{z_0}(w)|^2\,\D\mu(w)\\
&\succeq&
\frac{1}{\delta(z_0)^{n+1}}\mu\bigl(B_D(z_0, r)\bigr)
\end{eqnarray*}
as soon as $\delta(z_0)<\delta_r$, where $\delta_r>0$ is given by  Lemma~\ref{th:3.lemma2}.
Recalling Theorem~\ref{th:1.volume} we get
\[
\mu\bigl(B_D(z_0,r)\bigr)\preceq\delta(z_0)^{(n+1)\theta}\preceq \nu\bigl(B_D(z_0,r)\bigr)^\theta\;,
\]
and the assertion follows when $\delta(z_0)<\delta_r$. When $\delta(z_0)\ge\delta_r$ we have
\[
\mu\bigl(B_D(z_0,r)\bigr)\le \mu(D)\preceq \delta_r^{(n+1)\theta}\le \delta(z_0)^{(n+1)\theta}\preceq \nu\bigl(B_D(z_0,r)\bigr)^\theta
\]
because $\mu$ is a finite measure, and we are done.
\qed
\end{proof}

For the last proof we need a final

\begin{lemma}
\label{th:3.ultimo}
Let $D\subset\subset\C ^n$ be a bounded stongly pseudoconvex domain with $C^2$ boundary, and $\theta$,~$\eta\in\R$.
Then a finite positive Borel measure $\mu$ is $\theta$-Carleson if and only if
$\delta^\eta\mu$ is $(\theta+\frac{\eta}{n+1})$-Carleson.
\end{lemma}

\begin{proof}
Assume $\mu$ is $\theta$-Carleson, set $\mu_\eta=\delta^\eta\mu$, and choose $r>0$.
Then Theorem~\ref{th:1.volume} and Lemma~\ref{th:3.lemma5} yield
\begin{eqnarray*}
\mu_\eta\bigl(B_D(z_0,r)\bigr)&=&\int_{B_D(z_0,r)}\delta(w)^\eta\,d\mu(w)
\preceq \delta(z_0)^\eta \mu\bigl(B_D(z_0,r)\bigr)\\
&\preceq&  \delta(z_0)^\eta \nu\bigl(B_D(z_0,r)\bigr)^\theta
\preceq \nu\bigl(B_D(z_0,r)\bigr)^{\theta+\frac{\eta}{n+1}}\;,
\end{eqnarray*}
and so $\mu_\eta$ is $\left(\theta+\frac{\eta}{n+1}\right)$-Carleson. Since $\mu=(\mu_\eta)_{-\eta}$, the converse follows too.
\qed
\end{proof}

And at last we have reached the

\begin{proof}[of Theorem~\ref{th:3.ToepCarl}]
Let us assume that $T_\mu$ maps $A^p(D)$ continuously into $A^r(D)$, and let $r'$ be the conjugate exponent of~$r$. Since, by Corollary~\ref{th:3.coruno}, $k_{z_0}\in A^q(D)$ for all $q>1$, applying H\"older estimate to (\ref{eq:3.BerTop}) and using twice Corollary~\ref{th:3.coruno} we get
\begin{eqnarray*}
B\mu(z_0)\le\|T_\mu k_{z_0}\|_r\|k_{z_0}\|_{r'}&\preceq& \|k_{z_0}\|_p\|k_{z_0}\|_{r'}\\
&\preceq&
\delta(z_0)^{(n+1)(1-\frac{1}{p'}-\frac{1}{r})}=\delta(z_0)^{(n+1)(\frac{1}{p}-\frac{1}{r})}\;,
\end{eqnarray*}
where $p'$ is the conjugate exponent of~$p$.
By Theorem~\ref{th:3.thetaBerez} it follows that $\mu$ is $\left(1+\frac{1}{p}-\frac{1}{r}\right)$-Carleson, and Theorem~\ref{th:3.Carltheta} yields that $\mu$ is
$p$-Carleson for $A^p\bigl(D,(n+1)(\frac{1}{p}-\frac{1}{r})\bigr)$ as claimed.

Conversely, assume that $\mu$ is $p$-Carleson for $A^p\bigl(D,(n+1)(\frac{1}{p}-\frac{1}{r})\bigr)$; we must prove that $T_\mu$ maps continuously $A^p(D)$ into $A^r(D)$.
Put $\theta=1+\frac{1}{p}-\frac{1}{r}$. 
Choose $s\in(p,r)$ such that
\begin{equation}
\frac{\theta}{p'}<\frac{1}{s'}<\frac{\theta}{p'}+\frac{1}{(n+1)r}\;,
\label{eq:3.s}
\end{equation}
where $s'$ be its conjugate exponent of $s$; this can be done because $p'\ge s'\ge r'$ and
\[
\frac{\theta}{p'}<\frac{1}{r'}\;.
\]
Take $f\in A^p(D)$; since $|K(z,\cdot)|^{p'/s'}$ is plurisubharmonic and belongs to
$L^p\bigl(D,(n+1)(\theta-1)\bigr)$ by Theorem~\ref{th:3.thm}, 
applying the H\"older inequality, Corollary~\ref{th:3.cordue} 
and Theorem~\ref{th:3.thm} (recalling that $\theta<p'/s'$) we get
\begin{eqnarray*}
|T_\mu f(z)|&\le&\int_D |K(z,w)||f(w)|\,\D\mu(w)\\
&\le&\left[\int_D |K(z,w)|^{p/s}
|f(w)|^p\,\D\mu(w)\right]^{1/p}\left[\int_D|K(z,w)|^{p'/s'
}\,\D\mu(w)\right]^{1/p'}\\
&\preceq&  \left[\int_D |K(z,w)|^{p/s}
|f(w)|^p\,\D\mu(w)\right]^{1/p}\\
&&\qquad\times\left[\int_D|K(z,w)|^{p'/s'
}\delta(w)^{(n+1)(\theta-1)}\,\D\nu(w)\right]^{1/p'}\\
&\preceq&\left[\int_D |K(z,w)|^{p/s}
|f(w)|^p\,\D\mu(w)\right]^{1/p}\delta(z)^{(n+1)\frac{1}{p'}(\theta-\frac{p'}{s'})}\;.
\end{eqnarray*}
Applying the classical Minkowski integral inequality (see, e.g., \cite[6.19]{Fo} for a proof)
\[
\left[\int_D\left[\int_D|F(z,w)|^p\,\D\mu(w)\right]^{r/p}\D\nu(z)\right]^{1/r}\!\!\!\!\le \left[\int_D\left[\int_D|F(z,w)|^r\,\D\nu(z)\right]^{p/r}\D\mu(w)\right]^{1/p}
\]
we get
\begin{eqnarray*}
\|T_\mu f\|^p_r&\preceq&\left[\int_D\left[\int_D|K(z,w)^{p/s}
|f(w)|^p\delta(z)^{(n+1)\frac{p}{p'}(\theta-\frac{p'}{s'})}\,\D\mu(w)\right]^{r/p}\D\nu(z)\right]^{p/r}\\
&\le&\int_D|f(w)|^p\left[\int_D|K(z,w)|^{r/s}
\delta(z)^{(n+1)\frac{r}{p'}(\theta-\frac{p'}{s'})}\,\D\nu(z)\right]^{p/r}\,\D\mu(w)\;.
\end{eqnarray*}
To estimate the integral between square brackets we need to know that
\[
-1<(n+1)\frac{r}{p'}\left(\theta-\frac{p'}{s'}\right)<(n+1)\left(\frac{r}{s}-1\right)\;.
\]
The left-hand inequality is equivalent to the right-hand inequality in (\ref{eq:3.s}), and thus it is satisfied by assumption.
The right-hand inequality is equivalent to
\[
\frac{\theta}{p'}-\frac{1}{s'}<\frac{1}{s}-\frac{1}{r}\ \Longleftrightarrow\ \frac{\theta}{p'}<1-\frac{1}{r}\;.
\]
Recalling the definition of $\theta$ we see that this is equivalent to
\[
\frac{1}{p'}\left(1+\frac{1}{p}-\frac{1}{r}\right)<1-\frac{1}{r}\ \Longleftrightarrow\ \frac{1}{p'}<1-\frac{1}{r}\;,
\]
which is true because $p<r$. So we can apply Theorem~\ref{th:3.thm} and we get
\begin{eqnarray*}
\|T_\mu f\|_r^p&\preceq& \int_D |f(w)|^p\delta(w)^{(n+1)p\left[\frac{1}{p'}(\theta-1)+\frac{1}{r}-\frac{1}{p}\right]}\D\mu(w)\\
&=&\int_D |f(w)|^p\delta(w)^{-(n+1)(\theta-1)}\,\D\mu(w)\\
&\preceq& \|f\|_p^p\;,
\end{eqnarray*}
where in the last step we applied Theorem~\ref{th:3.Carltheta} to $\delta^{-(n+1)(\theta-1)}\mu$, which is 1-Carleson (Lemma~\ref{th:3.ultimo}) and hence $p$-Carleson for $A^p(D)$, and we are done.
\qed
\end{proof}

\input{referencAbateLille}

\begin{acknowledgement}
Partially supported by the FIRB 2012 grant ``Differential Geometry and Geometric Function Theory'', by the Progetto di Ricerca d'Ateneo 2015 ``Sistemi dinamici: logica, analisi complessa e teoria ergodica", and by GNSAGA-INdAM.
\end{acknowledgement}

\end{document}

%% file: referencAbateLille.tex
%
%
%